\documentclass[11pt, reqno]{amsart}
\usepackage[utf8]{inputenc}
\usepackage[a4paper, hmargin={2.7cm,2.7cm},vmargin={3.3cm,2.3cm}]{geometry}
\usepackage[english]{babel}
\usepackage{color}
\usepackage[nospace,noadjust]{cite}
\usepackage{amsmath}
\usepackage{amssymb}
\usepackage{amsfonts}
\usepackage{amsthm}
\usepackage{enumerate}
\usepackage[T1]{fontenc}
\usepackage{dsfont}
\usepackage{etoolbox}
\usepackage{mathrsfs}
\usepackage{enumitem}

\usepackage{todonotes}
\usepackage{hyperref}

\makeatletter
\providecommand*{\cupdot}{%
  \mathbin{%
    \mathpalette\@cupdot{}%
  }%
}
\newcommand*{\@cupdot}[2]{%
  \ooalign{%
    $\m@th#1\cup$\cr
    \sbox0{$#1\cup$}%
    \dimen@=\ht0 %
    \sbox0{$\m@th#1\cdot$}%
    \advance\dimen@ by -\ht0 %
    \dimen@=.5\dimen@
    \hidewidth\raise\dimen@\box0\hidewidth
  }%
}

\providecommand*{\bigcupdot}{%
  \mathop{%
    \vphantom{\bigcup}%
    \mathpalette\@bigcupdot{}%
  }%
}
\newcommand*{\@bigcupdot}[2]{%
  \ooalign{%
    $\m@th#1\bigcup$\cr
    \sbox0{$#1\bigcup$}%
    \dimen@=\ht0 %
    \advance\dimen@ by -\dp0 %
    \sbox0{\scalebox{2}{$\m@th#1\cdot$}}%
    \advance\dimen@ by -\ht0 %
    \dimen@=.5\dimen@
    \hidewidth\raise\dimen@\box0\hidewidth
  }%
}
\makeatother

\numberwithin{equation}{section}

\makeatletter
\patchcmd{\ttlh@hang}{\parindent\z@}{\parindent\z@\leavevmode}{}{}
\patchcmd{\ttlh@hang}{\noindent}{}{}{}
\makeatother

\usepackage[ddmmyyyy,hhmmss]{datetime}
\newdateformat{daymonthyeardate}{%
  \THEDAY.\THEMONTH.\THEYEAR}
\usepackage{fancyhdr}
\usepackage{lastpage}
\usepackage{etoolbox}

\theoremstyle{plain}
\newtheorem{theorem}{Theorem}[section]
\newtheorem{lemma}[theorem]{Lemma}
\newtheorem{proposition}[theorem]{Proposition}

\theoremstyle{definition}
\newtheorem{definition}[theorem]{Definition}
\newtheorem{examplex}[theorem]{Example}

\newtheorem*{assumption*}{Kernel conditions}

\theoremstyle{remark}
\newtheorem{remark}[theorem]{Remark}
\newtheorem*{remark*}{Remark}

\let\emptyset\varnothing

\DeclareMathOperator*{\dom}{dom}

\DeclareMathOperator*{\loc}{loc}

\DeclareMathOperator*{\Co}{Co}
\DeclareMathOperator*{\rel}{Rel}

\DeclareMathOperator*{\supp}{supp}

\DeclareMathOperator*{\esssup}{ess\,sup}

\newcommand{\R}{\mathbb{R}}
\newcommand{\N}{\mathbb{N}}

\newcommand{\T}{\mathbb{T}}
\newcommand{\CC}{\mathbb{C}}
\newcommand{\gramian}{\mathscr{G}}
\newcommand{\synthesis}{\mathscr{D}}
\newcommand{\frameop}{\mathscr{S}}
\newcommand{\analysis}{\mathscr{C}}
\newcommand{\identity}{\mathrm{id}}
\newcommand{\goodMatricesUnweighted}{\mathcal{C}}
\newcommand{\goodMatrices}{\goodMatricesUnweighted_w}
\newcommand{\dominated}{\prec}
\newcommand{\strongWiener}{\WstCw}
\newcommand{\controlFunctionSpace}{\strongWiener}
\newcommand{\controlFunctionNormIndex}{\strongWiener}
\newcommand{\lowdiag}{\alpha}
\newcommand{\updiag}{\beta}
\newcommand{\molenv}{\Phi}
\newcommand{\Kone}{BD}
\newcommand{\Ktwo}{LOC}
\newcommand{\Kthree}{WUC}
\newcommand{\subspace}{\subset}

\newcommand{\Reservoir}{\mathcal{R}}
\newcommand{\TestVectors}{\mathcal{H}_w^1}

\newcommand{\uniformity}{\mathscr{U}}

\newcommand{\CoY}{\Co(Y)}
\newcommand{\Hil}{\mathcal{H}}
\newcommand{\Hpi}{\mathcal{H}_{\pi}}

\newcommand{\RKHS}{\mathcal{K}}
\newcommand{\maxR}{M_Q^R}
\newcommand{\maxL}{M_Q}

\newcommand{\WL}{\mathcal{W}^L}
\newcommand{\WR}{\mathcal{W}^R}

\newcommand{\Wstw}{\mathcal{W}_w}
\newcommand{\WstCw}{\Wstw}

\newcommand{\WLw}{\mathcal{W}^L_w}
\newcommand{\WRw}{\mathcal{W}^R_w}

\newcommand{\envker}{\Theta}

\newcommand{\indicator}{\mathds{1}}
\newcommand{\eps}{\varepsilon}
\newcommand{\normalized}[1]{\widetilde{#1}}

\newcommand{\norm}[1]{\lVert#1\rVert}
\newcommand{\abs}[1]{|#1|}

\title[On dual molecules and convolution-dominated operators]
      {On dual molecules \\ and convolution-dominated operators}
\author{Jos\'{e} Luis Romero}
\author{Jordy Timo van Velthoven}
\author{Felix Voigtlaender}

\address[JLR, JTvV, FV]{Faculty of Mathematics,
University of Vienna,
Oskar-Morgenstern-Platz 1,
A-1090 Vienna, Austria.}

\address[JLR]{Acoustics Research Institute, Austrian Academy of Sciences,
Wohllebengasse 12-14 A-1040, Vienna, Austria}

\address[JTvV]{Department of Mathematics, Ghent University, Krijgslaan 281, Building S8, B-9000 Ghent, Belgium}

\address[FV]{Katholische Universität Eichstätt-Ingolstadt,
Lehrstuhl Wissenschaftliches Rechnen, Ostenstraße 26, 85072 Eichstätt}

\email{jose.luis.romero@univie.ac.at, jlromero@kfs.oeaw.ac.at}
\email{jordy-timo.van-velthoven@univie.ac.at, jordy.vanvelthoven@ugent.be}
\email{felix@voigtlaender.xyz}

\thanks{
J.~L.~R. and J.~v.~V. ~gratefully acknowledge support
from the Austrian Science Fund (FWF): P 29462 and Y 1199,
and from the WWTF grant INSIGHT (MA16-053).
J.~v.~V.~ is grateful for the hospitality and support
of the Katholische Universit\"at Eichst\"att-Ingolstadt during his visit.
The work of F.~V.~was partially supported by the
proFOR+ program of the Katholische Universität Eichstätt--Ingolstadt.
}

\subjclass[2010]{22A10, 42C15, 42C40, 43A15, 46E22}

\begin{document}

\begin{abstract}
 We show that sampling or interpolation formulas in reproducing kernel Hilbert
spaces can be obtained by reproducing kernels whose dual systems form molecules, ensuring
that the size profile of a function is fully reflected by the size profile of its sampled values.
The main tool is a local holomorphic calculus for convolution-dominated operators, valid for
groups with possibly non-polynomial growth. Applied to the matrix coefficients of a group
representation, our methods improve on classical results on atomic decompositions and bridge
a gap between abstract and concrete methods.
\end{abstract}

\maketitle

\section{Introduction}
\label{sec:Introduction}

One of the earliest examples for the discretization of integral expansions concerns
\emph{Calder\'on's reproducing formula}:
for an adequate $\psi \in L^2 (\mathbb{R})$, any $f \in L^2(\R)$ can be written as
\begin{equation}\label{eq:calderon_id}
  f = \int_{\mathbb{R}}
        \int_{ \mathbb{R}^+}
          \langle f, \psi_{b,a} \rangle \, \psi_{b,a}
        \; \frac{da \, db}{a^2} ,
\end{equation}
where
\begin{align}\label{eq_wavelet}
  \psi_{b, a}(x) := a^{-1/2} \cdot \psi \big(a^{-1} (x - b)  \big),
  \qquad x \in \mathbb{R}, a>0, b \in \mathbb{R}.
\end{align}
Calder\'on's formula \eqref{eq:calderon_id} expresses an arbitrary function $f \in L^2(\R)$
as a superposition of scaled and shifted versions \eqref{eq_wavelet} of a basic profile $\psi$,
which is nowadays called a \emph{wavelet}.
Approximation rates for the truncation of Calder\'on's formula
amount to fine smoothness and decay conditions, and lead to
the theory of Besov-Triebel-Lizorkin spaces \cite{frazier1991littlewood}.

Considerable research has been devoted to finding discrete analogues
of \eqref{eq:calderon_id} that involve only a countable, discrete subset of the coefficients
$(\langle f, \psi_{b,a} \rangle)_{b \in \R, a > 0}$ and yet retain the same information.
Orthonormal wavelets provide an example of such discretizations:
If $\psi \in L^2 (\mathbb{R})$ is such a wavelet, then
\begin{align}\label{eq:discrete_representation}
  f = \sum_{\gamma \in \Gamma}
        \langle f, \psi_{\gamma} \rangle \, \psi_{\gamma}
  \qquad \text{for all} \quad f \in L^2(\R),
\end{align}
where
\(
  \Gamma
  = \{
      (2^{j} k, 2^{j})
      \colon
      j,k \in \mathbb{Z}
    \}
\)
is the dyadic index set and $\{\psi_\gamma: \gamma \in \Gamma\}$ is an orthonormal basis.
In the language of function spaces, the scaled and shifted profiles $\psi_{b,a}$
are called \emph{(time-scale) atoms}, and \eqref{eq:discrete_representation} is an
\emph{atomic decomposition} \cite{frazier1990discrete,groechenig1991describing}.
Many situations of interest, including higher dimensional analogues,
non-dyadic index sets, or particular constraints on the function $\psi$, lead to
\emph{non-orthogonal} expansions of the form
\begin{align}\label{eq:discrete_representation_2}
  f = \sum_{\lambda \in \Lambda}
        \langle f, \psi_{\lambda} \rangle \, \widetilde{\psi_{\lambda}}
  \qquad \text{for all } \quad f \in L^2(\R),
\end{align}
where $\Lambda \subset \R \times (0,\infty)$,
and the functions $\widetilde{\psi_{\lambda}}$ are not exactly atoms
given by \eqref{eq_wavelet} but \emph{molecules},
that is, systems of functions that qualitatively behave like atoms
in that their derivatives satisfy growth estimates with respect to the parameters $(b,a)$
as if they were given by \eqref{eq_wavelet};
see \cite{frazier1990discrete, frazier1988phi, frazier1985decomposition, memoir}.
The notion of molecule stems from the theory of Hardy spaces
\cite{coifman1980intro, taibleson1980molecular, coifman1980representation,frazier1990discrete}
and can also be considered in higher dimension and with respect to anisotropic dilations
\cite{MR2186983, MR2179611, MR3452925, bownik2003anisotropic}.
For most applications, molecules are as good as atoms
because they share similar representation and approximation properties.

\smallskip

The discretization of Calder\'on's formula \eqref{eq:calderon_id}
has a natural interpretation as a sampling problem:
Let
\begin{align}\label{eq_aff_group}
  G = \mathbb{R} \rtimes \mathbb{R}^+
    = \big\{ (b, a) \; : \; b \in \mathbb{R}, a \in (0,\infty) \big\}
\end{align}
be the affine group, with multiplication $(b_0, a_0) (b_1, a_1) = (b_0 + a_0 \, b_1, a_0 \, a_1)$
and Haar measure $d\mu_G(b,a) = da \, db / a^2$.
Then $G$ acts on the Hilbert space $\mathcal{H} =L^2(\mathbb{R})$
by virtue of the representation ${\pi(b, a) f = a^{-1/2} \cdot f(a^{-1} (\cdot - b))}$,
and Calder\'on's formula reads
\begin{align}\label{eq:reproducing_id}
  f = \int_G
        \langle f, \pi(x) \psi \rangle \, \pi (x) \psi
      \; d\mu_G (x),
  \qquad \text{for all} \quad f \in \mathcal{H},
\end{align}
which means that the \emph{wavelet transform}
\begin{align}\label{eq_wav_trans}
  W_\psi: \mathcal{H} \to L^2(G),
  \qquad W_\psi f(b,a) = \langle f, \pi(b,a)\psi \rangle
\end{align}
is an isometry into a subspace of $L^2(G)$ \cite{grossmann1986transforms, heil1989continuous}.
In terms of the wavelet transform, the discrete expansion \eqref{eq:discrete_representation_2}
reads
\begin{align}\label{eq:discrete_representation_3a}
  f = \sum_{\lambda \in \Lambda}
        W_\psi f (\lambda) \, \widetilde{\psi_{\lambda}},
\end{align}
or, equivalently,
\begin{align}\label{eq:discrete_representation_3}
  W_\psi f = \sum_{\lambda \in \Lambda}
               W_\psi f (\lambda) \, W_\psi \widetilde{\psi_{\lambda}},
\end{align}
showing that the function $W_\psi f$ can be reconstructed
from its samples taken along  the subset $\Lambda \subset G$.

The property of $\{\psi_{\lambda}: \lambda \in \Lambda\}$ forming a family of molecules
can also be reformulated in terms of the wavelet transform,
namely by requiring that there exists a well-localized envelope $\Phi \in L^1(G)$ such that
\begin{align}\label{eq_intro_mol_1}
  |{W_\psi \widetilde{\psi_{\lambda}}(b,a)}|
  \leq \Phi \big( \lambda^{-1} (b,a) \big),
  \qquad \text{for all} \quad (b,a) \in G \text{ and } \lambda \in \Lambda.
\end{align}
Precise decay conditions on $\Phi$ amount to diverse qualities
of the set of molecules, such as the order of differentiability, rate of decay,
and number of vanishing moments \cite{Holschneider95, memoir, groechenig2009molecules}.
In light of \eqref{eq_intro_mol_1}, the molecule condition
pertains to the locality of the sampling expansion \eqref{eq:discrete_representation_3},
explaining its fundamental role in approximation theory:
not only do the samples ${\{W_\psi f (\lambda): \lambda \in \Lambda\}}$ characterize $W_\psi f$,
but the value $W_\psi f(b,a)$ can be approximated,
up to a well-controlled error, by a finite subset of samples
$\{W_\psi f (\lambda) : \lambda \in \Lambda \cap (b,a) U\}$, with $U \subset G$ compact.

\smallskip

Under suitable admissibility conditions,
a general unitary group representation $\pi:G \curvearrowright \mathcal{H}$
admits a \emph{reproducing formula} as in \eqref{eq:reproducing_id}
\cite{grossmann1985transforms,MR2130226, duflo1976on}.
Provided that the underlying vector $\psi \in \mathcal{H}$ is chosen adequately,
the integral reproducing formula can be approximated by Riemann-like sums,
yielding a discrete expansion of the form \eqref{eq:discrete_representation_3a}.
The theory of coorbit spaces
\cite{feichtinger1989banach1,feichtinger1989banach2,groechenig1991describing}
takes this approach further, showing that such discretizations extend to function spaces
that are defined by imposing adequate decay and integrability conditions
on the abstract wavelet transform \eqref{eq_wav_trans}.
Besov and Triebel-Lizorkin spaces are, for example,
coorbit spaces associated with the affine group \eqref{eq_aff_group}, while the
Schr\"odinger representation of the Heisenberg group leads to $L^p$ versions
of Bargmann-Fock spaces \cite{MR157250}, and the action of $\mathrm{SL}(2, \mathbb{R})$
on the unit disk leads to weighted Bergman spaces \cite{MR810448}.
See \cite[Section 3.3]{groechenig1991describing} for these and other examples.

Coorbit theory revealed that the classical atomic decompositions of Besov-Triebel-Lizorkin,
Bargmann-Fock, and Bergman spaces are all consequences of a single phenomenon:
the action of a suitable group.
The theory allowed to revisit classical atomic decompositions in an abstract and unified way,
and also lead to new examples.
Yet, as noted in the influential monograph \cite[Introduction]{memoir},
the discretization results in \cite{feichtinger1989banach1, groechenig1991describing}
fall slightly short of fully re-deriving the classical ones.
Indeed, while they show that any suitably admissible $\psi \in \mathcal{H}$
and any sufficiently dense set $\Lambda \subset G$ provide an expansion as in
\eqref{eq:discrete_representation_3a} which is also convergent in coorbit spaces,
the techniques in \cite{feichtinger1989banach1, groechenig1991describing} are not sufficient
to deduce the more precise information on the corresponding dual elements
$\{\widetilde{\psi_{\lambda}} :\lambda \in \Lambda\}$ that concrete constructions do deliver.
For instance, when applied to Calder\'on's formula \eqref{eq:calderon_id},
coorbit theory does not produce an expansion \eqref{eq:discrete_representation_2}
consisting of molecules \eqref{eq_intro_mol_1}, as, for example, \cite[Theorem~1.5]{memoir} does.
The absence of an abstract notion of molecule was noted in \cite{groechenig2009molecules},
where coorbit molecules are formally introduced and their basic properties are studied.
The question remained open, however, whether atomic decompositions consisting of molecules
can be obtained in full generality.
In the present article, we answer this question in the affirmative.
As an application of more general results, we sharpen the discretization techniques
for integrable group representations, and bridge a gap between what is achievable
with abstract and concrete methods.

\subsection{Main results}

\subsubsection{Sampling and frames in RKHS}

The discretization techniques from \cite{feichtinger1989banach1, groechenig1991describing}
have been adapted and extended to many different contexts.
Most of these can be modeled by a \emph{reproducing kernel Hilbert space} (RKHS)
$\RKHS \subset L^2 (G)$ of functions on a locally compact $\sigma$-compact group $G$;
see for instance \cite{MR3034426,MR2584749,bernier1996wavelets,MR2855776}.
As commonly done in the literature, we will assume that the \emph{reproducing kernel}
${k : G \times G \to \CC}$, that is, the integral kernel representing the projection
$L^2(G) \to \RKHS$, has \emph{off-diagonal decay}:
\begin{align}\label{eq_od_intro}
  \abs{k(x,y)} \leq \Phi(y^{-1}x),
  \qquad \text{for all} \quad x,y \in G,
\end{align}
where $\Phi$ belongs to a suitable class $\Wstw(G)$ of envelopes determined by $w$,
a submultiplicative weight on $G$.
A \emph{system of molecules} $\{h_\lambda: \lambda \in \Lambda\}$
is a set of functions $h_{\lambda} \in \RKHS$, which is indexed by a subset $\Lambda \subset G$,
and which satisfies the enveloping estimate
\begin{align}
\label{eq_intro_mol}
  \abs{h_\lambda(x)}
  \leq \min \big\{
              \Psi(\lambda^{-1}x), \quad
              \Psi(x^{-1} \lambda)
            \big\},
  \qquad \text{for all} \quad \lambda \in \Lambda \text{ and } x \in G,
\end{align}
with $\Psi \in \Wstw(G)$; see Section~\ref{sub:Envelopes}.
(This definition thus depends implicitly on the weight $w$.)

Functional expansions associated with a (possibly projective) group representation
fit this model by means of the isometric isomorphism
provided by the abstract wavelet transform \eqref{eq_wav_trans},
since the range of the wavelet transform is a RKHS.
The off-diagonal decay of the kernel amounts to suitable \emph{admissibility conditions},
which are stronger than square-integrability; see for instance \cite{MR2855776, groechenig2009molecules}.
Bergman spaces of analytic functions with general weights are also an example of RKHS
where the sampling and interpolation problems are interesting \cite{seip2004interpolation}.
These spaces are not known to be coorbit spaces of a group representation
except for certain specific weights, but they do fit the present framework;
see Example~\ref{ex_fock}.
Further examples are spaces of functions with finite rate of innovation
\cite{sun2006nonuniform, MR2390281}.

We now present our main result concerning sampling in RKHS.
For this result, we assume that the reproducing kernel $k$ satisfies
a certain off-diagonal decay condition (\Ktwo),
and a mild continuity requirement (\Kthree), both described in Section~\ref{sub:RKHS}.
Under these assumptions, we prove that sampling formulas can be implemented by molecules,
as stated in the following theorem.
Here, and below, $k$ will always denote the reproducing kernel of $\RKHS$.

\begin{theorem}\label{thm:frame_main1_intro}
  Let $\RKHS \subspace L^2 (G)$ be a reproducing kernel Hilbert space
  satisfying (\Ktwo) and (\Kthree).
  Then there exists a compact unit neighborhood $U \subset G$
  such that, for any $\Lambda \subset G$
  satisfying $G = \bigcup_{\lambda \in \Lambda} \lambda U$
  and $\sup_{x \in G} \# (\Lambda \cap xU) < \infty$,
  there exists a set of molecules $(h_{\lambda})_{\lambda \in \Lambda}$ in  $\RKHS$ such that
  \begin{align}\label{eq_exp1}
    f = \sum_{\lambda \in \Lambda}
          f(\lambda) \, h_\lambda
      = \sum_{\lambda \in \Lambda}
          \langle f,h_\lambda \rangle \, k(\cdot,\lambda)
    \qquad \text{for all} \quad f \in \RKHS ,
  \end{align}
  with unconditional convergence of the series in $\RKHS$.
\end{theorem}
In technical jargon, part of the statement of Theorem~\ref{thm:frame_main1_intro} is
that the reproducing kernels $\big( k(\cdot,\lambda) \big)_{ \lambda \in \Lambda}$
form a \emph{frame} for $\RKHS$---see Section~\ref{sub:FramesAndRieszSequences}.
The novelty of our result is that it provides a \emph{dual frame}
$(h_{\lambda})_{\lambda \in \Lambda}$ \emph{consisting of molecules}.
Moreover, the envelope for the molecules $(h_{\lambda})_{\lambda \in \Lambda}$ \eqref{eq_intro_mol}
belongs to \emph{the same weight class} $\Wstw(G)$ as
the envelope for the kernel $k$ \eqref{eq_od_intro}.
When applied to the RKHS associated with an integrable group representation,
Theorem~\ref{thm:frame_main1_intro} supersedes the discretization results
in \cite{feichtinger1989banach1, groechenig1991describing}
by providing more precise information on the dual functions $h_\lambda$;
see Section~\ref{sec:CoorbitSpaces}.

\subsubsection{Interpolation and Riesz sequences}
\label{sec_int}

We also consider the dual problem of finding \emph{Riesz sequences of reproducing kernels},
that is, families $\{k(\cdot, \lambda): \lambda \in \Lambda\}$
indexed by a set $\Lambda \subset G$ such that the following norm-equivalence holds:
\begin{align*}
  \Big\|
    \sum_{\lambda \in \Lambda}
      c_\lambda \, k(\cdot, \lambda)
  \Big\|_{\RKHS}
  \asymp \norm{c}_{\ell^2(\Lambda)}
  \qquad \text{for all} \quad c = (c_\lambda)_{\lambda \in \Lambda} \in \ell^2 (\Lambda).
\end{align*}
When $\{k(\cdot, \lambda): \lambda \in \Lambda\}$ is a Riesz sequence,
$\Lambda$ solves the following interpolation problem:
given data $a \in \ell^2(\Lambda)$, there exists a function $f \in \RKHS$ such that
\begin{align}\label{eq_interp}
  f(\lambda) = a_\lambda,  \quad \mbox{ for all }\lambda \in \Lambda.
\end{align}
The novel information in our results concerns the so-called \emph{biorthogonal system}
$(h_\lambda)_{ \lambda \in \Lambda }$, characterized by
\begin{align*}
  (h_\lambda)_{\lambda \in \Lambda}
  \subset \overline{\mathrm{span}} \{ k (\cdot, \lambda) : \lambda \in \Lambda\}
  \qquad \text{and} \qquad
  \langle k(\cdot, \lambda), h_{\lambda'} \rangle = \delta_{\lambda,\lambda'},
  \qquad \text{for all} \quad \lambda, \lambda' \in \Lambda.
\end{align*}
The biorthogonal system implements the coefficient functionals related to the Riesz sequence:
\begin{align*}
  f = \sum_{\lambda \in \Lambda}
        c_\lambda \, k(\cdot, \lambda)
  \mapsto c_{\lambda_0} = \langle f, h_{\lambda_0} \rangle,
\end{align*}
and also provides the interpolant
\begin{align}
\label{eq_interp_2}
f_a = \sum_{\lambda \in \Lambda} a_\lambda h_\lambda, \qquad a \in \ell^2(\Lambda)
\end{align}
satisfying \eqref{eq_interp}.

In order to study Riesz sequences, we assume that the diagonal of the reproducing kernel $k$ is bounded below;
see Condition~$(\Kone)$ in Section~\ref{sub:RKHS}.
Under this assumption, we prove the following.

\begin{theorem} \label{thm:riesz_main2_intro}
  Let $\RKHS \subspace L^2 (G)$ be a reproducing kernel Hilbert space
 satisfying (\Kone) and (\Ktwo).
  Then there exists a compact unit neighborhood $K \subset G$ such that,
  for any $\Lambda \subset G$ satisfying the separation condition
  \begin{align}\label{eq_sep}
    \lambda K \cap \lambda' K = \emptyset,
    \qquad \text{for all } \lambda, \lambda' \in \Lambda \text{ with } \lambda \neq \lambda',
  \end{align}
  the kernels $\big( k(\cdot, \lambda) \big)_{\lambda \in \Lambda}$ form a Riesz sequence
  in $\RKHS$ whose biorthogonal system forms a family of of molecules.
  In addition, under \eqref{eq_sep}, there exists an orthonormal sequence
  $(g_\lambda)_{\lambda \in \Lambda}$ of molecules
  (which is \emph{not} necessarily of the form $g_\lambda = k(\cdot, \lambda)$).
\end{theorem}

With respect to the interpolation problem \eqref{eq_interp}, Theorem \ref{thm:riesz_main2_intro}
implies that the interpolant defined in \eqref{eq_interp_2} reflects the size profile of the data
$a \in \ell^2(\Lambda)$.
When applied to the RKHS associated with an integrable group representation,
Theorem~\ref{thm:riesz_main2_intro} improves on corresponding results from \cite{feichtinger1989banach2}
by providing such qualitative information on the biorthogonal system; see Section~\ref{sec:CoorbitSpaces}.

\subsubsection{The canonical dual frame}

The dual frame $\{h_\lambda: \lambda \in \Lambda\}$ in \eqref{eq_exp1}
is in general only one of many systems yielding such an expansion.
Among all dual systems, the one providing coefficients with minimal $\ell^2$-norm
plays a distinguished role and is called the \emph{canonical dual frame};
see Section~\ref{sub:FramesAndRieszSequences}.
Under special conditions on the sampling set $\Lambda$,
we show that in fact the \emph{canonical} dual frame forms a system of molecules.

More precisely, to each $U$ and $\Lambda$ as in Theorem~\ref{thm:frame_main1_intro}
we associate a certain \emph{measure of uniformity} $\uniformity(\Lambda;U) \in [1,\infty]$;
see Definition \ref{def:Uniformity} below.
Uniform sets $\Lambda$ such as lattices or quasi-lattices have uniformity $\uniformity(\Lambda;U)=1$
for suitable $U$. Given this notion of uniformity, our result regarding the localization
of the canonical dual frame reads as follows:

\begin{theorem}\label{thm:frame_main2_intro}
  Let $\RKHS \subspace L^2 (G)$ be a reproducing kernel Hilbert space
  satisfying (\Ktwo) and (\Kthree).
  Then there exists a compact unit neighborhood $U \subset G$ and an $\eps > 0$ such that:
  for any $\Lambda \subset G$ satisfying
  $G = \bigcup_{\lambda \in \Lambda} \lambda U$, $\sup_{x \in G} \# (\Lambda \cap xU) < \infty$
  and $\uniformity(\Lambda; U) \leq 1 + \eps$, the kernels
  $\big( k(\cdot,\lambda) \big)_{\lambda \in \Lambda}$ form a frame for $\RKHS$
  whose \emph{canonical} dual frame consists of molecules.
\end{theorem}

The conditions of Theorem~\ref{thm:frame_main2_intro} are concretely satisfied
whenever $G$ admits lattices---or quasi-lattices---of arbitrarily high density
(or, equivalently, of arbitrarily small volume).
Such special sets, however, do not always exists.
Yet, we prove that the hypothesis of Theorem~\ref{thm:frame_main2_intro}
can always be met, and, as an application, derive the following existence result.

\begin{theorem}\label{thm:exist1_intro}
  Let $\RKHS \subspace L^2 (G)$ be a reproducing kernel Hilbert space satisfying (\Ktwo) and (\Kthree).
  There exists a frame of reproducing kernels $\big( k(\cdot,\lambda) \big)_{ \lambda \in \Lambda}$
  whose \emph{canonical} dual frame consists of molecules.
  In addition, there exists a \emph{Parseval frame} of molecules
  $(f_\lambda)_{ \lambda \in \Lambda}$ for $\RKHS$;
  that is, a frame for $\RKHS$ that coincides with its canonical dual frame.
\end{theorem}

For a Parseval frame $(f_\lambda)_{ \lambda \in \Lambda}$,
we have $f = \sum_{\lambda \in \Lambda} \langle f, f_\lambda \rangle \, f_\lambda$
for all $f \in \RKHS$,
which resembles an orthogonal expansion.
Such systems are very useful for representing operators.

\subsection{Technical overview and related work}
\label{sub:TechnicalOverview}

\subsubsection{Convolution-dominated operators}

An operator on $L^2(G)$ will be called \emph{con\-vo\-lution-do\-mi\-nated}
(CD) if it is represented by an integral kernel $H$ satisfying the enveloping condition
\begin{align}
\label{eq_intro_H}
  \abs{H(x,y)} \leq \Phi(y^{-1} x),
  \qquad \text{for all} \quad x,y \in G,
\end{align}
for a well-localized envelope function $\Phi \in L^1(G)$;
see \cite{batayneh2016localized, fendler2010convolution, sun2008wiener} for related notions.
Although we do not use any specific results from the literature,
we formulate many of the technical lemmas in terms of such operators.
Crucially, we also consider convolution-dominated \emph{matrices} whose entries are indexed by
subsets of $G$ that do not need to be subgroups. Such matrices provide an abstract analog of the
\emph{almost diagonal matrices} considered in \cite{frazier1990discrete},
while CD operators can be seen as an analog of the Cotlar-type operators of \cite{memoir}.

\subsubsection{Local spectral invariance}

Certain algebras of CD operators are \emph{spectrally invariant}:
this means that if a convolution-dominated operator in the class is invertible,
then its inverse belongs to the same class of operators
\cite{MR1609144, fendler2008convolution, fendler2007on}%
\footnote{For non-discrete groups, the class of CD operators needs to be augmented
with the multiples of the identity.}.
A notable example of this phenomenon pertains to groups of polynomial growth,
that is, groups $G$ with a compact unit neighborhood $U$
whose powers have Haar measure $\mu_G$ dominated by a polynomial: $\mu_G(U^n) \lesssim n^k$
\cite{MR323951, tessera2010left, MR2020712,sun2007wiener}.
Convolution-domination is measured in terms of weighted $L^p$-norms, and,
under suitable assumptions on the weights, is preserved under inversion.
The theory of \emph{localized frames} \cite{groechenig2004localization,balan2006density}
exploits such results to conclude that the canonical dual frame of a frame of molecules
is itself a frame of molecules.
Thus, if spectral invariance tools are applicable,
the subtleties in our results are mostly trivial.

Spectral invariance, however, might fail for groups not possessing
polynomial growth; see \cite[Section 5]{fendler2016on} and \cite{MR2827175}
for examples on the affine and free groups respectively.
In the absence of spectral invariance, the theory of localized frames does not apply.
For example, there are smooth and fast-decaying wavelets with several vanishing moments
which yield a frame for $L^2(\mathbb{R})$, but do \emph{not} admit a dual frame
formed by molecules of a similar quality, and, indeed,
do not lead to $L^p$-expansions \cite{tchamitchian1987biorthogonalite, tchamitchian1986calcul};
see also \cite[Section~9.2]{meyer1997wavelets} and \cite{tao}.
For certain particular wavelet construction schemes, or for specific wavelets such as the Mexican hat,
establishing the validity of $L^p$-expansions is significantly more challenging
than that of $L^2$ expansions \cite{MR2770530, MR2754776, MR3034769, MR2989978}.

Our main technical tool is a substitute for the spectral invariance
of algebras of convolution-dominated (CD) operators, which we call \emph{local spectral invariance}.
For a CD operator $T$ acting on a RKHS $\RKHS$, we show that there is a trade-off
between \emph{the envelope} $\Phi$ in \eqref{eq_intro_H}, which we assume to belong to a certain
weighted envelope class $\Wstw(G)$, and the \emph{spectral tightness} of the operator---that is,
its distance to the identity as an operator on $\RKHS$.
When these two objects are adequately balanced, the holomorphic calculus maps $T$
into a CD operator with envelope in the \emph{same class} $\Wstw(G)$---see
Theorems~\ref{thm:integral_eps_inverse} and \ref{thm:invertibility_convdom_matrix}.
Similar methods can be traced back to \cite[Section~9]{frazier1990discrete}
and \cite[Chapter~3]{memoir}, and are closely related to classical off-diagonal estimates
for matrices and exponential weights \cite{MR0493419, jaffard1990proprietes},
and their application to wavelets and parabolic variants \cite{MR1015808, jaffard1990proprietes,MR3062475}.
A trade-off principle for Calder\'on-Zygmund operators in the spirit of local spectral invariance
appears also in \cite{lisu12, MR3034769}.

\subsubsection{Almost tight frames}
\label{sec_alti}
One step in the proof of our results is to construct an \emph{almost tight} sampling set
for a RKHS $\RKHS$, that is, a subset $\Lambda = \Lambda(\eps) \subset G$ such that
\begin{align}\label{eq_intro_eps}
  (1-\varepsilon) \cdot \norm{f}_{L^2}^2
  \leq \sum_{\lambda \in \Lambda} \abs{f(\lambda)}^2
  \leq (1 + \varepsilon) \cdot \|f\|_{L^2}^2
  \qquad \text{for all} \quad f \in \RKHS,
\end{align}
with $\eps > 0$ arbitrary.
To obtain such a set $\Lambda$, we discretize the reproducing identity
of $\RKHS$ by Riemann-like sums as done by many others, including
\cite{feichtinger1989banach1, groechenig1991describing, fuehr2007sampling,
fornasier2005continuous, ali1993continuous, MR2584749, MR3034426, olsen1992note}.
Closest to our approach is \cite{olsen1992note}.

Deviating from \cite{feichtinger1989banach1, groechenig1991describing},
we do not extend the discretization technique to Banach norms
because that would not yield the quantitative consequences for dual systems that we are after.
Furthermore, in contrast to \cite{sun2006nonuniform,MR2584749,MR3034426,MR2442078,MR2390281}
we cannot rely on spectral invariance to obtain such consequences automatically.
The crucial step in our discretization is that we derive \eqref{eq_intro_eps}
while carefully controlling the qualities that are relevant to exploit local spectral invariance.

\subsubsection{Smoothness of the reproducing kernel}

The smoothness we assume on the reproducing kernel $k$ is considerably weaker
than in other works such as \cite{MR3034426,MR2584749,fornasier2005continuous},
since it only involves the \emph{(squared) absolute values} of the elements of $\RKHS$%
---see Section~\ref{sub:RKHS}, kernel condition (WUC).
Bargmann-Fock spaces of analytic functions on the plane are a case in point where this is relevant,
since classical Bergman bounds provide estimates only for the absolute values
of the functions in those RKHS; see Example~\ref{ex_fock}.
Stronger estimates, as required by many other works on RKHS, seem to be unavailable.
Similarly, since we only impose such weak smoothness assumptions on reproducing kernels,
our results are applicable to projective representations, where the presence of a cocycle
may render stronger smoothness assumptions inadequate---see Remark~\ref{rem_pro}.

\subsubsection{Frequency covers}

Certain representations of affine-type groups $G$ with general dilations
induce covers of the Fourier domain and the associated coorbit spaces can be identified
with so-called Besov-type spaces \cite{tr78, MR3345605, fuehr2019coorbit}.
In those cases, we can compare the present results to our recent work on
lattice expansions and frequency covers \cite{rovova19}.
In contrast to \cite{rovova19}, the present results apply
to general translation nodes without algebraic structure, have consequences
not only for Besov-type spaces but also for Triebel-Lizorkin-type spaces \cite{frazier1991littlewood},
and provide quantitative information on dual frames.
On the other hand, the results in \cite{rovova19} pertain to general frequency covers,
not necessarily related to a group of dilations, and describe the success
of the so-called Walnut-Daubechies criterion, about which this article makes no claim.

\subsubsection{Discretizing without off-diagonal decay}

The off-diagonal decay of the reproducing kernel \eqref{eq_od_intro}
is an essential assumption for our results, which establish the validity of comparable decay properties
for discrete expansions (and, indeed, we show in Remark~\ref{rem:KernelDecayNecessary}
that \eqref{eq_od_intro} is necessary for our results to hold).
In the absence of off-diagonal decay, even the existence of discrete expansions,
without claims on the quality of the dual systems, is non-trivial,
and was recently established in \cite{MR3415581,freeman2019discretization}; see also \cite{MR3935697}.
In the same spirit, the existence of discrete frames in the orbit of a possibly
\emph{non}-square-integrable representation (in particular, \emph{non}-integrable representation)
of a solvable Lie group has recently been proved in \cite{oussa2018frames, oussa2019compactly},
by means of a specific construction.

\subsection{Structure of the article}
\label{sub:PaperStructure}

Section~\ref{sec:Preliminaries} provides background on frames and Riesz sequences,
and weights on groups; furthermore, it introduces the class of envelopes that we use
to define convolution-dominated operators.
Section~\ref{sec:MoleculesInRKHS} introduces reproducing kernel Hilbert spaces,
systems of molecules, and the precise assumptions on the reproducing kernel required for our results.
Section~\ref{sec:ConvolutionDominatedOperators} contains the results on local spectral invariance.
These are subsequently put to use in Sections~\ref{sec:FramesWithDualMolecules}
and \ref{sec:RieszSequencesWithDualMolecules}, where the existence of
dual frames or biorthogonal systems consisting of molecules is proved.
Specifically, Theorems~\ref{thm:frame_main1_intro}, \ref{thm:frame_main2_intro},
and \ref{thm:exist1_intro} are proved in Section~\ref{sec_proofs}, while
Theorem~\ref{thm:riesz_main2_intro}, restated as Theorem~\ref{thm:riesz_main},
is proved in Section \ref{sub:RieszSequenceDualMolecules}.
In Section~\ref{sec_examples_a} we provide examples and explain how our results improve on classical coorbit theory.
Several technical results are postponed to the appendix.

\section{Preliminaries}
\label{sec:Preliminaries}

\subsection{Notation}
\label{sub:Notation}

Let $G$ be a $\sigma$-compact locally compact group with a left Haar measure $\mu_G$.
Denote by $\Delta : G \to \R^+$ the modular function on $G$,
where $\R^+ := (0,\infty)$.
For $z \in \CC$ and $r > 0$, we write $B_r (z) := \{ w \in \CC \colon |w - z| < r \}$.
We also use the notation $\T := \{ z \in \CC \colon |z| = 1 \}$.

Throughout the paper, the set $Q \subset G$ will be a fixed symmetric, open, relatively compact
neighborhood of the identity $e \in G$.

The left and right translate by $y \in G$ of a function $f : G \to \CC$ are defined by
$L_y f = f(y^{-1} \cdot)$, respectively $R_y f = f( \cdot \, y)$.
The involution of $f$ is defined by $f^{\vee} (x) = f(x^{-1})$ for $x \in G$.
The characteristic function of a set $X$ is denoted by $\mathds{1}_{X}$.

Given an index set $I$ and associated points $\lambda_i \in G$,
we refer to a collection $\Lambda = (\lambda_i)_{i \in I}$ as a \emph{family} in $G$.
Here, in contrast to a set, we allow for repetitions or multiplicities.
For notational simplicity, we will most of the time pretend that $\Lambda$ is a set;
for instance, with the notation $(g_\lambda)_{\lambda \in \Lambda}$, we actually mean
$(g_i)_{i \in I}$, i.e., $g_\lambda$ can depend on the index $i \in I$ satisfying $\lambda = \lambda_i$.
Furthermore, we use notations like $\sum_{\lambda \in \Lambda} c_\lambda$,
$\bigcup_{\lambda \in \Lambda} M_\lambda$, and $\# (\Lambda \cap M)$,
which have to be interpreted as $\sum_{i \in I} c_{\lambda_i}$,
$\bigcup_{i \in I} M_{\lambda_i}$, and $\# \{ i \in I \colon \lambda_i \in M \}$,
respectively.
For a family of sets $(M_i)_{i \in I}$, we use the notation $M = \bigcupdot_{i \in I} M_i$
to express that $M = \bigcup_{i \in I} M_i$ and furthermore $M_i \cap M_j = \emptyset$ for $i \neq j$.

\subsection{Discrete sets and covers}
\label{sub:DiscreteSets}

Let $\Lambda$ be family in $G$ and let $U \subset G$ be a unit neighborhood.
The family $\Lambda$ is called \emph{relatively separated} in $G$ if
\begin{equation}
  \rel (\Lambda)
  := \sup_{x \in G}
       \# \big( \Lambda \cap x Q\big)
   = \sup_{x \in G}
       \sum_{\lambda \in \Lambda}
         \mathds{1}_{x Q} (\lambda)
   = \sup_{x \in G}
       \sum_{\lambda \in \Lambda}
         \mathds{1}_{\lambda Q} (x)
   < \infty.
   \label{eq:RelativeSeparation}
\end{equation}
The family $\Lambda$ is called \emph{$U$-dense}
if $G = \bigcup_{\lambda \in \Lambda} \lambda U$.
For a relatively separated and $U$-dense $\Lambda$ in $G$,
a family $(U_{\lambda} )_{\lambda \in \Lambda}$ of Borel sets $U_{\lambda} \subset G$
satisfying $U_{\lambda} \subset \lambda U$ and $G = \bigcupdot_{\lambda \in \Lambda} U_{\lambda}$
is called a \emph{disjoint cover associated to $\Lambda$ and $U$}.
Covers of this type always exist:

\begin{lemma}\label{lem:relativelyUdense_disjointunion}
  Let $U \subset G$ be a relatively compact unit neighborhood.
  Suppose that $\Lambda$ is relatively separated and $U$-dense in $G$.
  Then there exists a countable disjoint cover $(U_{\lambda})_{\lambda \in \Lambda}$
  associated to $\Lambda$ and $U$.
\end{lemma}

\begin{proof}
  Since $G$ is $\sigma$-compact and $\Lambda$ is relatively separated,
  there exists an enumeration $(\lambda_n)_{n \in \N}$ of $\Lambda$
  (where we allow $\lambda_n = \lambda_m$ even if $n \neq m$).
  For $n \in \N$, define $U_{\lambda_n} := \lambda_n U \setminus \bigcup_{m=1}^{n-1} \lambda_{m} U$.
  Then $(U_{\lambda_n} )_{n \in \N}$ is a countable family satisfying the desired properties.
\end{proof}

A family $\Lambda$ in $G$ is called \emph{$U$-separated} in $G$
if $\lambda U \cap \lambda' U = \emptyset$ for all $\lambda, \lambda' \in \Lambda$
with $\lambda \neq \lambda'$.
The family $\Lambda$ is called \emph{separated} if it is $U$-separated
for some unit neighborhood $U \subset G$.
Any separated set is relatively separated.
For a separated and $U$-dense family, an associated disjoint cover can be chosen satisfying
convenient additional properties; see \cite[Section 3]{fuehr2007sampling}.

\begin{lemma}[{\cite{fuehr2007sampling}}]\label{lem:partition}
  Let $U, V \subset G$ be precompact unit neighborhoods satisfying $V V^{-1} \subset U$.
  Then there exists a $V$-separated and $U$-dense set $\Lambda$ in $G$.
  For any $V$-separated and $U$-dense family $\Lambda$, there exists a family $(U_{\lambda} )_{\lambda \in \Lambda}$
  such that $G = \bigcupdot_{\lambda \in \Lambda} U_{\lambda}$ as a disjoint union
  and moreover the sets $U_{\lambda} \subset G$ are relatively compact Borel sets
  satisfying $\lambda V \subset U_{\lambda} \subset \lambda U$ for all $\lambda \in \Lambda$.
\end{lemma}

Lastly, we formally define the notion of the \emph{uniformity} of a family $\Lambda$.

\begin{definition}\label{def:Uniformity}
  Let $U \subset G$ be a relatively compact unit neighborhood.
  Let $\Lambda$ be relatively separated and $U$-dense in $G$.
  The \emph{$U$-uniformity} $\uniformity(\Lambda;U)$ of $\Lambda$ is defined as
  \[
    \uniformity(\Lambda;U)
    := \inf \bigg\{
             \sup_{\lambda, \lambda' \in \Lambda}
               \frac{\mu_G (U_\lambda)}{\mu_G(U_{\lambda'})}
             \,\colon \;
             (U_{\lambda})_{\lambda \in \Lambda} \; \text{disjoint cover associated to $\Lambda$ and $U$}
           \bigg\}
    \in [1,\infty] .
  \]
\end{definition}

Any lattice subgroup $\Lambda \subset G$ with Borel fundamental domain $U \subset G$
satisfies $\uniformity (\Lambda; U) = 1$.
More generally, any \emph{quasi-lattice} $\Lambda$, that is, any $U$-dense set
satisfying $\mu_G(\lambda U \cap \lambda' U) = 0$ for $\lambda,\lambda' \in \Lambda$
with $\lambda \neq \lambda'$, also satisfies $\uniformity (\Lambda; U) = 1$.
In contrast to lattices, quasi-lattices can exist in a non-unimodular group;
see for instance \cite[Proposition~5.10]{fuehr2007sampling}.

\subsection{Envelope classes}
\label{sub:Envelopes}

The \emph{left-} and \emph{right maximal functions}
of an $f \in L^{\infty}_{\loc} (G)$ are defined by
\(
  \maxL f(x) = \esssup_{y \in x Q} |f(y)|
\)
and
\(
  \maxR f (x) = \esssup_{y \in Qx} |f(y)|,
\)
respectively.
The functions $\maxL f$ and $\maxR f$ are Borel measurable (resp.~continuous)
for $f \in L^{\infty}_{\loc} (G)$ (resp.~$f \in C(G)$),
and $|f(x)| \leq \maxL f(x)$ and $|f(x)| \leq \maxR f(x)$ for $\mu_G$-a.e $x \in G$.

A function $w : G \to (0,\infty)$ will be called an \emph{admissible weight},
if $w$ is measurable and submultiplicative,
meaning that ${w(xy) \leq w(x) \, w(y)}$ for all $x,y \in G$,
and if furthermore $w \geq 1$.
Given such a weight, we define $\| f \|_{L_w^1} := \| w \cdot f \|_{L^1} \in [0,\infty]$
for any measurable function $f : G \to \CC$, and we write $f \in L_w^1 (G)$
if and only if $\| f \|_{L_w^1} < \infty$.
The \emph{(two-sided) amalgam space} $\Wstw(G)$ is defined as
\[
  \Wstw(G) := \big\{ f \in C (G) \; : \; \maxR \maxL f \in L^1_w (G) \big\}
\]
and equipped with the norm
$\| f \|_{\Wstw} := \| \maxR \maxL f \|_{L^1_w} = \| \maxL \maxR f \|_{L^1_w}$.
For technical reasons, we also use
amalgam spaces that are only invariant under left
or right translations and consist of merely measurable functions, namely
\[
  \WLw(G) := \big\{ f \in L^{\infty}_{\loc} (G) : \maxL f \in L^1_w (G) \big\}
  \quad \text{and} \quad
  \WRw(G) := \big\{ f \in L^{\infty}_{\loc} (G) : \maxR f \in L^1_w (G) \big\} .
\]
Equipped with the norms $\| f \|_{\WLw} := \| \maxL f \|_{L^1_w}$
and $\| f \|_{\WRw} := \| \maxR f \|_{L^1_w}$, these spaces are Banach spaces satisfying
$\WLw(G) \hookrightarrow L^1_w (G)$ and $\WRw(G) \hookrightarrow L^1_w (G)$.
Moreover, the space $\WLw$ embeds into $L^{\infty} (G)$
(see Lemma~\ref{lem:LInftyEmbedding}), and $\Wstw(G) \subset \WLw(G) \cap \WRw(G)$.
If $w \equiv 1$, we simply write $\WL$ and $\WR$.

We now collect several estimates.
First, for $f_1 \in \WRw(G)$ and $f_2 \in \WLw(G)$, it holds that
\[
  \maxL (f_1 \ast f_2)(x) \leq (|f_1| \ast \maxL f_2)(x)
  \quad \text{and} \quad
  \maxR (f_1 \ast f_2 )(x) \leq (\maxR f_1 \ast |f_2|)(x)
\]
for $\mu_G$-a.e.~$x \in G$.
Combining this with the fact that
$\| f_1 \ast f_2 \|_{L_w^1} \leq \| f_1 \|_{L_w^1} \cdot \| f_2 \|_{L_w^1}$
and $L^1 (G) \ast L^{\infty} (G) \subset C(G)$,
we obtain the convolution relation
\begin{align}\label{eq:amalgam_convolution}
  \WRw(G) \ast \WLw(G) \hookrightarrow \WstCw(G)
  \quad \text{with} \quad
  \| f_1 \ast f_2 \|_{\WstCw} \leq \|f_1 \|_{\WRw} \, \|f_2 \|_{\WLw}.
\end{align}

The following estimates will be used repeatedly.
The proof is deferred to the Appendix~\ref{sub:AmalgamAppendix}.

\begin{lemma}\label{lem:SynthesisOperatorQuantitativeBound}
  If $\Phi, \Psi : G \to [0,\infty)$ are continuous
  and $\Lambda$ is relatively separated in $G$, then
  \begin{align}\label{eq:StandardEstimateOne}
    &\sum_{\lambda \in \Lambda}
      \Phi(\lambda^{-1} x) \, \Psi(y^{-1} \lambda)
    \leq \frac{\rel(\Lambda)}{\mu_G (Q)}
         \cdot \big(
                 \maxL \Psi \ast \maxR \Phi
               \big)
               (y^{-1} x)
    \qquad \text{for all} \quad x,y \in G,
    \\
    &\sum_{\lambda \in \Lambda}
      \Psi(y^{-1} \lambda)
      \leq \frac{\rel(\Lambda)}{\mu_G (Q)}
         \, \| \Psi \|_{\WL}
    \quad \text{for all } y \in G.
    \label{eq:StandardEstimateTwo}
  \end{align}
  Finally, if $\Theta \in L^1 (G)$ is continuous and satisfies $\Theta^{\vee} \in \WL(G)$,
  then the operator
  \[
    D_{\Theta,\Lambda} :
    \ell^2(\Lambda) \to L^2(G), \quad
    (c_\lambda)_{\lambda \in \Lambda} \mapsto \sum_{\lambda \in \Lambda}
                                                c_\lambda \, L_\lambda \Theta
  \]
  is well-defined and bounded,
  with
  \(
    \| D_{\Theta,\Lambda} \|^2_{\ell^2 \to L^2}
    \leq \frac{\rel(\Lambda)}{\mu_G (Q)} \, \| \Theta \|_{L^1} \, \| \Theta^\vee \|_{\WL}
    .
  \)
  The defining series is absolutely convergent $\mu_G$-a.e.~on $G$.
\end{lemma}

For more on amalgam spaces, including the two-sided version, see
\cite{holland1975harmonic, fournier1985amalgams, feichtinger1983banach, FelixPhD, romero2012characterization}.

\subsection{Frames and Riesz sequences}
\label{sub:FramesAndRieszSequences}

Let $\Hil$ be a separable Hilbert space.

A countable family $(g_{\lambda} )_{\lambda \in \Lambda}$ of vectors $g_{\lambda} \in \Hil$
is called a \emph{frame} for $\Hil$ if there exist constants $A, B > 0$,
called \emph{frame bounds}, such that
\begin{align} \label{eq:frame_ineq}
  A \, \| f \|_{\Hil}^2
  \leq \sum_{\lambda \in \Lambda}
         |\langle f, g_{\lambda} \rangle|^2
  \leq B \, \| f \|_{\Hil}^2
  \qquad \text{for all} \quad f \in \Hil.
\end{align}
A family $(g_{\lambda} )_{\lambda \in \Lambda}$ satisfying the upper frame bound
in \eqref{eq:frame_ineq} is called a \emph{Bessel sequence} in $\mathcal{H}$.
If $(g_{\lambda})_{\lambda \in \Lambda}$ is a Bessel sequence with bound $B>0$,
then the associated \emph{coefficient operator}
\[
  \analysis : \Hil \to \ell^2 (\Lambda), \quad
              f \mapsto \big( \langle f, g_{\lambda} \rangle \big)_{\lambda \in \Lambda}
\]
is well-defined and bounded, with $\| \analysis \|_{\Hil \to \ell^2 (\Lambda)} \leq B^{1/2}$.
Equivalently, the associated \emph{reconstruction operator}
\[
  \synthesis
  := \analysis^* : \ell^2 (\Lambda) \to \Hil, \quad
                   (c_{\lambda})_{\lambda \in \Lambda} \mapsto \sum_{\lambda \in \Lambda}
                                                                 c_{\lambda} \, g_{\lambda}
\]
is well-defined and bounded.
The \emph{Gramian} and the \emph{frame operator} of $(g_{\lambda} )_{\lambda \in \Lambda}$ are
the operators $\gramian := \analysis \circ \synthesis : \ell^2 (\Lambda) \to \ell^2 (\Lambda)$
and ${\frameop := \synthesis \circ \analysis : \Hil \to \Hil}$, respectively.
A system $(g_{\lambda})_{\lambda \in \Lambda}$ forms a frame for $\Hil$ if and only if
the frame operator $\frameop : \Hil \to \Hil$ is bounded and invertible.

Two Bessel sequences $(g_{\lambda} )_{\lambda \in \Lambda}$ and $(h_{\lambda} )_{\lambda \in \Lambda}$
are called \emph{dual frames} for $\Hil$ if
\[
  f = \sum_{\lambda \in \Lambda}
        \langle f, g_{\lambda} \rangle \, h_{\lambda}
  = \sum_{\lambda \in \Lambda}
      \langle f, h_{\lambda} \rangle \, g_{\lambda}
  \qquad \text{for all} \quad f \in \Hil .
\]
In this case, both $(g_{\lambda} )_{\lambda \in \Lambda}$ and $(h_{\lambda} )_{\lambda \in \Lambda}$
form frames for $\Hil$.
If $(g_{\lambda})_{\lambda \in \Lambda}$ is a frame for $\Hil$ with frame operator $\frameop$,
then $(h_{\lambda})_{\lambda \in \Lambda} = (\frameop^{-1} g_{\lambda} )_{\lambda \in \Lambda}$
forms a dual frame of $(g_{\lambda} )_{\lambda \in \Lambda}$,
called the \emph{canonical dual frame}.

A countable family $(g_{\lambda} )_{\lambda \in \Lambda}$ of vectors $g_{\lambda} \in \Hil$
is called a \emph{Riesz sequence} in $\Hil$ if there exist constants $A, B > 0$,
called \emph{Riesz bounds}, such that
\[
  A \, \| c \|_{\ell^2 (\Lambda)}^2
  \leq \bigg\| \sum_{\lambda \in \Lambda} c_{\lambda} \, g_{\lambda} \bigg\|^2_{\Hil}
  \leq B \, \| c \|_{\ell^2 (\Lambda)}^2
  \qquad \text{for all} \quad c = (c_{\lambda})_{\lambda \in \Lambda} \in \ell^2 (\Lambda) .
\]
For more on frames and Riesz sequences,
the reader is referred to \cite{young2001introduction, christensen2016introduction}.

\section{Molecules in reproducing kernel Hilbert spaces}
\label{sec:MoleculesInRKHS}

\subsection{Reproducing kernel Hilbert spaces}
\label{sub:RKHS}

Let $\RKHS \subspace L^2 (G)$ be a \emph{reproducing kernel Hilbert space (RKHS)};
that is, $\RKHS \subset L^2(G)$ is closed and separable, and for each $x \in G$,
the point evaluation $\RKHS \ni f \mapsto f(x) \in \CC$ is a well-defined bounded linear functional.
By the Riesz representation theorem, this implies that for each $x \in G$
there exists $k_x \in \RKHS$ such that ${f(x) = \langle f, k_x \rangle}$ for all $f \in \RKHS$.
The kernel ${k : G \times G \to \CC, \; (x,y) \mapsto \langle k_y, k_x \rangle}$
is called the \emph{reproducing kernel} of $\RKHS$.
The orthogonal projection $P_{\RKHS} : L^2 (G) \to L^2 (G)$ onto $\RKHS$ is given by
\begin{align} \label{eq:RKHS_repformula}
  P_{\RKHS} f(x)
  = \int_G k(x,y) f(y) \; d\mu_G (y)
  = \langle f, k_x \rangle_{L^2},
\end{align}
where $k_x (y) = k(y,x) = \overline{k(x,y)}$.
Note by separability of $\RKHS$ that $k : G \times G \to \CC$ is measurable, since
\(
  k(x,y)
  = \langle k_y, k_x \rangle
  = \sum_{n=1}^\infty \langle k_y, \phi_n \rangle \langle \phi_n, k_x \rangle
  = \sum_{n=1}^\infty \phi_n(x) \overline{\phi_n(y)} ,
\)
where $(\phi_n)_{n \in \N}$ is an orthonormal basis for $\RKHS$.

Throughout the paper, $k$ will always denote the reproducing kernel of $\RKHS$.
We next list three conditions on the reproducing kernel.
These conditions are \emph{not} assumed to hold throughout the paper;
instead it will be explicitly mentioned in each theorem which conditions are assumed.

\begin{assumption*}
Let $\RKHS \subspace L^2(G)$ be a RKHS with reproducing kernel $k : G \times G \to \CC$.
Let $w: G \to (0,\infty)$ be a fixed admissible weight.
We will consider the following kernel conditions:
\begin{enumerate}
  \item[(\Kone)] \emph{Bounded diagonal:}
                 There exist constants $\lowdiag, \updiag > 0$ such that
                 \begin{align} \label{eq:diagonal_kernel}
                   0 < \lowdiag \leq k(x,x) \leq \updiag < \infty
                   \qquad \text{for all} \quad x \in G .
                 \end{align}

  \item[(\Ktwo)] \emph{Localization:} There exists a non-negative \emph{envelope}
                 $\Theta \in \WstCw (G)$ such that
                 \begin{align}
                   |k(x,y)| \leq \envker (y^{-1} x)
                   \qquad \text{for all} \quad
                   x,y \in G .
                   \label{eq:localization_kernel}
                 \end{align}

  \item[(\Kthree)] \emph{Weak uniform continuity:}
                   There exists a non-negative $\Theta' \in \WstCw(G)$ such that,
                   for all $f \in \RKHS$ and $x,y \in G$:
                   \begin{align} \label{eq:WUC}
                   \big|\,
                     |f(x)|^2 - |f(y)|^2
                   \big|
                   \leq \eta(y^{-1} x) \cdot
                        \int_G
                          |f(z)|^2 \,
                          \Theta'(z^{-1} y)
                        \; d\mu_G (z)
                   \end{align}
                   where $\eta : G \to [0,\infty)$ is a function satisfying
                   $\eta(x) \to 0$ as $x \to e$.
\end{enumerate}
\end{assumption*}

Note that (WUC) concerns the \emph{absolute values} of functions in $\mathcal{K}$.
While Lemma~\ref{lem:SUCImpliesWUC} below gives a simple sufficient condition for (WUC)
involving the oscillation of reproducing kernels, the weaker assumption (WUC) is instrumental
to treat certain spaces of analytic functions; see Example~\ref{ex_fock} below.

The notion of a molecule in a reproducing kernel Hilbert space is defined as follows.

\begin{definition}\label{def:molecule}
  Let $\Lambda$ be relatively separated in $G$,
  let $\RKHS \subspace L^2(G)$ be a RKHS, and let $w : G \to (0,\infty)$ be an admissible weight.
  A family $(g_{\lambda} )_{\lambda \in \Lambda}$ of vectors $g_{\lambda} \in \RKHS$
  is said to be \emph{a system of $w$-molecules}
  if there exists a non-negative $\molenv \in \WstCw(G)$ such that
  \[
    |g_{\lambda} (x)|
    \leq \min \big\{ \molenv (\lambda^{-1} x), \molenv(x^{-1} \lambda) \big\}
    \qquad \text{for all} \quad
    x \in G \text{ and } \lambda \in \Lambda.
  \]
\end{definition}

\subsection{Discussion on kernel conditions}
\label{sub:DiscussionOfStandingHypothesis}

In this subsection we will discuss the conditions (\Kone), (\Ktwo), (\Kthree) and their relations.
The following observation will often be useful.

\begin{lemma}\label{lem:LocalizationImpliesUpDiag}
  If condition (\Ktwo) holds,
  then the upper bound in \eqref{eq:diagonal_kernel} holds,
  with ${\beta := \| \Theta \|_{L^2}^2}$.
\end{lemma}

\begin{proof}
 For all $x \in G$, we have
 \(
    k(x,x)
    = \langle k_x, k_x \rangle
    = \| k_x \|_{L^2}^2
    \leq \| \Theta(x^{-1} \cdot) \|_{L^2}^2
    =    \| \Theta \|_{L^2}^2
    < \infty.
 \)
\end{proof}

The next lemma provides a condition that might be easier to verify
than condition (\Kthree) in particular examples; see for instance Section~\ref{sec:CoorbitSpaces}.

\begin{lemma}\label{lem:SUCImpliesWUC}
  Let $\RKHS \subspace L^2(G)$ be a RKHS satisfying condition (\Ktwo).
  Suppose that there is a (not necessarily measurable) function $\Gamma : G \times G \to \T$
  such that
  \begin{align} \label{eq:SUC}
    \big\| k_x - \Gamma(x,y) \, k_y \big\|_{L^1} \to 0
    \quad \text{as} \quad y^{-1} x \to e .
  \end{align}
  Then $k : G \times G \to \CC$ satisfies kernel condition (\Kthree).
  In addition, the condition \eqref{eq:SUC} holds if and only if
  $\|k_x - \Gamma(x,y) \, k_y \|_{L^2} \to 0$ as $y^{-1} x \to e$.
\end{lemma}

\begin{proof}
Let $\Theta \in \WstCw(G)$ be as in the localization estimate \eqref{eq:localization_kernel}.
For $f \in \RKHS$ and $x,y \in G$, we have
\[
  \big| |f(x)|^2 - |f(y)|^2 \big|
  =    \big| |f(x)| - |f(y)| \big| \cdot \big( |f(x)| + |f(y)| \big).
\]
Since we are only interested in the behavior as $y^{-1} x \to e$,
we may assume that $y^{-1} x \in Q$.
Then ${z^{-1} x = z^{-1} y y^{-1} x \in z^{-1} y Q}$, and hence
\begin{align}\label{eq:estimate1}
  |k_x (z)| + |k_y (z)|
  \leq \Theta (z^{-1} x) + \Theta(z^{-1} y)
  \leq 2 \, \maxL \Theta (z^{-1} y)
  \qquad \text{for all} \quad z \in G .
\end{align}
A direct calculation using this estimate, combined with the Cauchy-Schwarz inequality, gives
\begin{align*}
  \big| |f(x)| - |f(y)| \big|
  & = \big| |f(x)| - |\overline{\Gamma(x,y)} \, f(y)| \big|\\
  &\leq \bigg|
       \int_G
         f(z)
         \big( \,
           \overline{k_x (z)} - \overline{\Gamma(x,y) \, k_y (z)}
         \, \big)
       \; d\mu_G (z)
     \bigg| \\
  &\leq \bigg(
          \int_G
            |f(z)|^2
            |k_x (z) - \Gamma(x,y) \, k_y(z)|
          \; d\mu_G (z)
        \bigg)^{1/2} \\
  &   \quad \quad \cdot
      \bigg(
        \int_G
          |k_x(z) - \Gamma(x,y) \, k_y(z)|
        \; d\mu_G (z)
      \bigg)^{1/2} \\
  &\leq \bigg(
          \int_G
            |f(z)|^2 \,
            \cdot 2 \, \maxL \Theta (z^{-1} y)
          \; d\mu_G (z)
        \bigg)^{1/2}
        \| k_x - \Gamma(x,y) \, k_y\|_{L^1}^{1/2} .
\end{align*}
Similarly as above, since $y^{-1} x \in Q = Q^{-1}$, it follows that
${x^{-1} z = (y^{-1} x)^{-1} y^{-1} z \in Q y^{-1} z}$, and hence
$\Theta(x^{-1} z) + \Theta(y^{-1} z) \leq 2 \, \maxR \Theta (y^{-1} z)$.
Combining $|k_x(y)| = |k(y,x)| = |k(x,y)|$ with \eqref{eq:localization_kernel}
and \eqref{eq:estimate1}, this yields
\[
  |k_x (z)| + |k_y (z)|
  \leq 2 \cdot \big[ \maxL \Theta (z^{-1} y) \big]^{1/2}
         \cdot \big[ \maxR \Theta (y^{-1} z) \big]^{1/2} .
\]
Consequently,
\begin{align*}
  |f(x)| + |f(y)|
  & \leq \int_G
           |f(z)|
           \cdot \big( |k_x(z)| + |k_y(z)| \big)
          \; d\mu_G (z) \\
  & \leq 2
         \int_G
           |f(z)| \cdot
           \big[ \maxL \Theta (z^{-1} y) \big]^{1/2}
           \cdot \big[ \maxR \Theta (y^{-1} z) \big]^{1/2}
         \; d\mu_G (z) \\
  & \leq 2 \,
         \bigg(
           \int_G
             |f(z)|^2 \, \maxL \Theta (z^{-1} y)
           \; d \mu_G (z)
         \bigg)^{1/2}
         \bigg(
           \int_G \maxR \Theta (y^{-1} z) \; d \mu_G (z)
         \bigg)^{1/2} \\
  & = 2 \, \| \maxR \Theta \|_{L^1}^{1/2} \cdot
      \bigg(
        \int_G
          |f(z)|^2 \, \maxL \Theta (z^{-1} y)
        \; d \mu_G (z)
      \bigg)^{1/2} .
\end{align*}
Combining the obtained inequalities gives
\begin{align*}
  \big|
    \, |f(x)|^2 - |f(y)|^2
  \big|
  \lesssim \| k_x - \Gamma(x,y) \, k_y\|^{1/2}_{L^1}
           \int_G
             |f(z)|^2 \, \maxL \Theta (z^{-1} y)
           \; d\mu_G (z),
\end{align*}
where the implied constant does not depend on $f,x,y$, but only on $\Theta$.
Because $\Theta \in \WstCw(G)$ implies that $M_Q \Theta \in \WstCw (G)$,
this proves the first part of the lemma.

For the ``in addition'' claim, suppose the kernel $k : G \times G \to \CC$
satisfies property \eqref{eq:SUC}.
By Lemma~\ref{lem:LocalizationImpliesUpDiag}, there exists a constant $\updiag > 0$ such that
$|k_x (z)| = |\langle k_x, k_z \rangle| \leq \| k_x \|_{L^2} \, \| k_z \|_{L^2} \leq \updiag$,
for all $x, z \in G$.
Therefore,
\[
  \| k_x - \Gamma(x,y) \, k_y \|_{L^2}^2
  = \!\! \int_G |k_x (z) - \Gamma(x,y) \, k_y (z)|^2 \, d\mu_G (z)
  \leq 2 \updiag \int_G |k_x (z) - \Gamma(x,y) \, k_y(z)| \, d \mu_G (z),
\]
which easily shows that $\|k_{x} - \Gamma(x,y) \, k_y \|_{L^2} \to 0$ as $y^{-1} x \to e$,
by assumption \eqref{eq:SUC}.

Conversely, let $\eps > 0$ be arbitrary.
Choose a compact $K \subset G$ with
\({
  \int_{G \setminus K}
    \Theta(z)
  \, d\mu_G (z)
  \!\leq\! \frac{\eps}{4} .
}\)
By assumption, there is a compact symmetric unit neighborhood $U \subset Q$
such that
\[
  \| k_x - \Gamma(x,y) \, k_y \|_{L^2}
  \leq \frac{\eps / 2}{1 + \sqrt{\mu_G(Q K)}}
  \leq \frac{\eps / 2}{1 + \sqrt{\mu_G(U K)}}
\]
for all $x,y \in G$ with $y^{-1} x \in U$.
Thus, if $y^{-1} x \in U$, then
it follows by the localization estimate \eqref{eq:localization_kernel}
and the Cauchy-Schwarz inequality that
\begin{align*}
  \| k_x - \Gamma(x,y) \, k_y \|_{L^1}
  & \leq \int_{G \setminus x U K}
           |k(z,x)| + |k(z,y)|
         \, d \mu_G (z) \\
  & \quad \quad \quad + \int_G
              \indicator_{x U K} (z) \cdot
              |k_x (z) - \Gamma(x,y) \, k_y (z)|
            \, d \mu_G (z) \\
  & \leq \int_{G \setminus x U K}
           \envker(x^{-1} z) + \envker(y^{-1} z)
         \, d \mu_G (z)
         + \| \indicator_{x U K} \|_{L^2}
           \cdot \| k_x - \Gamma(x,y) k_y \|_{L^2} \\
  & \leq 2 \int_{G \setminus K} \envker(w) \, d\mu_G (w)
         + \sqrt{\mu_G(x U K)} \cdot \frac{\eps/2}{1 + \sqrt{\mu_G(U K)}}
    \leq \eps ,
\end{align*}
where it was used that $G \setminus U K \subset G \setminus K$
and $G \setminus y^{-1} x UK \subset G \setminus K$ as $x^{-1} y \in U$.
\end{proof}

The following example provides a setting in which condition (\Kthree)
is satisfied, but the uniformity condition \eqref{eq:SUC} might fail.

\begin{examplex}[Weighted Fock spaces of entire functions]\label{ex_fock}
Let $\phi: \CC^n \to \mathbb{R}$ be twice continuously differentiable,
and assume that there exist constants $m, M > 0$ such that
\begin{align*}
  m \, I_n
  \leq \big( \partial_j \overline{\partial_k} \phi(z) \big)_{j,k \in \{1,...,n\}}
  \leq M \, I_n, \qquad z \in \CC^n,
\end{align*}
in the sense of positive definite matrices (in particular, $\phi$ is a so-called plurisubharmonic function).
The \emph{weighted Fock space} of entire functions is
\[
  \mathcal{F}^2_{\phi} (\CC^n)
  := \big\{
       f:\CC^n \to \CC \mbox{ holomorphic}
       \; : \;
       f \cdot e^{- \phi} \in L^2 (\CC^n, dm)
     \big\},
\]
where $dm$ denotes the Lebesgue measure on $\CC^n \cong \R^{2n}$.
To fit our context, we renormalize the space as
\(
  \RKHS_{\phi}
  := \big\{
       g = f e^{-\phi}
       \; : \;
       f \in \mathcal{F}^2_{\phi} (\CC^n)
     \big\}
  \subset L^2(\CC^n,dm).
\)
Kernel conditions (\Kone) and (\Kthree) amount to pointwise estimates
for the so-called \emph{Bergman kernel} and can be found in \mbox{\cite[Proposition~9]{lindholm2001sampling}},
\cite[Section~3]{schuster2012toeplitz}, and \cite{delin1998pointwise}.
Smoothness estimates such as the ones in Lemma~\ref{lem:SUCImpliesWUC} may fail;
nevertheless, the kernel condition (\Kthree) is satisfied.
This follows from classical smoothness estimates for the (squared) \emph{absolute values}
of functions $g \in \RKHS_{\phi}$ (weighted Bergman bounds):
\[
  \big|
    \nabla (|f|^2 \, e^{-2\phi} )
  \big| (z)
  \leq C_0 \int_{B_1 (z)}
             |f(w)|^2 \, e^{-2\phi(w)}
           \; dm(w)
  \qquad \text{for} \quad f \in \mathcal{F}^2_{\phi} (\CC^n),
\]
see for instance \cite[Proposition 2.3]{schuster2012toeplitz},
or \cite[Lemma 17]{lindholm2001sampling}.
\end{examplex}

\section{Convolution-dominated operators and systems of molecules}
\label{sec:ConvolutionDominatedOperators}

\subsection{Convolution-dominated integral operators}
\label{sub:ConvolutionDominatedIntegralOperators}

In this section we introduce the class of integral kernels
that will be considered in the remainder.
See \cite{shin2009stability,sun2008wiener,batayneh2016localized,fendler2016on,fendler2010convolution}
for related notions.

\begin{definition}\label{def:AdmissibleKernels}
  Let $\RKHS \subspace L^2 (G)$ be a RKHS and let $w : G \to (0,\infty)$ be an admissible weight.
  A measurable function $H : G \times G \to \CC$ is called \emph{$w$-localized} in $\RKHS$
  if
  \begin{enumerate}[label=(\roman*)]
    \item $H(\cdot, y) \in \RKHS$ for all $y \in G$;
          \vspace{0.1cm}

    \item $\overline{H(x,\cdot)} \in \RKHS$ for all $x \in G$;
          \vspace{0.1cm}

    \item There exists a non-negative \emph{envelope} $\molenv \in \WstCw(G)$ such that
          \begin{align} \label{eq:envelope_w-localized}
            \max \big\{ |H(x,y)|, |H(y,x)| \big\} \leq \molenv (y^{-1} x)
            \qquad \text{for all} \quad x,y \in G .
          \end{align}
  \end{enumerate}
  Given a $w$-localized kernel $H : G \times G \to \CC$, the \emph{adjoint kernel}
  $\widetilde{H} : G \times G \to \CC$ is defined by $\widetilde{H} (x,y) := \overline{H(y,x)}$.
\end{definition}

The following lemma provides some basic properties of localized kernels that will be used in the sequel.
For completeness we provide a proof in Appendix~\ref{sub:LocalizedIntegralKernelsAppendix}.

\begin{lemma}\label{lem:AdmissibleKernelProperties}
  Let $\RKHS \subspace L^2 (G)$ be a RKHS.
  Let $H : G \times G \to \CC$ be a $w$-localized kernel in $\RKHS$
  for some admissible weight $w : G \to (0,\infty)$.
  Then the associated integral operator
  \[
    T_H : L^p(G) \to L^p (G), \quad
    T_H f (x) = \int_G H(x,y) \, f(y) \, d \mu_G (y),
  \]
  is well-defined and bounded for arbitrary $p \in [1,\infty]$,
  with absolute convergence of the defining integral for all $x \in G$.
  Moreover, the following properties hold:
  \begin{enumerate}[label=(\roman*)]
    \item $T_H$ maps $L^2(G)$ into $\RKHS$.

    \item For all $x, y \in G$, it holds that
          \[
            H(x,y) = \langle T_H k_y, k_x \rangle .
          \]
          In particular, this implies
          \begin{equation}
            |H(x,y)| \leq \| T_H \|_{\RKHS \to L^2} \cdot \| k_y \|_{L^2} \cdot \| k_x \|_{L^2},
            \qquad \text{for all} \quad x,y \in G.
            \label{eq:IntegralKernelPointwiseBound}
          \end{equation}

    \item The adjoint kernel $\widetilde{H}$ of $H$ is $w$-localized in $\RKHS$.

    \item If $L : G \times G \to \CC$ is $w$-localized in $\RKHS$,
          then so is $H \odot L : G \times G \to \CC$, where
          \[
            \quad
            (H \odot L)(x,y)
            := \int_G H(x,z) \, L(z,y) \, d\mu_G (z)
             = T_H [L(\cdot,y)] (x)
            = \overline{T_{\widetilde{L}} [\overline{H(x,\cdot)}] (y)}
          \]
          for all $x, y \in G$.
          Furthermore, $T_H \circ T_L = T_{H \odot L}$, and
          $\widetilde{H \odot L} = \widetilde{L} \odot \widetilde{H}$.
  \end{enumerate}
\end{lemma}

\subsection{Local spectral invariance of integral operators}

The following theorem provides a substitute for the spectral invariance of
CD operators that is still sufficient for our purposes.
(See e.g. \cite[Sections~10.21--10.29]{RudinFA} for background on the holomorphic functional calculus.)

\begin{theorem}\label{thm:integral_eps_inverse}
  Let $\RKHS \subspace L^2 (G)$ be a RKHS satisfying (\Ktwo).
  Given an admissible weight $w : G \to (0,\infty)$,
  let $H : G \times G \to \CC$ be a $w$-localized  kernel in $\RKHS$.
  Denote by $\Theta \in \Wstw(G)$ and $\Phi \in \Wstw(G)$ envelopes for the kernels $K = k$ and $H$,
  respectively.

  For arbitrary $\delta > 0$, there is an $\eps = \eps(\Theta, \Phi, \delta, w) \in (0, \delta)$
  (independent of $H$) such that: If
  \begin{enumerate}
    \item The  operator $T_H : \RKHS \to \RKHS$
          satisfies $\| T_H - \identity_{\RKHS} \|_{\RKHS \to \RKHS} \leq \eps$;
          \vspace{0.1cm}

    \item The function $\phi : B_\delta (1) \to \CC$ is holomorphic;
  \end{enumerate}
  then the operator  $\phi(T_H) : \RKHS \to \RKHS$
  defined through the holomorphic functional calculus
  satisfies $\phi(T_H) = T_{H_\phi}|_{\RKHS}$ for a $w$-localized kernel
  ${H_\phi : G \times G \to \CC}$ in $\RKHS$.
\end{theorem}

\begin{proof}
  Let $\delta > 0$ and let $\beta := \| \Theta \|_{L^2}^2$ be the upper bound in \eqref{eq:diagonal_kernel}
  provided by Lemma~\ref{lem:LocalizationImpliesUpDiag}.
  Note that since $H : G \times G \to \CC$ is $w$-localized in $\RKHS$,
  it follows by Lemma~\ref{lem:AdmissibleKernelProperties}
  that $T_H : \RKHS \to \RKHS$ is well-defined and bounded.
  The proof will be split into several steps:

  \smallskip{}

  \textbf{Step 1.} \emph{(Choosing $\eps$)}.
  Define $\envker' , \Phi' : G \to [0,\infty)$ by
  \(
    \envker' (x) := \min \{ \envker(x), \envker(x^{-1}) \}
  \)
  and
  \(
    \Phi' (x) := \min \{ \Phi(x), \Phi(x^{-1}) \} .
  \)
  Since $|k(x,y)| = |k(y,x)|$, the conditions~\eqref{eq:localization_kernel}
  and \eqref{eq:envelope_w-localized} are still satisfied if we replace $\Phi$ and $\envker$
  by $\Phi'$ and $\envker'$, respectively.
  Furthermore, $\Phi', \envker' \in \WstCw (G)$ simply because $0 \leq \Phi' \leq \Phi$
  and $0 \leq \envker' \leq \envker$.

  Let $\Phi_{\eps} := \min \{ \eps \updiag, \envker' + \Phi' \}$ for $\eps > 0$.
  Since
  \({
    \maxR \maxL \Phi_\eps
    \leq \min \{ \eps \updiag, \maxR \maxL \Theta + \maxR \maxL \Phi \},
  }\)
  it follows from the dominated convergence theorem that
  $\| \Phi_{\eps} \|_{\WstCw} \to 0$ as $\eps \downarrow 0$.
  Choose $\eps = \eps(\Phi, \Theta, \beta, \delta, w) \in (0, \frac{\delta}{2})$ such that
  $\| \Phi_{\varepsilon} \|_{\WstCw} \leq \frac{\delta}{4}$.

  \medskip{}

  \textbf{Step 2.} \emph{(Representing $\phi(T_H)$ as a series)}.
  Let $\phi : B_\delta (1) \to \CC$ be holomorphic.
  By assumption, $\| \identity_{\RKHS} - T_H \|_{\RKHS \to \RKHS} \leq \eps < \frac{\delta}{2}$,
  and hence $\sigma(T_H) \subset B_{\delta/2} (1)$.
  This implies that $\phi(T_H) : \RKHS \to \RKHS$ is a well-defined bounded linear operator.

  By expanding $\phi$ into a power series,
  we can write $\phi(z) = \sum_{n=0}^\infty a_n \, (z - 1)^n$ for all ${z \in B_\delta (1)}$,
  for a suitable sequence $(a_n)_{n \in \N_0} \subset \CC$.
  The series representing $\phi$ convergences locally uniformly on $B_\delta (1)$.
  Therefore, elementary properties of the holomorphic functional calculus
  (see \mbox{\cite[Theorem~10.27]{RudinFA}}) show that
  \begin{equation}
    \phi(T_H) = \sum_{n=0}^\infty a_n \, (T_H - \identity_{\RKHS})^n,
    \label{eq:SpectralCalculusSeriesRepresentation}
  \end{equation}
  with convergence in the operator norm. An application of the Cauchy-Hadamard formula
  gives $\delta \leq \big[\, \limsup_{n \to \infty} |a_n|^{1/n} \,\big]^{-1}$.
  Thus, there is some $N = N(\phi,\delta) \in \N$ such that $|a_n|^{1/n} \leq \frac{2}{\delta}$
  for all $n \geq N$.
  Consequently, there is $C_\phi = C_\phi (\delta) > 0$ such that
  \begin{equation}
    |a_n| \leq C_\phi \cdot (2/\delta)^n
    \qquad \text{for all } n \in \N_0
    .
    \label{eq:SpectralCalculusCoefficientBound}
  \end{equation}

  \medskip{}

  \textbf{Step 3.} \emph{(Integral representation of $T_H - \identity_{\RKHS}$)}.
  By the reproducing formula \eqref{eq:RKHS_repformula}, we have
  \[
    (T_K f) (x)
    = \int_G f(y) K(x,y) \; d\mu_G (y)
    = \langle f, k_x \rangle
    = f(x)
    \qquad \text{for} \quad f \in \RKHS \text{ and } x \in G .
  \]
  Hence, $T_K = \identity_{\RKHS}$ and therefore $T_H - \identity_{\RKHS} = T_{H - K}$.

  \medskip{}

  \textbf{Step 4.} \emph{($w$-localization of $K-H$)}.
  Since $\| T_{H - K} \|_{\RKHS \to \RKHS} = \| T_H - \identity_{\RKHS} \| \leq \eps$
  and $\|k_x \|_{L^2} \leq \updiag^{1/2}$
  for all $x \in G$ by \eqref{eq:diagonal_kernel},
  it follows by the point-wise estimate~\eqref{eq:IntegralKernelPointwiseBound} that
  \begin{equation} \label{eq:K-H_estimate1}
    |(H - K)(x,y) |
    \leq \| T_{H - K} \|_{\RKHS \to \RKHS} \|k_y \|_{L^2} \| k_x \|_{L^2}
    \leq \varepsilon \updiag
    \qquad \text{for all } x,y \in G.
  \end{equation}
  On the other hand, as in Step~1,
  \begin{align}\label{eq:K-H_estimate2}
    |(H - K) (x,y)|
    \leq |K(x,y)| + |H(x,y)|
    \leq (\envker' + \Phi')(y^{-1} x)
    \qquad \text{for all } x,y \in G.
  \end{align}
  Combining \eqref{eq:K-H_estimate1} and \eqref{eq:K-H_estimate2}
  yields $|(H - K) (x,y)| \leq \Phi_{\eps} (y^{-1} x)$ for all $x,y \in G$.

  \medskip{}

  \textbf{Step 5.} \emph{(Estimating products of $\Phi_\eps$ and $H - K$)}.
  Define inductively
  \[
    \Phi_{\varepsilon}^{\ast(1)} := \Phi_{\varepsilon}
    \qquad \text{and} \qquad
    \Phi_{\varepsilon}^{\ast(n+1)} := \Phi_{\varepsilon} \ast \Phi_{\varepsilon}^{\ast(n)} ,
  \]
  as well as
  \[
    (H - K)^{\circ(1)} := H - K
    \qquad \text{and} \qquad
    (H - K)^{\circ(n+1)} := (H - K)^{\circ(n)} \odot (H - K),
  \]
  where the operation $\odot$ is as defined in Lemma~\ref{lem:AdmissibleKernelProperties}.
  This lemma also shows for arbitrary $n \in \N$ that
  $(T_H - \identity_{\RKHS})^n = T_{(H - K)^{\circ(n)}}$,
  and that $(H - K)^{\circ(n)}$ is $w$-localized in $\RKHS$,
  which in particular means that
  \begin{equation}
    (H - K)^{\circ(n)} (\cdot, y) \in \RKHS
    \qquad \text{and} \qquad
    \overline{(H - K)^{\circ(n)} (x, \cdot)} \in \RKHS
    \qquad \text{for} \quad x,y \in G .
    \label{eq:IteratedKernelsAreInRKHS}
  \end{equation}
  Also, by associativity of the convolution on $L^1(G)$, it follows by an induction argument that
  $\Phi_\eps^{\ast(n+1)} = \Phi_\eps^{\ast(n)} \ast \Phi_\eps$ for all $n \in \N$.

  Furthermore, by induction on $n \in \N$ one can show that
  \begin{align}\label{eq:induction_convolution}
    \| \Phi_{\varepsilon}^{\ast (n)} \|_{\WstCw}
    \leq \| \Phi_{\varepsilon} \|_{\WstCw} \| \Phi_{\varepsilon} \|^{n-1}_{\WRw}
    \qquad \text{for all } n \in \N .
  \end{align}
  Indeed, the case $n = 1$ is clear.
  For the induction step, we use the convolution relation \eqref{eq:amalgam_convolution}
  to deduce that
  \(
    \| \Phi_{\varepsilon}^{\ast(n+1)} \|_{\WstCw}
    \leq \| \Phi_{\varepsilon} \|_{\WRw} \| \Phi_{\varepsilon}^{\ast(n)} \|_{\WstCw}
    \leq \| \Phi_{\varepsilon} \|_{\WRw}^{(n+1)-1} \| \Phi_{\varepsilon} \|_{\WstCw} .
  \)

  Finally, again by induction on $n \in \N$, it follows that
  \begin{equation}\label{eq:induction_product}
    \max \big\{
           |(H - K)^{\circ (n)} (x, y) |, \,\,
           |(H - K)^{\circ (n)} (y, x) |
         \big\}
    \leq \Phi_{\varepsilon}^{\ast (n)} (y^{-1} x),
    \quad \text{for all } x,y \in G.
  \end{equation}
  Here, the base case $n=1$ follows from Step~4 since $\Phi_\eps (y^{-1} x) = \Phi_\eps (x^{-1} y)$.
  For the induction step, let $x, y \in G$.
  Then a change of variables, combined with the induction hypothesis, shows
  \begin{align*}
    \big| (H - K)^{\circ(n+1)} (x,y) \big|
    & \leq \int_G
             \big| (H - K)^{\circ(n)} (x,z) \big| \cdot \big| (H - K)(z, y) \big|
           \; d\mu_G (z) \\
    &\leq \int_G
            \Phi^{\ast (n)}_{\varepsilon} (z^{-1} x) \,
            \Phi_{\varepsilon} (y^{-1} z)
          \; d\mu_G (z)
    = \int_G
        \Phi_{\varepsilon} (t) \,
        \Phi^{\ast (n)}_{\varepsilon} (t^{-1} y^{-1} x)
      \; d\mu_G (t) \\
    & = \big( \Phi_{\varepsilon} \ast \Phi_{\varepsilon}^{\ast (n)} \big) (y^{-1} x)
      = \Phi_{\eps}^{\ast(n+1)}(y^{-1} x) .
  \end{align*}
  Similarly, we see
  \(
    \big| (H - K)^{\circ(n+1)} (y,x) \big|
     \leq \int_G
             \Phi_\eps^{\ast(n)} (y^{-1} z) \,
             \Phi_\eps (z^{-1} x)
           \; d \mu_G (z)
      =    \Phi_\eps^{\ast(n+1)} (y^{-1} x) .
  \)
  Using these estimates, \eqref{eq:induction_product} follows by induction.

  \medskip{}

  \textbf{Step 6.} \emph{(Construction of the $w$-localized kernel $H_\phi$)}.
  Consider the series $\sum_{n = 1}^{\infty} a_n \, \Phi_{\eps}^{\ast(n)}$.
  Since $\eps \in (0, \frac{\delta}{2})$ was chosen such that
  $\|\Phi_{\eps} \|_{\WstCw} \leq \frac{\delta}{4}$, it follows by \eqref{eq:induction_convolution}
  and \eqref{eq:SpectralCalculusCoefficientBound} that
  \[
    \sum_{n = 1}^{\infty}
      |a_n| \, \| \Phi_{\eps}^{\ast(n)} \|_{\WstCw}
    \leq \sum_{n = 1}^{\infty}
           |a_n| \, \| \Phi_{\varepsilon} \|_{\WstCw}^{n} \\
    \leq C_\phi \sum_{n = 1}^{\infty}
                  \bigg( \frac{2}{\delta} \bigg)^{n}
                  \cdot \bigg( \frac{\delta}{4} \bigg)^{n}
    =    C_\phi ,
  \]
  showing that the series $\sum_{n = 1}^{\infty} \! a_n \, \Phi_{\eps}^{\ast(n)}$
  is norm-convergent in the Banach space ${\WstCw(G) \! \hookrightarrow \! C_b (G)}$.

  Let
  \(
    \widetilde{\Phi}
    := |a_0| \, \Theta '
       + \sum_{n = 1}^{\infty} |a_n| \, \Phi_{\varepsilon}^{\ast(n)}
    \in \WstCw(G)
    .
  \)
  Note that the kernel ${H_\phi : G \!\times\! G \to \CC}$ given by
  \[
    H_\phi (x,y)
    := a_0 \, K(x,y) + \sum_{n = 1}^{\infty} a_n \, (H - K)^{\circ (n)} (x,y)
    \qquad \text{for} \quad x,y \in G
  \]
  is well-defined with the series converging absolutely, and
  \begin{equation}
    \begin{split}
      |H_\phi (x,y)|
      & \leq |a_0| \,\, |K(x,y)|
             + \sum_{n = 1}^{\infty}
                 |a_n| \,\, |(H - K)^{\circ(n)} (x,y)| \\
      & \leq |a_0| \cdot \envker' (y^{-1} x)
             + \sum_{n = 1}^{\infty}
                 |a_n| \, \Phi_{\varepsilon}^{\ast (n)} (y^{-1} x)
        =    \widetilde{\Phi} (y^{-1} x)
    \end{split}
    \label{eq:SpectralCalculusKernelEnvelope}
  \end{equation}
  by \eqref{eq:induction_product}.
  Similar arguments also show that $|H_\phi (y,x)| \leq \widetilde{\Phi} (y^{-1} x)$,
  so that $\widetilde{\Phi} \in \WstCw(G)$ is an envelope for $H_\phi$;
  see Equation~\eqref{eq:envelope_w-localized}.

  To prove that $H_\phi$ is $w$-localized in $\RKHS$, it remains to show
  $H_\phi (\cdot, y) \in \RKHS$ and ${\overline{H_\phi (x, \cdot)} \in \RKHS}$ for all $x,y \in G$.
  To see this, note that $\widetilde{\Phi} \in \WstCw \subset L^1 \cap L^\infty \subset L^2 (G)$.
  In combination with \eqref{eq:SpectralCalculusKernelEnvelope}
  and the dominated convergence theorem, this implies that the series defining
  $H_\phi (\cdot, y)$ converges in $L^2(G)$.
  Since $\RKHS \subset L^2(G)$ is closed and since $K(\cdot, y) \in \RKHS$
  and $(H - K)^{\circ(n)}(\cdot, y) \in \RKHS$ (see Equation~\eqref{eq:IteratedKernelsAreInRKHS}),
  this implies $H_\phi (\cdot, y) \in \RKHS$, as claimed.
  The proof of $\overline{H_\phi (x, \cdot)} \in \RKHS$ is similar.

  \medskip

  \textbf{Step 7.} \emph{(Showing $\phi(T_H) = T_{H_\phi}|_{\RKHS}$)}.
  Let $f \in \RKHS$ and $x \in G$ be arbitrary.
  Then an application of the dominated convergence theorem
  (which is justified by the analogue of Equation~\eqref{eq:SpectralCalculusKernelEnvelope} which
  shows that $\sum_{n=1}^\infty |a_n| \, |(H-K)^{\circ(n)}(x,y)| \leq \widetilde{\Phi}(x^{-1} y)$,
  and because ${\widetilde{\Phi} \in \Wstw(G) \hookrightarrow L^2(G)}$),
  combined with the identity $T_K = \identity_{\RKHS}$, shows that
  \begin{align*}
    T_{H_\phi} f (x)
    & = a_0 \, T_K f(x)
          + \sum_{n = 1}^{\infty}
              a_n \, T_{(H - K)^{\circ (n)}} f(x) \\
    & = a_0 \, \big[ (T_H - \identity_{\RKHS})^0 f \big](x)
        + \sum_{n = 1}^{\infty}
            a_n \, \big[ (T_H - \identity_{\RKHS})^n f \big](x) \\
    & = \bigg(
          \sum_{n = 0}^{\infty}
            a_n \, \big( T_H - \identity_{\RKHS} \big)^n f
        \bigg)(x)
      = \big[ \phi(T_H) f \big] (x) ,
  \end{align*}
  where the last step used Equation~\eqref{eq:SpectralCalculusSeriesRepresentation}.
  Thus, $\phi(T_H) = T_{H_\phi}|_{\RKHS}$.
\end{proof}

\begin{remark}
  The number $\eps = \eps(\Theta, \Phi, \delta, w)$ in Theorem~\ref{thm:integral_eps_inverse}
  depends not only on the norm $\norm{\Phi}_{\Wstw}$ but on the full envelope $\Phi$.
\end{remark}

\subsection{Convolution-dominated matrices}
\label{sub:ConvolutionDominatedMatrixOperators}

In this section we consider  convolution-dominated matrices.
Matrices of similar type have been studied in a variety of settings in the literature;
see for instance
\cite{fendler2007on, fendler2008convolution, shin2009stability, sun2007wiener, tessera2010left}.

\begin{definition}\label{def:GoodMatrices}
Let $\Lambda , \Gamma$ be relatively separated families in $G$,
let $w : G \to (0,\infty)$ be measurable, and let
\(
  M
  = (M_{\lambda,\gamma})_{\lambda \in \Lambda, \gamma \in \Gamma}
  \in \CC^{\Lambda \times \Gamma}
\).
A non-negative function $\Theta \in \WstCw(G)$ is a \emph{$w$-envelope}
for $M$ (written $M \dominated \Theta$) if
\[
  |M_{\lambda,\gamma}|
  \leq \min \big\{ \Theta (\gamma^{-1} \lambda), \quad \Theta (\lambda^{-1} \gamma) \big\},
  \qquad \text{for all} \quad \lambda \in \Lambda \text{ and } \gamma \in \Gamma .
\]
Define the space
\[
  \goodMatrices (\Gamma, \Lambda)
  := \big\{
      M \in \CC^{\Lambda \times \Gamma}
       \colon
       \exists \, \Theta \in \controlFunctionSpace(G) \text{ such that } M \dominated \Theta
     \big\},
\]
and let
\[
  \| M \|_{\goodMatrices}
  := \inf \big\{
            \| \Theta \|_{\controlFunctionNormIndex}
            \; : \;
            M \dominated \Theta \in \controlFunctionSpace(G)
          \big\}
  < \infty
\]
for $M \in \goodMatrices (\Gamma,\Lambda)$.
In case of $\Lambda = \Gamma$, we will simply write
$\goodMatrices(\Lambda) := \goodMatrices(\Lambda,\Lambda)$.
\end{definition}

The following result collects some basic properties of convolution-dominated matrices.
For completeness we provide a proof in Appendix~\ref{sub:convolutionmatrix_appendix}.

\begin{proposition}\label{prop:GoodMatrixSpaceReasonable}
  Let $\Lambda,\Gamma,\Omega$ be relatively separated families in $G$
  and let $w : G \to (0,\infty)$ be an admissible weight.
  The following properties hold:
  \begin{enumerate}[label=(\roman*)]
    \item The pair $\big( \goodMatrices(\Gamma,\Lambda) , \| \cdot \|_{\goodMatrices} \big)$
          is a Banach space.
          \vspace{0.1cm}

    \item Each
          \(
            M
            = (M_{\lambda,\gamma})_{\lambda \in \Lambda, \gamma \in \Gamma}
            \in \goodMatrices(\Gamma,\Lambda)
          \)
          satisfies the following Schur-type conditions:
          \begin{equation}
            \sum_{\lambda \in \Lambda}
              |M_{\lambda,\gamma}|
            \leq \frac{\rel(\Lambda)}{\mu_G(Q)} \, \| M \|_{\goodMatrices}
            \quad \text{and} \quad
            \sum_{\gamma \in \Gamma}
              |M_{\lambda,\gamma}|
            \leq \frac{\rel(\Gamma)}{\mu_G(Q)} \, \| M \|_{\goodMatrices}.
            \label{eq:GoodMatricesAreSchur}
          \end{equation}

    \item The embedding
          \(
            \goodMatrices(\Gamma, \Lambda)
            \hookrightarrow \mathcal{B}(\ell^p(\Gamma), \ell^p (\Lambda))
          \)
          holds for all $1 \leq p \leq \infty$, with
          \begin{equation}
            \qquad \quad
            \| M \|_{\ell^p (\Gamma) \to \ell^p(\Lambda)}
            \leq \frac{ \max \{ \rel(\Lambda), \rel(\Gamma) \}}
                      {\mu_G (Q)}
                 \cdot \| M \|_{\goodMatrices (\Gamma,\Lambda)},
            \quad  \quad M \in \goodMatrices(\Gamma,\Lambda) .
            \label{eq:GoodMatricesActOnEllP}
          \end{equation}

    \item \label{enu:ConvoDominatedMatricesMultiplication}
          If $M \in \goodMatrices(\Gamma,\Omega)$
          and $N \in \goodMatrices (\Lambda,\Gamma)$, then the product
          $M N \in \goodMatrices(\Lambda,\Omega)$, with
          \[
            \| M N \|_{\goodMatrices}
            \leq \frac{2\rel(\Gamma)}{\mu_G(Q)}
                  \| M \|_{\goodMatrices} \,
                  \| N \|_{\goodMatrices} .
          \]
  \end{enumerate}
\end{proposition}

\subsection{Local spectral invariance for matrices}

The following result provides an analogue of Theorem~\ref{thm:integral_eps_inverse}
for convolution-dominated matrices.
The proof strategy is similar, with the technical caveat that the index sets may not be subgroups of $G$.

\begin{theorem}\label{thm:invertibility_convdom_matrix}
  Let $R > 0$, let $w : G \to (0,\infty)$ be an admissible weight,
  let $\Lambda$ be a relatively separated family in $G$ with $\rel (\Lambda) \leq R$,
  and let $\Phi \in \WstCw (G)$ be non-negative.

  For arbitrary $\delta > 0$,
  there exists $\eps = \eps(\Phi, R, \delta, w) \in (0,\delta)$ such that:
  If
  \begin{enumerate}
    \item $M \in \goodMatrices (\Lambda)$ has envelope $\Phi$ and satisfies
          \(
            \|
              M - \identity_{\ell^2 (\Lambda)}
            \|_{\ell^2 (\Lambda) \to \ell^2 (\Lambda)}
            \leq \eps
          \);
          \vspace{0.1cm}

    \item The function $\phi : B_\delta(1) \to \CC$ is holomorphic;
  \end{enumerate}
  then the operator $\phi(M) : \ell^2 (\Lambda) \to \ell^2(\Lambda)$
  defined through the holomorphic functional calculus is well-defined
  and its associated matrix satisfies $\phi(M) \in \goodMatrices (\Lambda)$.
\end{theorem}

\begin{proof}
Let $M \in \goodMatrices (\Lambda)$ with envelope $\Phi$.
The proof proceeds in four steps:

\medskip{}

\textbf{Step 1.} \emph{(Choosing $\eps$)}.
  Choose a symmetric function
  $\varphi \in C_c(G) \hookrightarrow \WstCw(G)$ satisfying $\varphi \geq 0$ and $\varphi(e) = 1$.
  Define $\Psi := \varphi + \Phi \in \WstCw (G)$.
  Furthermore, define $\Psi_k := \min \{ k^{-1}, \Psi \}$ for $k \in \mathbb{N}$.
  Then
  $
    \maxL \maxR \Psi_k (x)
    \leq \min \big\{ \tfrac{1}{k}, \maxL\maxR \Psi (x) \big\}
   \to 0 $
   as $k \to \infty$,
  with pointwise convergence.
  Since also $0 \leq \maxL \maxR \Psi_k \leq \maxL \maxR \Psi \in L_w^1(G)$,
  the dominated convergence theorem implies that
  $
    \| \Psi_k \|_{\WstCw}
    = \| \maxL \maxR \Psi_k \|_{L^1_w}
    \to 0$
    as $ k \to \infty.$
  Let $C_1 := \max \big\{ 1, \frac{2 R}{\mu_G (Q)} \big\}$
  and define $L := \frac{4}{\delta} C_1 > 0$.
  As we just saw, there is $k \in \N_{\geq 2/\delta}$
  such that $\| \Psi_k \|_{\WstCw} \leq L^{-1}$.
  Fix this choice of $k$ for the remainder of the proof and set $\eps := k^{-1}$,
  noting that $\eps \leq \delta/2 < \delta$ and that indeed
  $\eps = \eps(\Phi,R,\delta,w)$.

  \medskip{}

  \textbf{Step 2.} \emph{(A series representation of $\phi(M)$).}
  We write ${\phi(z) = \sum_{n=0}^\infty a_n \, (z - 1)^n}$ for all $z \in B_\delta (1)$
  and a suitable sequence $(a_n)_{n \in \N_0} \subset \CC$,
  with the series converging locally uniformly on $B_\delta (1)$.
  By assumption and by our choice of $\eps$, we have
  \(
    \| M - \identity_{\ell^2(\Lambda)} \|_{\ell^2(\Lambda) \to \ell^2(\Lambda)}
    \leq \eps \leq \delta/2
  \),
  and hence $\sigma(M) \subset \overline{B_{\delta/2}}(1) \subset B_\delta (1)$.
  Thus, by elementary properties of the holomorphic functional calculus
  (see \cite[Theorem~10.27]{RudinFA}), we see
  \begin{equation}
    \phi(M)
    = \sum_{n=0}^\infty
        a_n \, \big( M - \identity_{\ell^2(\Lambda)} \big)^n
    \label{eq:ConvoDominatedMatricesHolomorphicFunctionProofExpansion}
  \end{equation}
  with convergence of the series in the operator norm.

  By the Cauchy-Hadamard formula, it follows that
  $\delta \leq \big[ \limsup_{n\to\infty} |a_n|^{1/n} \big]^{-1}$.
  Thus, there is ${N = N(\phi, \delta) \in \N}$ satisfying $|a_n|^{1/n} \leq 2/\delta$
  for all $n \geq N$, and hence a ${C_\phi = C_\phi(\delta) > 0}$ such that
  \begin{equation}
    |a_n| \leq C_\phi \cdot (2/\delta)^n
    \qquad \text{for all } n \in \N_0 .
    \label{eq:ConvoDominatedMatricesHolomorphicFunctionProofCoefficientEstimate}
  \end{equation}

  \medskip{}

  \textbf{Step 3.} \emph{(Showing $\| M - \identity_{\ell^2(\Lambda)} \|_{\goodMatrices} \leq L^{-1}$)}.
  As usual, identify ${M - \identity_{\ell^2(\Lambda)} : \ell^2(\Lambda) \to \ell^2(\Lambda)}$
  with the matrix
  \(
    N
    = (N_{\lambda,\lambda'})_{\lambda, \lambda' \in \Lambda}
    \in \CC^{\Lambda \times \Lambda}
  \)
  with entries defined by
  \begin{align*}
    N_{\lambda,\lambda'}
    & := \big\langle
           (M - \identity_{\ell^2(\Lambda)}) \, \delta_{\lambda'}, \,\,
           \delta_\lambda
         \big\rangle
       = M_{\lambda, \lambda'} - \delta_{\lambda,\lambda'} .
  \end{align*}
  Since by assumption
  $\| M - \identity_{\ell^2(\Lambda)} \|_{\ell^2(\Lambda) \to \ell^2(\Lambda)} \leq \eps = k^{-1}$,
  we easily see $|N_{\lambda,\lambda'}| \leq k^{-1}$ for all $\lambda, \lambda' \in \Lambda$.
  Furthermore, since $\Phi$ is an envelope for $M$,
  \[
    |N_{\lambda,\lambda'}|
    =    \Big|
           M_{\lambda, \lambda'}
           - \delta_{\lambda,\lambda'}
         \Big|
    \leq \Phi \big( (\lambda')^{-1} \lambda \big) + \varphi \big( (\lambda')^{-1} \lambda \big)
    =    \Psi \big( (\lambda')^{-1} \lambda \big).
  \]
  With the same arguments, we also see $|N_{\lambda,\lambda'}| \leq \Psi (\lambda^{-1} \lambda')$.
  Overall, we have thus shown $N \dominated \Psi_k$, and hence
  \(
    \| N \|_{\goodMatrices}
    \leq \| \Psi_k \|_{\WstCw}
    \leq L^{-1};
  \)
  see Step~1.

  \medskip{}

  \textbf{Step 4.} \emph{(Convergence of the series
  \eqref{eq:ConvoDominatedMatricesHolomorphicFunctionProofExpansion} in $\goodMatrices(\Lambda)$)}.
  Recall that $C_1 = \max \big\{ 1, \frac{2 R}{\mu_G (Q)} \big\}$ and that $\rel (\Lambda) \leq R$.
  Therefore, we see by Part~\ref{enu:ConvoDominatedMatricesMultiplication}
  of Proposition~\ref{prop:GoodMatrixSpaceReasonable} and an easy induction that
  \[
    N^n \in \goodMatrices(\Lambda)
    \quad \text{with} \quad
    \| N^n \|_{\goodMatrices}
    \leq C_1^n \cdot \| N \|_{\goodMatrices}^n
    \leq (C_1 / L)^n
    =    (\delta / 4)^n
    \quad \text{for all } n \in \N.
  \]
  In view of the estimate~\eqref{eq:ConvoDominatedMatricesHolomorphicFunctionProofCoefficientEstimate}
  for the coefficients $a_n$, this implies
  \[
    \sum_{n=1}^\infty
      |a_n| \cdot \big\| \big( M - \identity_{\ell^2(\Lambda)} \big)^n \big\|_{\goodMatrices}
    \leq C_\phi \cdot
         \sum_{n=1}^\infty
           \Big(
             \frac{2}{\delta}
           \Big)^n
           \cdot \Big(
                   \frac{\delta}{4}
                 \Big)^n
    =   C_\phi \cdot \sum_{n=1}^\infty \Big( \frac{1}{2} \Big)^n
    < \infty .
  \]
  By completeness of $\goodMatrices(\Lambda)$, this implies that the series
  $T_0 := \sum_{n=1}^\infty a_n \, (M - \identity_{\ell^2(\Lambda)})^n$
  converges in $\goodMatrices (\Lambda)$.
  Finally, note that
  $(M - \identity_{\ell^2(\Lambda)})^0 = \identity_{\ell^2(\Lambda)} \in \goodMatrices (\Lambda)$
  as well, since $\identity_{\ell^2(\Lambda)} \dominated \varphi$.
  Therefore, ${a_0 \cdot (M - \identity_{\ell^2(\Lambda)})^0 + T_0 \in \goodMatrices}$.
  Since $\goodMatrices \hookrightarrow \mathcal{B} (\ell^2(\Lambda))$
  by Equation~\eqref{eq:GoodMatricesActOnEllP}
  and in view of the series representation
  \eqref{eq:ConvoDominatedMatricesHolomorphicFunctionProofExpansion} of $\phi(M)$,
  this shows that $\phi(M) \in \goodMatrices(\Lambda)$.
\end{proof}

\subsection{Systems of molecules and convolution-dominated operators}
\label{sub:molecules_convolutiondominated}

This section provides several results on the relation between systems of molecules
and convolution-dominated operators.

\begin{lemma}\label{lem:molecule_bessel}
  Let $\RKHS \subspace L^2 (G)$ be a RKHS,
  let $\Lambda$ be a relatively separated family in $G$,
  and let $(g_{\lambda})_{\lambda \in \Lambda}$ be a system of $w$-molecules in $\RKHS$
  for some admissible weight $w$.
  Then $(g_{\lambda})_{\lambda \in \Lambda}$ forms a Bessel sequence in $\RKHS$.
\end{lemma}

\begin{proof}
  Let $\Psi \in \Wstw(G)$ be an envelope for $(g_\lambda)_{\lambda \in \Lambda}$.
  Define $\Phi(x) := \min \{ \Psi(x), \Psi(x^{-1}) \}$ for $x \in G$,
  and note that $|g_\lambda(\cdot)| \leq L_\lambda \Phi$ and that $\Phi$
  is continuous and satisfies $\Phi \in L^1(G)$ and, moreover, $\Phi^{\vee} = \Phi \in \WL(G)$.
  Now, let $a = (a_\lambda)_{\lambda \in \Lambda} \in \ell^2(\Lambda)$ be finitely supported.
  Then, Lemma~\ref{lem:SynthesisOperatorQuantitativeBound} shows
  \(
    \big\|
      \sum_{\lambda \in \Lambda}
        a_\lambda \, g_\lambda
    \big\|_{L^2}
    \leq \big\|
           \sum_{\lambda \in \Lambda}
             |a_\lambda| \, L_\lambda \Phi
         \big\|_{L^2}
    \lesssim \| a \|_{\ell^2},
  \)
  where the implied constant only depends on $\Lambda$ and $\Phi$.
  This proves the claim.
\end{proof}

\begin{lemma}\label{lem:molecule_convintegeral}
  Let $\RKHS \subspace L^2(G)$ be a RKHS and let $w : G \to (0,\infty)$ be an admissible weight.
  Suppose $(g_{\lambda})_{\lambda \in \Lambda}$ is a system of $w$-molecules in $\RKHS$.
  Then the following assertions hold:
  \begin{enumerate}[label=(\roman*)]
    \item The kernel $H : G \times G \to \CC$ given by
          \begin{align}\label{eq:H_integralkernel}
                H(x,y) = \sum_{\lambda \in \Lambda}
                                g_{\lambda} (x) \, \overline{g_{\lambda}(y)}
          \end{align}
          is well-defined (with absolute convergence of the series) and $w$-localized in $\RKHS$;
          that is, there exists $\widetilde{\Phi} \in \WstCw (G)$ such that
          \begin{align}
            \max \big\{ |H(x,y)|, |H(y,x)| \big\} \leq \widetilde{\Phi} (y^{-1} x),
            \qquad \text{for all} \quad x,y \in G,
            \label{eq:localization_sum_molecules}
          \end{align}
          and
          $H(\cdot, y) \in \RKHS$ and $\overline{H(x, \cdot)} \in \RKHS$ for all $x,y \in G$.
          \vspace*{0.1cm}

    \item Suppose $\RKHS$ satisfies (\Ktwo).
          If $U \subset Q$ is such that $\Lambda$ is $U$-dense with
          associated disjoint cover $(U_{\lambda})_{\lambda \in \Lambda}$,
          and $(\tau_{\lambda})_{\lambda \in \Lambda} \subset [0,\infty)$
          satisfies $\tau_{\lambda} \leq C \cdot \mu_G (U_{\lambda})$
          for all $\lambda \in \Lambda$ and some $C > 0$,
          then for the vectors $g_{\lambda} := \tau_{\lambda}^{1/2} \cdot k_{\lambda} \in \RKHS$,
          the function $\widetilde{\Phi} \in \WstCw$ in \eqref{eq:localization_sum_molecules}
          can be chosen as $\widetilde{\Phi} = C \cdot([\maxL \Theta] \ast [\maxR \Theta])$,
          where $\Theta \in \WstCw(G)$ is as in \eqref{eq:localization_kernel}.
          \vspace*{0.1cm}

    \item The frame operator $\frameop : \RKHS \to \RKHS$ associated to
          $(g_{\lambda})_{\lambda \in \Lambda}$ coincides with $T_{H}|_{\RKHS} : \RKHS \to \RKHS$
          with the $w$-localized integral kernel $H$ defined in \eqref{eq:H_integralkernel}.
          \vspace*{0.1cm}

    \item If $H' : G \times G \to \CC$ is a $w$-localized kernel in $\RKHS$,
          then the family $(h_{\lambda})_{\lambda \in \Lambda}$ defined by
          \[
            h_{\lambda}(x)
            := \int_G H'(x,y) \, g_{\lambda} (y) \; d\mu_G (y)
            = (T_{H'} \, g_\lambda) (x)
          \]
          forms a system of $w$-molecules in $\RKHS$.
  \end{enumerate}
\end{lemma}

\begin{proof}
  (i) Let $\Phi \in \WstCw(G)$ be an envelope for $(g_{\lambda})_{\lambda \in \Lambda}$.
  Since $\Phi$ is continuous, we can apply Equation~\eqref{eq:StandardEstimateOne},
  which shows for all $x,y \in G$ that
  \[
    |H(x,y)|
    \leq \sum_{\lambda \in \Lambda}
           |g_\lambda(x)| \,\, |g_\lambda(y)|
    \leq \sum_{\lambda \in \Lambda}
           \Phi(\lambda^{-1} x) \, \Phi(y^{-1} \lambda)
    \leq \frac{\rel(\Lambda)}{\mu_G(Q)}
         (\maxL \Phi \ast \maxR \Phi) (y^{-1} x) .
  \]
  Since $\maxL \Phi \ast \maxR \Phi \in \WstCw (G)$
  by the convolution relation \eqref{eq:amalgam_convolution}
  and since $H(y,x) = \overline{H(x,y)}$,
  the localization estimate \eqref{eq:localization_sum_molecules} follows,
  and we see that the series defining $H$ converges absolutely.

  To show that $H(\cdot, y) \in \RKHS$ for $y \in G$,
  first note as a consequence of Equation~\eqref{eq:StandardEstimateTwo} that
  \(
    \sum_{\lambda \in \Lambda}
      |g_\lambda (y)|
    \leq \sum_{\lambda \in \Lambda}
           \Phi(y^{-1} \lambda)
    \leq \frac{\rel(\Lambda)}{\mu_G(Q)} \, \| \Phi \|_{\WL}
  \)
  for all $y \in G$.
  Also, ${C := \sup_{\lambda \in \Lambda} \| g_{\lambda} \|_{L^2}^2 < \infty}$,
  since $(g_{\lambda})_{\lambda \in \Lambda}$ is Bessel by Lemma~\ref{lem:molecule_bessel}.
  Thus,
  \(
    \sum_{\lambda \in \Lambda}
      |g_{\lambda}(y) | \,\, \| g_{\lambda} \|_{L^2}
    \leq C^{1/2} \,
         \frac{\rel(\Lambda)}{\mu_G (Q)} \,
         \| \Phi \|_{\WL}
    < \infty ,
  \)
  showing that the series defining $H(\cdot, y)$ converges in $L^2 (G)$.
  Since $g_{\lambda} \in \RKHS$ for all $\lambda \in \Lambda$,
  it follows that $H(\cdot, y) \in \RKHS$ for all $y \in G$.
  Hence, also $\overline{H(x, \cdot)} = H(\cdot, x) \in \RKHS$ for all $x \in G$.

  \medskip{}

  (ii) Let $x,y \in G$.
  For $\lambda \in \Lambda$ and $z \in \lambda U \subset \lambda Q$,
  we have $\lambda^{-1} x = \lambda^{-1} z z^{-1} x \in Q z^{-1} x$
  and $y^{-1} \lambda = y^{-1} z (\lambda^{-1} z)^{-1} \in y^{-1} z Q$.
  Since $Q$ is open and $\Theta$ is continuous,
  this implies $|k_\lambda (x)| \leq \Theta(\lambda^{-1} x) \leq \maxR \Theta (z^{-1} x)$
  and $|k_\lambda (y)| \leq \Theta(y^{-1} \lambda) \leq \maxL \Theta (y^{-1} z)$
  for all $\lambda \in \Lambda$ and $z \in \lambda U \supset U_\lambda$.
  Hence,
  \begin{align*}
    |H(x,y)|
    &\leq C \sum_{\lambda \in \Lambda}
              \mu_G(U_{\lambda}) \,
              |k_{\lambda} (x)| \,
              |k_{\lambda} (y)|
    \leq C \sum_{\lambda \in \Lambda}
             \int_{U_{\lambda}}
               \maxR \Theta (z^{-1} x) \,
               \maxL \Theta (y^{-1} z)
             \; d\mu_G (z) \\
    &= C
       \int_G
         \maxL \Theta (t) \,
         \maxR \Theta (t^{-1} y^{-1} x)
       \; d\mu_G (t)
    = C \cdot (\maxL \Theta \ast \maxR \Theta ) (y^{-1} x).
  \end{align*}
  This proves the claim since $H(y,x) = \overline{H(x,y)}$.

  \medskip{}

  (iii) A combination of Part~(i) and Lemma~\ref{lem:AdmissibleKernelProperties}
  shows that $T_{H}|_{\RKHS} : \RKHS \to \RKHS$ is well-defined and bounded.
  Let $f \in \RKHS$ and $x \in G$.
  Then, since the proof of Part~(i) shows for each $x \in G$ that the series defining
  $\overline{H(x,\cdot)} = \sum_{\lambda \in \Lambda} \overline{g_\lambda (x)} g_\lambda$
  converges in $L^2(G)$, it follows that
  \[
    T_H f(x)
    = \langle f, \overline{H(x,\cdot)} \rangle_{L^2}
    = \sum_{\lambda \in \Lambda}
        \langle
          f, \,
          \overline{g_{\lambda} (x)} \, g_{\lambda}
        \rangle_{L^2}
    = \sum_{\lambda \in \Lambda}
        \langle
          f,
          g_{\lambda}
        \rangle_{L^2} \,
        g_{\lambda} (x),
    \qquad \text{for all} \quad x \in G,
  \]
  and hence $\frameop = T_H|_{\RKHS}$.

  \medskip{}

  (iv) Let $ \Phi' \in \WstCw(G)$ be an envelope for $H'$.
  Since $H'$ is $w$-localized in $\RKHS$,
  Lemma~\ref{lem:AdmissibleKernelProperties}(i) shows
  that $h_{\lambda} \in \RKHS$ for all $\lambda \in \Lambda$.
  Moreover,
  \(
    |h_{\lambda} (x)|
    \leq \int_G \Phi'(y^{-1} x) \, \Phi (\lambda^{-1} y) \; d\mu_G (y)
    =    (\Phi \ast \Phi')(\lambda^{-1} x),
  \)
  and
  \(
    |h_\lambda (x)|
    \leq \int_G
           \Phi' (x^{-1} y) \, \Phi(y^{-1} \lambda)
         \, d \mu_G (y)
    = (\Phi' \ast \Phi) (x^{-1} \lambda)
  \)
  for all $x \in G$ and $\lambda \in \Lambda$.
  Since $\Phi \ast \Phi' + \Phi' \ast \Phi \in \WstCw(G)$
  by the convolution relation \eqref{eq:amalgam_convolution}, the result follows.
\end{proof}

\begin{remark}[Necessity of (\Ktwo)]\label{rem:KernelDecayNecessary}
  If $(g_\lambda)_{\lambda \in \Lambda}$ is a frame for $\RKHS$ with dual frame
  $(h_\lambda)_{\lambda \in \Lambda}$, then
  \[
    k(x,y)
    = \langle k_y, k_x \rangle
    = \sum_{\lambda \in \Lambda}
        \langle k_y, h_\lambda \rangle \, \langle g_\lambda, k_x \rangle
    = \sum_{\lambda \in \Lambda}
        g_\lambda(x) \, \overline{h_\lambda (y)}.
  \]
  Thus, as in the proof of Part~(i) of Lemma~\ref{lem:molecule_convintegeral},
  if $(g_\lambda)_{\lambda \in \Lambda}$ and $(h_\lambda)_{\lambda \in \Lambda}$ are both
  systems of $w$-molecules, then it follows that $k$ is $w$-localized.
\end{remark}

The last result of this section will be useful in studying Riesz sequences of molecules.

\begin{lemma}\label{lem:molecule_convmatrix}
  Let $\RKHS \subspace L^2(G)$ be a RKHS and let $w : G \to (0,\infty)$ be an admissible weight.
  Suppose $(g_{\lambda})_{\lambda \in \Lambda}$ is a system of $w$-molecules
  with envelope $\Phi \in \WstCw(G)$.
  Then the following assertions hold:
  \begin{enumerate}[label=(\roman*)]
    \item The Gramian matrix
          \(
            \gramian
            = \big(
                \langle
                  g_{\lambda'},
                  g_{\lambda}
                \rangle
              \big)_{\lambda, \lambda' \in \Lambda}
          \)
          associated to $(g_{\lambda} )_{\lambda \in \Lambda}$ satisfies
          \[
            |\langle g_{\lambda'}, g_{\lambda} \rangle|
            = |\langle g_{\lambda}, g_{\lambda'} \rangle|
            \leq \widetilde{\Phi} \big( (\lambda')^{-1} \lambda \big),
            \qquad \text{for all} \quad \lambda, \lambda' \in \Lambda,
          \]
          where $\widetilde{\Phi} \in \WstCw (G)$ is defined as
          $\widetilde{\Phi} := \Phi \ast \Phi$.
          Consequently, $\gramian \in \goodMatrices (\Lambda)$.
          \vspace*{0.1cm}

    \item For a relatively separated $\Gamma$ in $G$, let
          \({
            M = (M_{\gamma, \lambda} )_{\gamma \in \Gamma, \lambda \in \Lambda}
              \in \goodMatrices(\Lambda, \Gamma)
            .
          }\)
          Then also the family $(h_{\gamma} )_{\gamma \in \Gamma}$ defined by
          \[
            h_{\gamma}
            := M (g_\lambda)_{\lambda \in \Lambda}
            := \sum_{\lambda \in \Lambda}
                 M_{\gamma, \lambda} \,\, g_{\lambda}
          \]
          forms a system of $w$-molecules in $\RKHS$.
  \end{enumerate}
\end{lemma}

\begin{proof}
(i) For arbitrary $\lambda, \lambda' \in \Lambda$, we have
    \[
      |\langle g_{\lambda'}, g_{\lambda} \rangle|
      =    |\langle g_{\lambda}, g_{\lambda'} \rangle|
      \leq \int_G
             \Phi (x^{-1} \lambda) \, \Phi \big( (\lambda')^{-1} x \big)
           \; d\mu_G (x)
      = (\Phi \ast \Phi) \big( (\lambda')^{-1} \lambda \big).
    \]
    The convolution relation \eqref{eq:amalgam_convolution} shows that
    $\Phi \ast \Phi \in \WstCw(G)$, and hence $\gramian \in \goodMatrices(\Lambda)$.

\medskip{}

(ii) Let $\Theta \in \strongWiener(G)$ be an envelope function for $M$.
     Let $x \in G$ and $\gamma \in \Gamma$ be arbitrary.
     Then, Equation~\eqref{eq:StandardEstimateOne} shows that
     \(
       |h_\gamma(x)|
       \leq \sum_{\lambda \in \Lambda}
              \Theta(\gamma^{-1} \lambda) \, \Phi(\lambda^{-1} x)
       \leq \frac{\rel(\Lambda)}{\mu_G(Q)}
            (\maxL \Theta \ast \maxR \Phi) (\gamma^{-1} x)
     \)
     and likewise
     \(
       |h_\gamma (x)|
       \leq \sum_{\lambda \in \Lambda}
              \Theta(\lambda^{-1} \gamma) \, \Phi(x^{-1} \lambda)
       \leq \frac{\rel(\Lambda)}{\mu_G(Q)}
            (\maxL \Phi \ast \maxR \Theta) (x^{-1} \gamma) .
     \)
     Since we have
     $(\maxL \Theta) \ast (\maxR \Phi) + (\maxL \Phi) \ast (\maxR \Theta) \in \WstCw(G)$
     by the relation \eqref{eq:amalgam_convolution}, the result follows.
\end{proof}

\section{Dual frames of molecules}
\label{sec:FramesWithDualMolecules}

This section is devoted to proving the existence of frames of reproducing kernels
with dual systems that also form a system of molecules.

\subsection{Almost tight frames}

The following result provides auxiliary frames that are almost tight.
Note that condition (\Kthree) is assumed here.

\begin{lemma}\label{lem:sufficient_frame}
  Let $\RKHS \subspace L^2 (G)$ be a RKHS satisfying (\Ktwo) and (\Kthree)
  for some admissible weight $w : G \to (0,\infty)$.

  For every $\eps \in (0,1)$, there exists a compact unit neighborhood $U \subset G$ such that:
  If $\Lambda$ is relatively separated and $U$-dense in $G$ and
  $(U_\lambda)_{\lambda \in \Lambda}$ is a disjoint cover associated to $\Lambda$ and $U$,
  then $\big( \mu_G (U_{\lambda})^{1/2} k_{\lambda} \big)_{\lambda \in \Lambda}$
  forms a frame for $\RKHS$ with lower frame bound
  \(
     1 - \eps
  \)
  and upper frame bound
  \(
     1 + \eps
  \).
\end{lemma}

\begin{proof}
Let $\eps \in (0,1)$ be arbitrary.
Let $\eta : G \to [0,\infty)$ and $\Theta' \in \WstCw(G)$ be as provided
by the weak uniform continuity property \eqref{eq:WUC}.
Since $\eta(x) \to 0$ as $x \to e$, we can choose a compact unit neighborhood
$U \subset Q \subset G$ such that
\(
  \| \eta \|_{\sup, U}
  := \sup_{x \in U}
       \eta(x)
  \leq \eps / (1 + \| \Theta' \|_{\WL})
  .
\)
With this choice of $U$, let $\Lambda$ and $(U_\lambda)_{\lambda \in \Lambda}$
be as in the statement of the lemma.

Fix $f \in \RKHS$ and $\lambda \in \Lambda$.
Since $U_\lambda \subset \lambda U$, it follows that $\lambda^{-1} x \in U$ for $x \in U_\lambda$.
Thus, \eqref{eq:WUC} yields
\begin{align*}
  \bigg|
    \mu_G (U_\lambda) \, |f(\lambda)|^2
    - \int_{U_{\lambda}} |f(x)|^2
  \; d\mu_G (x) \bigg|
  & \leq \int_{U_\lambda} \big| |f(\lambda)|^2 - |f(x)|^2 \big| \; d \mu_G (x) \\
  & \leq \int_{U_\lambda}
           \eta(\lambda^{-1} x)
           \int_G
             |f(z)|^2 \, \Theta' (z^{-1} \lambda)
           \; d \mu_G (z)
         \; d \mu_G (x) \\
  & \leq \mu_G (U_\lambda) \,
        \| \eta \|_{\sup, U}
         \int_G |f(z) |^2 \, \Theta'(z^{-1} \lambda) \; d\mu_G (z) .
\end{align*}
An application of the triangle inequality gives
\begin{align*}
  \int_{U_{\lambda}} |f(x)|^2 \; d\mu_G (x)
  \leq \mu_G (U_\lambda) \, |f(\lambda)|^2
       + \mu_G (U_\lambda) \,
         \| \eta \|_{\sup, U}
         \int_G
           |f(z)|^2 \, \Theta' (z^{-1} \lambda)
         \; d\mu_G (z) ,
\end{align*}
and summing this inequality over $\lambda \in \Lambda$ yields
\begin{align} \label{eq:triangle_sum}
  \| f \|_{L^2}^2
  \leq \sum_{\lambda \in \Lambda}
         \mu_G (U_{\lambda}) \, |f(\lambda)|^2
       + \| \eta \|_{\sup, U}
         \int_G
           |f(z)|^2 \,
           \sum_{\lambda \in \Lambda}
             \mu_G (U_{\lambda}) \,
           \Theta' (z^{-1} \lambda)
         \; d\mu_G (z).
\end{align}
To further estimate the right-hand side of \eqref{eq:triangle_sum}, fix $z \in G$.
For arbitrary $y \in U_{\lambda} \subset \lambda U$, it follows that
$\lambda^{-1} y \in U \subset Q = Q^{-1}$ and hence
$z^{-1} \lambda = z^{-1} y (\lambda^{-1} y)^{-1} \in z^{-1} y Q$,
which implies that $\Theta' (z^{-1} \lambda) \leq (\maxL \Theta')(z^{-1} y)$,
since $Q$ is open and $\Theta'$ is continuous.
Therefore, for each $z \in G$,
\begin{align*}
  \sum_{\lambda \in \Lambda}
    \mu_G(U_{\lambda}) \, \Theta' (z^{-1} \lambda)
  & = \sum_{\lambda \in \Lambda}
        \int_{U_{\lambda}}
          \Theta' (z^{-1} \lambda)
        \; d \mu_G (y)
  \leq \int_G
         (\maxL \Theta') (z^{-1} y)
       \; d\mu_G (y)
    = \| \Theta' \|_{\WL} .
\end{align*}
Combining this with \eqref{eq:triangle_sum} yields
\(
  \| f \|_{L^2}^2
  \leq \sum_{\lambda \in \Lambda}
         \mu_G (U_{\lambda}) \,
         |f(\lambda)|^2
       + \| \eta \|_{\sup, U}
         \| \Theta' \|_{\WL} \,
         \| f \|_{L^2}^2.
\)
Since $\| \eta \|_{\sup,U} \leq \eps / (1 + \| \Theta' \|_{\WL})$, we thus see that
\[
  (1- \eps) \, \|f \|_{L^2}^2
  \leq \sum_{\lambda \in \Lambda}
         \mu_G (U_{\lambda}) \, |f(\lambda)|^2
  =    \sum_{\lambda \in \Lambda}
         |\langle f, \mu_G(U_\lambda)^{1/2} \, k_\lambda \rangle|^2 ,
\]
which is the desired lower bound.
The upper bound follows similarly.
\end{proof}

\subsection{Dual frames of molecules}
\label{sub:ConstructingFramesWithDualMolecules}

Using the existence of almost tight frames provided by Lemma \ref{lem:sufficient_frame},
we  derive our main result regarding the existence of well-localized dual frames.

\begin{theorem}\label{thm:frame_main1}
  Let $\RKHS \subspace L^2 (G)$ be a RKHS satisfying (\Ktwo) and (\Kthree)
  for some admissible weight $w : G \to (0,\infty)$.
  There exists a compact unit neighborhood $U \subset G$
  such that, for any relatively separated and $U$-dense $\Lambda $ in $ G$,
  the following assertions hold:
  \begin{enumerate}[label=(\roman*)]
  \item  The system $(k_{\lambda})_{\lambda \in \Lambda}$ is a frame for $\RKHS$
         and admits a dual frame $(h_{\lambda})_{\lambda \in \Lambda}$
         of $w$-molecules.

  \item There exists a tight frame $(g_{\lambda})_{\lambda \in \Lambda}$ for $\RKHS$
        which forms a family of $w$-molecules.
  \end{enumerate}
\end{theorem}
\begin{proof}
  Let $\Theta \in \Wstw(G)$ be as in \eqref{eq:localization_kernel}
  and set $\widetilde{\Phi} := \maxL \Theta \ast \maxR \Theta$.
  By Theorem~\ref{thm:integral_eps_inverse}, we see that there exists
  $\eps = \eps(\Theta, w) \in (0, 1)$ such that
  if $H : G \times G \to \CC$ is $w$-localized in $\RKHS$ with envelope $\widetilde{\Phi}$
  and if $\| T_H - \identity_{\RKHS} \|_{\RKHS \to \RKHS} \leq \eps$, then there are kernels
  $H_1, H_2 : G \times G \to \CC$ that are $w$-localized in $\RKHS$ and such that
  $T_H^{-1} = T_{H_1}|_{\RKHS}$ and $T_H^{-1/2} = T_{H_2}|_{\RKHS}$,
  where the operators on the left-hand side are defined by the holomorphic functional calculus.
  Next, choose the compact unit neighborhood $U \subset G$
  as provided by Lemma~\ref{lem:sufficient_frame} for this choice of $\eps$,
  and let $\Lambda$ be relatively separated and $U$-dense in $G$.

  By Lemma~\ref{lem:relativelyUdense_disjointunion}, there is a disjoint cover
  $(U_\lambda)_{\lambda \in \Lambda}$ associated to $\Lambda$ and $U$.
  Set $\tau_\lambda := \mu_G (U_\lambda)$ for $\lambda \in \Lambda$.
  Since $\tau_\lambda \leq \mu_G(\lambda U) = \mu_G(U) < \infty$ for all $\lambda \in \Lambda$,
  and by (\Ktwo), it follows that
  $(\tau_\lambda^{1/2} \, k_\lambda)_{\lambda \in \Lambda}$ forms a system of $w$-molecules.
  Furthermore, Lemma~\ref{lem:molecule_convintegeral} shows that the
  (pre)-frame operator $\frameop : \RKHS \to \RKHS$ associated to this family
  is an integral operator $\frameop = T_H |_{\RKHS}$ whose integral kernel
  $H : G \times G \to \CC$ is $w$-localized in $\RKHS$ with envelope $\widetilde{\Phi}$.
  Furthermore, our choice of $U$ (via Lemma~\ref{lem:sufficient_frame}) ensures that
  \({
    -\eps \, \| f \|_{L^2}^2
    \leq \langle (\frameop - \identity_{\RKHS}) f, f \rangle
    \leq \eps \, \| f \|_{L^2}^2
  }\)
  for all $f \in \RKHS$, and hence
  $\| \frameop - \identity_{\RKHS} \|_{\RKHS \to \RKHS} \leq \eps < 1$.
  This implies that
  $\frameop^{-1} = T_{H_1}|_{\RKHS}$ and $\frameop^{-1/2} = T_{H_2}|_{\RKHS}$,
  for suitable integral kernels $H_1, H_2 : G \times G \to \CC$ that are $w$-localized in $\RKHS$.

  With this preparation, we can prove the two statements of the theorem:

  (ii) Part~(iv) of Lemma~\ref{lem:molecule_convintegeral} shows that
  \(
    (g_\lambda)_{\lambda \in \Lambda}
    := \big(
         \frameop^{-1/2} (\tau_\lambda^{1/2} \, k_\lambda)
       \big)_{\lambda \in \Lambda}
  \)
  is a family of $w$-molecules in $\RKHS$.
  By elementary frame theory,
  $(g_\lambda)_{\lambda \in \Lambda}$ forms a \emph{tight} frame for $\RKHS$.

 (i) Recall that $(k_\lambda)_{\lambda \in \Lambda}$
  is a family of $w$-molecules, by~(\Ktwo).
  Thus Lemma~\ref{lem:molecule_bessel} shows that $(k_\lambda)_{\lambda \in \Lambda}$
  is a Bessel sequence in $\RKHS$, and hence so is the family
  \(
    (h_\lambda)_{\lambda \in \Lambda}
    := \big(
         \frameop^{-1} (\tau_\lambda \, k_\lambda)
       \big)_{\lambda \in \Lambda} ,
  \)
  since $\tau_\lambda = \mu_G (U_\lambda) \leq \ \mu_G (U) < \infty$
  for all $\lambda \in \Lambda$.
  Furthermore, since $\frameop^{-1} = T_{H_1}|_{\RKHS}$ for the $w$-localized kernel $H_1$,
  Lemma~\ref{lem:molecule_convintegeral} shows that
  $(\frameop^{-1} k_\lambda)_{\lambda \in \Lambda}$ is a family of $w$-molecules.
  Since $\tau_\lambda \leq \mu_G (U)$ for all $\lambda \in \Lambda$,
  this implies that $(h_\lambda)_{\lambda \in \Lambda}$ is a family of $w$-molecules as well.
  Since
  \[
    f
    = \frameop^{-1} \frameop f
    = \frameop^{-1}
      \Big(
        \sum_{\lambda \in \Lambda}
          \langle f, \tau_\lambda^{1/2} \, k_\lambda \rangle
          \,\, \tau_\lambda^{1/2} \, k_\lambda
      \Big)
    = \frameop^{-1}
      \Big(
        \sum_{\lambda \in \Lambda}
          \langle f, k_\lambda \rangle \, \frameop h_\lambda
      \Big)
    = \sum_{\lambda \in \Lambda}
        \langle f, k_\lambda \rangle \, h_\lambda
  \]
  for all $f \in \RKHS$, it follows that $(k_\lambda)_{\lambda \in \Lambda}$
  and $(h_\lambda)_{\lambda \in \Lambda}$ form a pair of dual frames.
\end{proof}

\subsection{Canonical dual frames of molecules}%
\label{sec:coherent_systems_with_localized_canonical_duals}

Our main result in this section is the following statement showing that if
$\Lambda$ is chosen such that ${\uniformity(\Lambda;U) < 1 + \eps}$
for a sufficiently small unit neighborhood $U \subset G$ and sufficiently small $\eps > 0$,
then the reproducing kernels $(k_\lambda)_{\lambda \in \Lambda}$ form a frame
whose \emph{canonical} dual frame again forms a system of molecules.

\begin{theorem}\label{thm:AbstractCanonicalDualUniformityCondition}
  Let $\RKHS \subspace L^2(G)$ be a RKHS satisfying (\Ktwo) and (\Kthree)
  for some admissible weight $w : G \to (0,\infty)$.
  There is a compact unit neighborhood $U \subset G$ and an $\eps > 0$ such that:
  If $\Lambda$ is relatively separated and $U$-dense in $G$ with uniformity
  $\uniformity(\Lambda;U) < 1 + \eps$, then  $(k_\lambda)_{\lambda \in \Lambda}$
  is a frame for $\RKHS$, and furthermore:
  \begin{enumerate}[label=(\roman*)]
    \item The associated inverse frame operator $\frameop^{-1} : \RKHS \to \RKHS$ coincides
          with an integral operator $ T_{H'}|_{\RKHS} : \RKHS \to \RKHS$
          for a $w$-localized kernel ${H' : G \times G \to \CC}$ in $\RKHS$.
          \vspace{0.1cm}

    \item The \emph{canonical} dual frame $(\frameop^{-1} k_\lambda)_{\lambda \in \Lambda}$
          of $(k_{\lambda})_{\lambda \in \Lambda}$ is a system of $w$-molecules in $\RKHS$.
  \end{enumerate}
\end{theorem}

\begin{proof}
  Let $U \subset G$ be a compact unit neighborhood and let $\Lambda$ be relatively separated
  and $U$-dense in $G$, with a disjoint cover $(U_{\lambda})_{\lambda \in \Lambda}$
  associated to $\Lambda$ and $U$.
  Let $(\tau_{\lambda})_{\lambda \in \Lambda} \subset [0,\infty)$ satisfy
  $\tau_{\lambda} \leq 2 \, \mu_G (U_{\lambda})$ for all $\lambda \in \Lambda$.
  By Theorem~\ref{thm:integral_eps_inverse} and Lemma~\ref{lem:molecule_convintegeral}(ii),
  there is $\eps = \eps(w,\Theta) \in (0, 1/8)$ such that if the (pre)-frame operator
  $\widetilde{\frameop} : \RKHS \to \RKHS$
  of $(\tau_{\lambda}^{1/2} \, k_{\lambda})_{\lambda \in \Lambda}$
  satisfies ${\| \widetilde{\frameop} - \identity_{\RKHS} \|_{\RKHS \to \RKHS} \leq 4 \eps}$, then
  $\big( \tau_\lambda^{1/2} \, k_\lambda \big)_{\lambda \in \Lambda}$ forms a frame for $\RKHS$
  whose inverse frame operator $\widetilde{\frameop}^{-1}$ is an integral operator
  whose integral kernel is $w$-localized in $\RKHS$.
  In particular, by Lemma~\ref{lem:molecule_convintegeral}(iv), the canonical dual frame of
  $\big( \tau_\lambda^{1/2} \, k_\lambda \big)_{\lambda \in \Lambda}$ forms a system of $w$-molecules.
  The remainder of the proof consists in constructing a suitable compact unit neighborhood
  $U \subset G$ such that if $\Lambda$ satisfies the assumptions of the theorem,
  then one can choose the disjoint cover $(U_\lambda)_{\lambda \in \Lambda}$
  and $\tau \in (0,\infty)$ such that if we set
  $\tau_\lambda = \tau \in (0,\infty)$ for all $\lambda \in \Lambda$,
  then $\tau_\lambda \leq 2 \, \mu_G (U_\lambda)$
  and $\| \widetilde{\frameop} - \identity_{\RKHS} \|_{\RKHS \to \RKHS} \leq 4 \eps$.

  Lemma~\ref{lem:sufficient_frame} yields a compact symmetric unit neighborhood $U \subset Q$
  such that, for any relatively separated and $U$-dense family $\Lambda$ in $G$
  with disjoint cover $(U_\lambda)_{\lambda \in \Lambda}$ associated to $\Lambda$ and $U$,
  the frame operator $\frameop_0 : \RKHS \to \RKHS$ of
  $\big( \mu_G(U_\lambda)^{1/2} \cdot k_\lambda \big)_{\lambda \in \Lambda}$
  satisfies $\| \identity_{\RKHS} - \frameop_0 \|_{\RKHS \to \RKHS} \leq \eps$.
  Let $\Lambda$ be relatively separated and $U$-dense in $G$
  with uniformity ${\uniformity(\Lambda;U) < 1 + \eps}$.
  Then, by definition, there is a cover $G = \bigcupdot_{\lambda \in \Lambda} U_\lambda$
  of measurable sets $U_\lambda \subset \lambda U$
  satisfying ${\frac{\mu_G (U_\lambda)}{\mu_G (U_{\lambda'})} \leq 1 + \eps}$
  for all $\lambda, \lambda' \in \Lambda$.
  Fix $\lambda_0 \in \Lambda$ and define
  $
    \tau_\lambda := \tau := \mu_G (U_{\lambda_0})
    $
    for all $ \lambda \in \Lambda.$
  Note that
  \(
    \tau_\lambda
    = \mu_G (U_{\lambda_0})
    \leq (1 + \eps) \, \mu_G (U_\lambda)
    \leq 2 \, \mu_G(U_\lambda)
  \)
  for all $\lambda \in \Lambda$.
  Similarly, for arbitrary $\lambda \in \Lambda$, it holds that
  $\frac{\mu_G(U_{\lambda_0})}{\mu_G(U_\lambda)} \leq 1+\eps$ and
  $\frac{\mu_G(U_\lambda)}{\mu_G(U_{\lambda_0})} \leq 1+\eps$.
  Thus
  \[
    \frac{\mu_G(U_{\lambda_0})}{\mu_G(U_\lambda)}
    \geq \frac{1}{1+\eps}
    \geq \frac{1 - \eps^2}{1+\eps}
    =    1 - \eps
  \]
  and $\Big| \frac{\mu_G (U_{\lambda_0})}{\mu_G(U_\lambda)} - 1 \Big| \leq \eps$.
  This shows that
  \(
    |\tau - \mu_G (U_\lambda)|
    = \mu_G (U_\lambda)
      \cdot \Big|
              \frac{\mu_G (U_{\lambda_0})}{\mu_G (U_\lambda)}
              - 1
            \Big|
    \leq \eps \cdot \mu_G (U_\lambda) .
  \)
  Since $\eps \, \| \frameop_0 \|_{\RKHS \to \RKHS} \leq \eps (1+\eps)   \leq 2 \eps$,
  it follows therefore that
  \begin{align*}
    \big|
      \langle \widetilde{\frameop} f, f \rangle
      - \langle \frameop_0 f , f \rangle
    \big|
    & = \bigg|
          \sum_{\lambda \in \Lambda}
            \Big(
              \big[ \tau - \mu_G \big( U_\lambda \big) \big]
              \cdot |\langle f, k_{\lambda} \rangle|^2
            \Big)
        \bigg|
     \leq \eps
           \sum_{\lambda \in \Lambda}
             \Big(
               \mu_G(U_\lambda) \cdot
               |\langle f, k_{\lambda} \rangle|^2
             \Big) \\
      &=    \eps \, \langle \frameop_0 f, f \rangle
      \leq 2 \eps \, \| f \|_{L^2}^2
  \end{align*}
  for all $f \in \RKHS$.
  Since $\widetilde{\frameop} - \frameop_0$ is Hermitian, this shows that
  $\| \widetilde{\frameop} - \frameop_0 \|_{\RKHS \to \RKHS} \leq 2 \eps$,
  and hence $\| \identity_{\RKHS} - \widetilde{\frameop} \, \| \leq 3 \eps$.
  By the choice of $\eps$, this implies that the system
  $(\tau^{1/2} \, k_\lambda)_{\lambda \in \Lambda}$ forms a frame of $w$-molecules for $\RKHS$
  whose inverse frame operator $\widetilde{\frameop}^{-1}$ is an integral operator
  whose integral kernel is $w$-localized in $\RKHS$.
  Since the frame operator $\frameop : \RKHS \to \RKHS$ of $(k_\lambda)_{\lambda \in \Lambda}$
  is given by $\frameop = \tau^{-1} \cdot \widetilde{\frameop}$,
  this implies all properties stated in the theorem.
\end{proof}

A family $\Lambda$ in $G$ satisfying the assumptions of the preceding theorem can always be chosen:

\begin{lemma}\label{lem:ExistenceOfSetsWithGoodUniformity}
  Let $U \subset G$ be an arbitrary unit neighborhood and let $\eps > 0$.
  Then there is a relatively separated and $U$-dense set $\Lambda$ with uniformity
  $\uniformity(\Lambda;U) < 1 + \eps$.
\end{lemma}

\begin{proof}
  First, consider the special case that $G$ is a discrete group, with counting measure $\mu_G$.
  Choose $\Lambda = G$ and $U_\lambda := \{\lambda\}$ for $\lambda \in \Lambda$,
  noting that $U_\lambda \subset \lambda U$, since $U$ contains the neutral element.
  Since $G$ is discrete and $Q \subset G$ is precompact, $Q$ is finite.
  Thus, $\Lambda$ is relatively separated, since
  ${\rel(\Lambda) = \sup_{x \in G} \# (\Lambda \cap xQ) \leq \#Q < \infty}$.
  Clearly, $G = \bigcupdot_{\lambda \in \Lambda} U_\lambda$.
  Finally, $\mu_G (U_\lambda) = \mu_G (\{e\})$ for all $\lambda \in \Lambda$, from which
  it easily follows that $\uniformity(\Lambda;U) \leq 1 $.

  In the remainder of the proof, assume that $G$ is non-discrete.
  Choose a compact symmetric unit-neighborhood $V \subset G$ satisfying $V V V V \subset U$,
  and set $W := V V$.
  By Lemma~\ref{lem:partition}, there is a set $\Lambda_0 \subset G$
  which is $V$-separated and $W$-dense and such
  that there is a partition $G = \bigcupdot_{\lambda \in \Lambda_0} W_\lambda$
  into relatively compact Borel sets $W_\lambda$ satisfying
  $\lambda V \subset W_\lambda \subset \lambda W$ for all $\lambda \in \Lambda_0$.

  Fix $N \in \N$ with $N \geq 10$ and $\frac{1}{N-1} \leq \frac{\eps}{2}$. Define
  \[
    N_\lambda := \bigg\lfloor \frac{\mu_G (W_\lambda)}{\mu_G (V)} \cdot N \bigg\rfloor
    \qquad \text{for all} \quad \lambda \in \Lambda_0 \, .
  \]
  Then $N_\lambda \geq N \geq 10$ for all $\lambda \in \Lambda_0$.
  More precisely, it holds that
  \begin{equation}
    \frac{\mu_G(W_\lambda)}{\mu_G (V)} \cdot (N - 1)
    \leq \frac{\mu_G(W_\lambda)}{\mu_G (V)} \cdot N - 1
    \leq N_\lambda
    \leq \frac{\mu_G (W_\lambda)}{\mu_G(V)} \cdot N
    \qquad \text{for all} \quad \lambda \in \Lambda_0 \, .
    \label{eq:FactorEstimate}
  \end{equation}
  By iteratively applying Lemma~\ref{lem:nondiscrete_meanvalue},
  it follows that for each $\lambda \in \Lambda_0$ there is a partition
  $W_\lambda = \bigcupdot_{\ell=1}^{N_\lambda} W_\lambda^{(\ell)}$ into measurable sets satisfying
  $\mu_G \big( W_\lambda^{(\ell)} \big) = \frac{\mu_G (W_\lambda)}{N_\lambda} > 0$;
  in particular, $W_\lambda^{(\ell)} \neq \emptyset$.
  For each $\lambda \in \Lambda_0$ and $\ell \in \{1, \dots, N_\lambda\}$, choose
  $x_\lambda^{(\ell)} \in W_\lambda^{(\ell)}$. Define
  \[
    I := \big\{
           (\lambda,\ell)
           \colon
           \lambda \in \Lambda_0 , \; \ell \in \{1, \dots, N_\lambda\}
         \big\}
    \quad \text{and} \quad
    \Lambda
    := \big\{
         x_\lambda^{(\ell)}
         \colon
         (\lambda,\ell) \in I
       \big\} .
  \]
  In the following, we show that $\Lambda$ satisfies the required properties.

  To see that $\Lambda$ is relatively separated, note that there are $M \in \N$ and
  $y_1, \dots, y_M \in G$ such that
  $Q \, W^{-1} \subset \bigcup_{t=1}^M y_t Q \vphantom{\rule[-0.18cm]{0cm}{0.5cm}}$,
  since $Q$ is a unit-neighborhood and $Q \, W^{-1} \subset G$ is relatively compact.
  Since $x_\lambda^{(\ell)} \in W_\lambda^{(\ell)} \subset W_\lambda \subset \lambda W$,
  it follows that if $x \in G$ is arbitrary and $x_\lambda^{(\ell)} \in x Q$,
  then $\lambda W \cap x Q \neq \emptyset$, and hence
  $\lambda \in x Q \, W^{-1} \subset \bigcup_{t=1}^M x y_t Q$.
  Therefore, for any $x \in G$,
  \[
    \sum_{\lambda \in \Lambda}
      \indicator_{x Q} (\lambda)
    \leq \sum_{(\lambda,\ell) \in I}
           \indicator_{x Q} \big( x_\lambda^{(\ell)} \big)
    \leq \sum_{\lambda \in \Lambda_0}
         \bigg[
           N_\lambda
           \sum_{t=1}^M
             \indicator_{x y_t Q} (\lambda)
         \bigg]
    \leq M \, N \cdot \frac{\mu_G (W)}{\mu_G (V)} \cdot \rel (\Lambda) ,
  \]
  where the last step follows from the right-most inequality in \eqref{eq:FactorEstimate}.

  Note that $G = \bigcupdot_{(\lambda,\ell) \in I} W_\lambda^{(\ell)}$,
  which in particular implies that the $x_\lambda^{(\ell)}$ are pairwise distinct.
  Furthermore, since
  $x_\lambda^{(\ell)} \in W_\lambda^{(\ell)} \subset W_\lambda \subset \lambda W$,
  it follows that $\lambda \in x_\lambda^{(\ell)} W$ by symmetry of $W$, and thus
  \(
    W_\lambda^{(\ell)}
    \subset \lambda W
    \subset x_\lambda^{(\ell)} WW
    \subset x_\lambda^{(\ell)} U
  \)
  for all $(\lambda, \ell) \in I$.
  Thus, $\Lambda$ is $U$-dense in $G$.

  Lastly, note that, for arbitrary $(\lambda,\ell), (\theta,k) \in I$,
  \begin{align*}
    \frac{\mu_G (W_\lambda^{(\ell)})}{\mu_G (W_{\theta}^{(k)})}
      = \frac{\mu_G (W_\lambda)}{\mu_G (W_\theta)} \cdot N_\theta \cdot N_\lambda^{-1}
    &   \leq \frac{\mu_G (W_\lambda)}{\mu_G (W_\theta)}
             \cdot \frac{\mu_G(W_\theta)}{\mu_G(V)} \, N
             \cdot \frac{\mu_G(V)}{\mu_G(W_\lambda)} \, (N-1)^{-1} \\
    & =    \frac{N}{N-1}
      =    1 + \frac{1}{N-1}
      \leq 1 + \frac{\eps}{2} ,
  \end{align*}
  where the first inequality follows from \eqref{eq:FactorEstimate}.
  Since $G = \bigcupdot_{(\lambda,\ell) \in I} W_\lambda^{(\ell)}$
  and $W_\lambda^{(\ell)} \subset x_\lambda^{(\ell)} U$,
  this implies $\uniformity(\Lambda;U) \leq 1 + \frac{\eps}{2} $, as desired.
\end{proof}

\subsection{Proofs of the theorems in the introduction.}
\label{sec_proofs}

Theorem~\ref{thm:frame_main1_intro} corresponds to Part~(i) of Theorem~\ref{thm:frame_main1},
and Theorem~\ref{thm:frame_main2_intro} is a consequence
of Theorem~\ref{thm:AbstractCanonicalDualUniformityCondition}.
Lastly, the statement of Theorem~\ref{thm:exist1_intro} on the canonical dual frame
follows from Theorem~\ref{thm:AbstractCanonicalDualUniformityCondition}
and Lemma~\ref{lem:ExistenceOfSetsWithGoodUniformity}, whereas the existence of tight frames
is proven in part (ii) of Theorem~\ref{thm:frame_main1}.

\section{Dual Riesz sequences of molecules}
\label{sec:RieszSequencesWithDualMolecules}

This section provides a proof of the main results on Riesz sequences.

\subsection{Almost orthogonal Riesz sequences}

We start with the construction of an auxiliary system of molecules
that forms an ``almost orthogonal'' Riesz sequence.

\begin{lemma}\label{lem:AlmostTightRieszSequenceExistence}
  Let $\RKHS \subspace L^2 (G)$ be a RKHS satisfying (\Kone) and (\Ktwo)
  for some admissible weight $w : G \to (0,\infty)$.

  For every $\eps \in (0,1)$, there exists a compact set $K \subset G$ such that for each
  \mbox{$K$-separated} family
  $\Lambda $ in $ G$, the system $(\normalized{k_{\lambda}})_{\lambda \in \Lambda}$
  consisting of normalized kernels $\normalized{k_x} := k_x / \| k_x \|_{L^2}$ forms
  a Riesz sequence in $\RKHS$ with lower Riesz bound $(1-\eps)^2$ and upper Riesz bound $(1+\eps)^2$.
\end{lemma}
\begin{proof}
  Note that if $\Theta$ is replaced by $\Theta_0 (x) := \min \{ \Theta(x), \Theta(x^{-1}) \}$,
  then~\eqref{eq:localization_kernel} still holds.
  Hence, it may be assumed that the envelope $\Theta$ is symmetric.
  The proof proceeds in three steps:

  \medskip{}

  \textbf{Step 1}
  \emph{(Estimating $\|\normalized{k_{\lambda}} \cdot \indicator_{\lambda K} \|_{L^2}$).}
  Since $G$ is $\sigma$-compact, there is an increasing sequence $(K_n)_{n \in \N}$ of compact
  sets $K_n \subset G$ such that $G = \bigcup_{n \in \N} K_n$.
  Since ${\Theta \in L^2(G)}$
  and $\indicator_{K_n^c} \cdot \Theta \to 0$ pointwise as $n \to \infty$ with
  $|\indicator_{K_n^c} \cdot \Theta|^2 \leq |\Theta|^2 \in L^1(G)$,
  the dominated convergence theorem yields that there exists $n_0 \in \N$ such that
  $\| \indicator_{K_{n_0}^c} \cdot \Theta \|_{L^2} \leq \lowdiag^{1/2} \cdot \eps / 2$,
  where $\lowdiag > 0$ is as in \eqref{eq:diagonal_kernel}.

  The estimate
  \(
    |k_\lambda (x)|
    \leq \Theta(\lambda^{-1} x)
    =    L_\lambda \Theta (x)
  \)
  yields, for any $\lambda \in G$ and any measurable set $K \subset G$
  with $K \supset K_{n_0}$, that
  \begin{align*}
    \big|
      \| \normalized{k_\lambda} \cdot \indicator_{\lambda K} \|_{L^2}^2
      - 1
    \big|
    & = \| k_\lambda \|_{L^2}^{-2} \cdot
        \big|
          \| k_\lambda \|_{L^2}^2
          - \| k_\lambda \cdot \mathds{1}_{\lambda K} \|_{L^2}^2
        \big|
      =    \| k_\lambda \|_{L^2}^{-2}
           \cdot \| k_\lambda \cdot \indicator_{(\lambda K)^c} \|_{L^2}^2 \\
    & \leq \| k_\lambda \|_{L^2}^{-2}
           \cdot \big\| (L_\lambda \Theta) \cdot (L_\lambda \indicator_{K_{n_0}^c}) \big\|_{L^2}^2
     =    \| k_\lambda \|_{L^2}^{-2} \cdot \| \Theta \cdot \indicator_{K_{n_0}^c} \|_{L^2}^2
     \leq  \bigg(\frac{\eps}{2}\bigg)^2 .
  \end{align*}
  Since $\eps < 1$, it follows that
  \[
    \Big( 1 - \frac{\eps}{2} \Big)^2
    \leq 1 - \Big( \frac{\eps}{2} \Big)^2
    \leq \| \normalized{k_\lambda} \cdot \indicator_{\lambda K}\|_{L^2}^2
    \leq 1 + \Big( \frac{\eps}{2} \Big)^2
    \leq \Big( 1 + \frac{\eps}{2} \Big)^2,
  \]
  and hence
$
    1 - \frac{\eps}{2}
    \leq \big\|\, \normalized{k_\lambda} \cdot \indicator_{\lambda K} \big\|_{L^2}
    \leq 1 + \frac{\eps}{2}
    $
   for $\lambda \in G$ and $K \subset G$ measurable with $K \supset K_{n_0}$.

  \medskip{}

  \textbf{Step 2}
  \emph{(Construction of a compact set $K_0$).}
  Let $(K_n)_{n \in \N}$ be the family from Step~1, and define
  $\widetilde{K_n} := (K_n \overline{Q})^{-1}$.
  Note because of $\Theta = \Theta^{\vee}$ that if
  \[
    0 \neq \maxL \big( (\Theta \cdot \indicator_{\widetilde{K}_n^c} )^{\vee} \big) (x)
      =    \big\| (\envker \cdot \indicator_{\widetilde{K}_n^c})^\vee \big\|_{L^\infty (x Q)}
      =    \big\| \envker \cdot \indicator_{(K_n \overline{Q})^c} \big\|_{L^\infty (x Q)},
  \]
  then $\emptyset \neq (K_n \overline{Q})^c \cap x Q$.
  Hence, there is some $q \in Q$ satisfying ${x q \in (K_n \overline{Q})^c}$,
  whence $x \notin K_n$.
  Thus,
  \(
    0 \leq \maxL \big( (\Theta \cdot \indicator_{\widetilde{K}_n^c})^{\vee} \big)
      \leq \indicator_{K_n^c} \cdot \maxL \envker
      \to 0
  \)
  pointwise as $n \to \infty$.
  By the dominated convergence theorem,
  it therefore follows that
  \(
    \big\| \big( \Theta \cdot \indicator_{\widetilde{K}_n^c} \big)^\vee \big\|_{\WL}
    \to 0
  \)
  as $n \to \infty$.
  Hence, there is some $n' \in \N$ such that if  $K_0 := \widetilde{K}_{n'}$, then
  \(
    \| \Theta \|_{L^1}
     \big\| (\Theta \cdot \indicator_{K_0^c})^\vee \big\|_{\WL}
    < \lowdiag \cdot \mu_G (Q) \cdot \big( \eps / 2 \big)^2 .
  \)

  \medskip{}

  \textbf{Step 3} \emph{(Completing the proof).}
  With $K_{n_0}$ and $K_0$ as above, choose $\varphi \in C_c (G)$
  with ${0 \leq \varphi \leq 1}$ and $\varphi \equiv 1$ on $K_0$.
  Let $K_0 ' := \supp \varphi$ and define
  $K := \overline{Q} \cup K_{n_0} \cup K_0 '$
  and ${\Theta_\varphi := \Theta \cdot (1 - \varphi)}$,
  noting that $\Theta_{\varphi} \leq \Theta \cdot \indicator_{K_0^c}$ and hence
  \begin{equation}
    \| \Theta_\varphi \|_{L^1}  \| \Theta_\varphi^{\vee} \|_{\WL}
    \leq \| \Theta \|_{L^1}  \big\| (\Theta \cdot \indicator_{K_0^c})^{\vee} \big\|_{\WL}
    <    \lowdiag \cdot \mu_G(Q) \cdot \big( \eps / 2 \big)^2 .
    \label{eq:AlmostOrthogonalRieszSequenceMainStep}
  \end{equation}
  Let $\Lambda$ be $K$-separated in $G$
  and set $g_\lambda := \normalized{k_\lambda} \cdot \indicator_{\lambda K}$
  for $\lambda \in \Lambda$.
  Note that if
  $x \in (\lambda K)^c$, then ${\lambda^{-1} x \notin K \supset K_0 ' = \supp \varphi}$
  and therefore $1 - \varphi(\lambda^{-1} x) = 1$, yielding that
  $\indicator_{(\lambda K)^c} \leq L_\lambda (1 - \varphi)$, and
  \({
    |\normalized{k_\lambda} - g_\lambda|
    = \| k_\lambda \|_{L^2}^{-1} \cdot |k_\lambda| \cdot \indicator_{(\lambda K)^c}
    \leq \lowdiag^{-1/2} \cdot L_\lambda \Theta_\varphi
  }\).
  Therefore, if $c = (c_\lambda)_{\lambda \in \Lambda} \in \CC^\Lambda$ is finitely supported, then
   \begin{align*}
    \bigg|
      \Big\|
        \sum_{\lambda \in \Lambda}
          c_\lambda \, \widetilde{k_\lambda}
      \Big\|_{L^2}
      - \Big\|
          \sum_{\lambda \in \Lambda}
            c_\lambda \, g_\lambda
        \Big\|_{L^2}
    \bigg|
    & \leq \Big\|
             \sum_{\lambda \in \Lambda}
               c_\lambda \cdot (\, \normalized{k_\lambda} - g_\lambda \,)
           \Big\|_{L^2}
      \leq \lowdiag^{-1/2} \, \Big\|
                                \sum_{\lambda \in \Lambda}
                                  |c_\lambda| \, L_\lambda \Theta_\varphi
                              \Big\|_{L^2}.
  \end{align*}
  To estimate this further, note that since $\Lambda$ is $K$-separated in $G$
  and $K \supset Q = Q^{-1}$, it follows that $\rel(\Lambda) \leq 1$.
  Therefore, Equation~\eqref{eq:AlmostOrthogonalRieszSequenceMainStep}
  and Lemma~\ref{lem:SynthesisOperatorQuantitativeBound} imply that
  \begin{align*}
    \Big\|
      \sum_{\lambda \in \Lambda}
        |c_\lambda| \, L_\lambda \Theta_\varphi
    \Big\|^2_{L^2}
    =
      \Big\|
        D_{\Theta_\varphi, \Lambda} \, (|c_\lambda|)_{\lambda \in \Lambda}
      \Big\|_{L^2}^2
    & \leq
           \frac{1}{\mu_G (Q)} \,
            \| \Theta_\varphi^{\vee} \|_{\WL} \,
            \| \Theta_\varphi \|_{L^1}
            \| c \|_{\ell^2}^2 \!
      \leq \lowdiag \cdot
            \big( \eps / 2 \big)^2 \cdot
            \| c \|_{\ell^2}^2
    .
  \end{align*}
  Hence,
  \[
    \bigg|
      \Big\|
        \sum_{\lambda \in \Lambda}
          c_\lambda \, \widetilde{k_\lambda}
      \Big\|_{L^2}
      - \Big\|
          \sum_{\lambda \in \Lambda}
            c_\lambda \, g_\lambda
        \Big\|_{L^2}
    \bigg|
    \leq \frac{\eps}{2} \cdot \| c \|_{\ell^2}.
  \]
  On the other hand, since $(\lambda K)_{\lambda \in \Lambda}$ is pairwise disjoint, the family
  \(
    (g_\lambda)_{\lambda \in \Lambda}
    = (\normalized{k_\lambda} \cdot \indicator_{\lambda K})_{\lambda \in \Lambda}
  \)
  is orthogonal.
  This, in combination with $
    1 - \frac{\eps}{2}
    \leq \big\|\, \normalized{k_\lambda} \cdot \indicator_{\lambda K} \big\|_{L^2}
    \leq 1 + \frac{\eps}{2}
    $ from Step~1, yields
  \[
    \Big( 1 - \frac{\eps}{2} \Big) \cdot \| c \|_{\ell^2}
    \leq \Big\| \sum_{\lambda \in \Lambda} c_\lambda \, g_\lambda \Big\|_{L^2}
    \leq \Big( 1 + \frac{\eps}{2} \Big) \cdot \| c \|_{\ell^2} .
  \]
  Combining the obtained estimates with the triangle inequality gives
  \[
    \big( 1 - \eps \big) \cdot \| c \|_{\ell^2}
    \leq \bigg\|
           \sum_{\lambda \in \Lambda}
             c_\lambda \, \normalized{k_\lambda}
         \bigg\|_{L^2}
    \leq \big( 1 + \eps \big) \cdot \| c \|_{\ell^2}
  \]
  for every finitely supported sequence $c \in \CC^{\Lambda}$, as desired.
\end{proof}

\subsection{Biorthogonal systems of molecules}
\label{sub:RieszSequenceDualMolecules}

Using the auxiliary Riesz sequence constructed in Lemma~\ref{lem:AlmostTightRieszSequenceExistence},
we now prove the existence of Riesz sequences whose biorthogonal
system forms a family of molecules.
The following theorem corresponds to Theorem~\ref{thm:riesz_main2_intro} of the introduction.

\begin{theorem}\label{thm:riesz_main}
  Let $\RKHS \subspace L^2 (G)$ be a RKHS
  satisfying (\Kone) and (\Ktwo) for some admissible weight $w : G \to (0, \infty)$.
  There exists a compact unit neighborhood $K \subset G$ such that,
  for every $K$-separated family $\Lambda$ in $G$, the following assertions hold:
  \begin{enumerate}[label=(\roman*)]
  \item The family $(k_{\lambda} )_{\lambda \in \Lambda}$ is a Riesz sequence in $\RKHS$,
        and its unique biorthogonal system
        \({
          (h_{\lambda})_{\lambda \in \Lambda}
          \subset \overline{\mathrm{span}} \{ k_\lambda \colon \lambda \in \Lambda \}
        }\)
        is a family of $w$-molecules.
        \vspace{0.1cm}

  \item There exists an orthonormal sequence $(g_{\lambda})_{\lambda \in \Lambda}$
        in $\overline{\mathrm{span}} \{ k_\lambda \colon \lambda \in \Lambda \}$
        consisting of $w$-molecules.
\end{enumerate}
\end{theorem}

\newcommand{\GramianTilde}{\!\widetilde{\,\gramian\,}\!}

\begin{proof}
  Let $\Phi := \lowdiag^{-1} \cdot (\Theta \ast \Theta) \in \Wstw(G)$, with $\alpha$ and $\Theta$
  as in \eqref{eq:diagonal_kernel} and \eqref{eq:localization_kernel}.
  Theorem~\ref{thm:invertibility_convdom_matrix} yields
  $\eps = \eps(\alpha,\Theta,w) \in (0,1)$ such that for any $Q$-separated family $\Lambda$ in $G$
  and any $M \in \goodMatrices(\Lambda)$ with envelope $\Phi$ and satisfying
  $\| M - \identity_{\ell^2(\Lambda)} \|_{\ell^2 \to \ell^2} \leq \eps$, we have
  $M^{-1} \in \goodMatrices(\Lambda)$ and $M^{-1/2} \in \goodMatrices(\Lambda)$ as well.
  Using Lemma~\ref{lem:AlmostTightRieszSequenceExistence}, choose a compact set
  $K \subset G$ such that $K \supset Q$ and such that for every $K$-separated family
  $\Lambda$ in $G$, the family $(\widetilde{k_\lambda})_{\lambda \in \Lambda}$ of normalized kernels
  $\widetilde{k_\lambda} := k_\lambda / \| k_\lambda \|_{L^2}$ is a Riesz sequence
  with lower Riesz bound $(1 - \frac{\eps}{3})^2$ and upper Riesz bound $(1 + \frac{\eps}{3})^2$.

  Let $\Lambda$ be $K$-separated in $G$.
  By Lemma~\ref{lem:molecule_convmatrix} and because of $\| k_\lambda \|_{L^2}^2 \geq \lowdiag$,
  the Gramian $\GramianTilde \in \mathcal{B}(\ell^2(\Lambda))$
  of the family $(\widetilde{k_\lambda})_{\lambda \in \Lambda}$ satisfies
  $\GramianTilde \in \goodMatrices(\Lambda)$ with envelope $\Phi$ as defined above.
  Furthermore, the fact that $(\widetilde{k_\lambda})_{\lambda \in \Lambda}$
  has lower and upper Riesz bounds $(1 - \frac{\eps}{3})^2$ and $(1 + \frac{\eps}{3})^2$ means that
  \({
    \| \identity_{\ell^2(\Lambda)} - \GramianTilde \|_{\ell^2 \to \ell^2}
    \leq \max \big\{ (1 + \frac{\eps}{3})^2 - 1, 1 - (1 - \frac{\eps}{3})^2 \big\}
    \leq \eps .
  }\)
  By choice of $\eps$, this implies $\GramianTilde^{-1} \in \goodMatrices(\Lambda)$
  and $\GramianTilde^{-1/2} \in \goodMatrices(\Lambda)$.
  With this preparation, we now prove both parts of the theorem.

  (i) Note that $\overline{\GramianTilde^{-1}} \in \goodMatrices (\Lambda)$,
  where $\overline{\GramianTilde^{-1}}$ denotes the matrix obtained from
  $\GramianTilde^{-1}$ by conjugating each entry.
  Using the notation of Lemma~\ref{lem:molecule_convmatrix}, define
  \(
    (\widetilde{h_{\lambda'}})_{\lambda' \in \Lambda}
    := \overline{\GramianTilde^{-1}} ( \,\widetilde{k}_\lambda)_{\lambda \in \Lambda}
  \)
  and
  \({
    h_{\lambda'}
    := \| k_{\lambda'} \|_{L^2}^{-1} \cdot \widetilde{h_{\lambda'}}
  }\)
  for $\lambda' \in \Lambda$.
  Since $\| k_{\lambda} \|_{L^2}^2 \geq \lowdiag$ and since $(k_\lambda)_{\lambda \in \Lambda}$
  is a family of $w$-molecules by condition (\Ktwo),
  we see that also $(\widetilde{k_\lambda})_{\lambda \in \Lambda}$ is a family of $w$-molecules.
  Therefore, Lemma~\ref{lem:molecule_convmatrix} shows that the same holds for the families
  $(\widetilde{h_\lambda})_{\lambda \in \Lambda}$ and $(h_{\lambda})_{\lambda \in \Lambda}$.
  Furthermore, the series defining
  \(
    \widetilde{h_{\lambda'}}
    = \sum_{\lambda \in \Lambda}
        \overline{(\GramianTilde^{-1})_{\lambda', \lambda}} \, \widetilde{k}_\lambda
  \)
  converges in $L^2(G)$ by \eqref{eq:GoodMatricesAreSchur},
  which implies that
  \(
    \widetilde{h_{\lambda'}}
    \in \overline{\vphantom{h} \mathrm{span}} \{ k_\lambda \colon \lambda \in \Lambda \}
  \)
  and hence
  \(
    h_{\lambda'}
    \in \overline{\vphantom{h} \mathrm{span}} \{ k_\lambda \colon \lambda \in \Lambda \}
  \)
  for all $\lambda' \in \Lambda$.
  To show that $(h_\lambda)_{\lambda \in \Lambda}$ is biorthogonal
  to $(k_\lambda)_{\lambda \in \Lambda}$, we compute
  \begin{align*}
    \langle h_{\lambda'}, k_\lambda \rangle
    & = \| k_{\lambda'} \|_{L^2}^{-1}
        \Big\langle
          \sum_{\nu \in \Lambda}
            \overline{(\GramianTilde^{-1})_{\lambda',\nu}} \,\, \widetilde{k}_\nu,
          k_\lambda
        \Big\rangle
      = \frac{\| k_\lambda \|_{L^2}}{\| k_{\lambda'} \|_{L^2}} \,\,
        \sum_{\nu \in \Lambda}
          \overline{(\GramianTilde^{-1})_{\lambda',\nu}} \,
          \langle \widetilde{k}_\nu, \widetilde{k}_\lambda \rangle \\
    & = \frac{\| k_\lambda \|_{L^2}}{\| k_{\lambda'} \|_{L^2}} \,\,
        \overline{
          \sum_{\nu \in \Lambda}
            (\GramianTilde^{-1})_{\lambda',\nu} \,
            (\GramianTilde)_{\nu,\lambda}
        }
      = \frac{\| k_\lambda \|_{L^2}}{\| k_{\lambda'} \|_{L^2}} \,\,
        \overline{(\identity_{\ell^2(\Lambda)})_{\lambda',\lambda}}
      = \delta_{\lambda',\lambda},
  \end{align*}
  which completes the proof of Part~(i).

  For (ii), similar arguments as in (i) show that the system
  \(
    (g_{\lambda'})_{\lambda' \in \Lambda}
    = \overline{\GramianTilde^{-1/2}} \, ( \widetilde{k_{\lambda}})_{\lambda \in \Lambda}
  \)
  forms a system of $w$-molecules, and that
  \(
    g_{\lambda'}
    \in \overline{\vphantom{h} \mathrm{span}} \{ k_\lambda \colon \lambda \in \Lambda \}
  \).
  For brevity, let us set $A := \GramianTilde$, noting that $A = A^\ast$
  and hence $A^{-1/2} = (A^{-1/2})^\ast$ as well, meaning
  $(\overline{A^{-1/2}})_{\lambda,\gamma} = (A^{-1/2})_{\gamma,\lambda}$
  for all $\lambda, \gamma \in \Lambda$.
  Using this identity and that
  $A_{\theta,\gamma} = \langle \widetilde{k_\gamma}, \widetilde{k_\theta} \rangle$,
  a direct computation shows
  \begin{align*}
    \langle g_{\lambda'}, g_\lambda \rangle
    & = \sum_{\gamma, \theta \in \Lambda}
          \Big\langle
            \big( \overline{A^{-1/2}} \big)_{\lambda', \gamma} \, \widetilde{k_\gamma} , \,\,
            \big( \overline{A^{-1/2}} \big)_{\lambda, \theta} \, \widetilde{k_\theta}
          \Big\rangle
      = \sum_{\gamma, \theta \in \Lambda}
          (A^{-1/2})_{\gamma,\lambda'} \,
          (A^{-1/2})_{\lambda, \theta} \,
          A_{\theta, \gamma} \\
    & = \sum_{\gamma \in \Lambda}
          (A^{-1/2} A)_{\lambda, \gamma} \, (A^{-1/2})_{\gamma, \lambda'}
      = \big( A^{-1/2} A A^{-1/2} \,\big)_{\lambda, \lambda'}
      = \delta_{\lambda, \lambda'} ,
  \end{align*}
  for all $\lambda, \lambda' \in \Lambda$, since $A$ and $A^{-1/2}$ commute.
  Thus, $(g_\lambda)_{\lambda \in \Lambda}$ is an orthogonal family.
\end{proof}

\section{Coorbit spaces associated to integrable group representations and other examples}
\label{sec_examples_a}

\subsection{Coorbit spaces} \label{sec:CoorbitSpaces}
The main ingredient of coorbit theory is an irreducible, square-in\-te\-grable unitary representation
$(\pi, \Hpi)$ of a locally compact group $G$ on a Hilbert space $\Hpi$.
For $g \in \Hpi$, define $V_g : \Hpi \to C_b (G)$
by $ V_g f (x) = \langle f, \pi (x) g \rangle$ for $x \in G$.
The \emph{orthogonality relations} for square-integrable representations
 \cite{carey1976square,duflo1976on}
show that there is a unique densely defined, self-adjoint, positive operator
$C_{\pi} : \dom(C_{\pi}) \to \Hpi$ such that:

\begin{enumerate}
  \item[(1)] The domain of $C_{\pi}$ consists precisely of the \emph{admissible vectors}
             of $(\pi, \Hpi)$, that is, it satisfies
             \({
               \dom(C_{\pi})
               = \{
                   g \in \Hpi: V_g g \in L^2 (G)
                 \}
               = \{
                   g \in \Hpi \colon V_g f \in L^2(G) \text{ for all } f \in \Hpi
                 \}.
             }\)
             \vspace*{0.1cm}

  \item[(2)] For all $f_1, f_2 \in \Hpi$ and $g_1, g_2 \in \dom(C_{\pi})$, it holds that
             \begin{align}\label{eq:ortho_relations_a}
               \langle V_{g_1} f_1 , \,\, V_{g_2} f_2 \rangle_{L^2(G)}
               = \langle C_{\pi} \, g_2, C_{\pi} \, g_1 \rangle_{\Hpi}
                 \cdot \langle f_1, f_2 \rangle_{\Hpi} .
             \end{align}
\end{enumerate}
This implies for every ${0 \neq g \in \dom(C_\pi)}$
that $V_g (\Hpi) \subset L^2(G)$ is a RKHS with reproducing kernel
\begin{equation}
  k(x,y) = \| C_\pi g \|_{\Hpi}^{-2} \cdot V_g g (y^{-1} x) .
  \label{eq:CoorbitReproducingKernel}
\end{equation}
Indeed, \eqref{eq:ortho_relations_a} implies that
\(
  V_g f (x)
  = \langle f, \pi(x) g \rangle_{\Hpi}
  = \big\langle V_g f, \| C_\pi g \|_{\Hpi}^{-2} \cdot V_g [\pi(x) g] \big\rangle
  = \langle V_g f, k(\cdot, x) \rangle
\)
for all $f \in \Hpi$ and $x \in G$.
Completeness of $V_g (\Hpi)$ follows from
$\| V_g f \|_{L^2(G)}^2 = \| C_\pi g \|_{\Hpi}^2 \, \| f \|_{\Hpi}^2$
which is again a consequence of \eqref{eq:ortho_relations_a}.

The discretization scheme from \cite{feichtinger1989banach1,feichtinger1989banach2,groechenig1991describing}
is concerned with sampling the reproducing formula associated to $V_g (\Hpi)$
to obtain coherent frames $\bigl(\pi(x_i) g\bigr)_{i \in I}$ for $\Hpi$.
In fact, the results in \cite{feichtinger1989banach1,feichtinger1989banach2,groechenig1991describing}
show that the frame expansions do not only hold in $\Hpi$, but
that they extend to a large class of so-called \emph{coorbit spaces}.

A coorbit space $\Co (Y)$ is determined by a
solid, left and right translation-invariant Banach function space $Y \subset L_{\loc}^1 (G)$
of (equivalence classes of a.e.~equal) functions $G \to \CC$;
see \cite{feichtinger1989banach1} for details on the definitions.
Associated to such a space $Y$, it is assumed that there exists a
\emph{control weight} $w : G \to (0,\infty)$ for $Y$, meaning that $w$ satisfies the following:
\begin{enumerate}
  \item[(w1)] $w : G \to [1,\infty)$ is continuous and submultiplicative,
              \vspace*{0.1cm}

  \item[(w2)] $w$ satisfies the symmetry condition
              \begin{equation}
                w(x) = \Delta(x^{-1}) \, w(x^{-1})
                \quad \forall \, x \in G ,
                \label{eq:ControlWeightSymmetry}
              \end{equation}

  \item[(w3)] $w$ is \emph{compatible with $Y$}; that is,
              \begin{equation}
                w(x)
                \geq C \, \max \big\{
                                  \| L_x \|_{Y \to Y}, \,\,
                                  \| L_{x^{-1}} \|_{Y \to Y}, \,\,
                                  \Delta(x^{-1}) \, \| R_{x^{-1}} \|_{Y \to Y}, \,\,
                                  \| R_x \|_{Y \to Y}
                                \big\}
                \label{eq:CoorbitWeightCondition}
              \end{equation}
              and
              \begin{equation}
                Y \ast L_w^1 \subset Y
                \qquad \text{with} \qquad
                \| f \ast g \|_Y \leq \| f \|_Y \cdot \| g \|_{L_w^1} .
                \label{eq:CoorbitConvolutionRelation}
              \end{equation}
\end{enumerate}
Given such a space $Y$ with control weight $w$, it is assumed in
\cite{feichtinger1989banach1,feichtinger1989banach2,groechenig1991describing}
that $0 \neq g \in \Hpi$ is a \emph{better vector}, meaning that $V_g g \in \WLw (G)$.%
\footnote{Such better vectors exist if and only if $(\pi, \Hpi)$ is $w$-integrable,
that is, if and only if a non-zero vector $g \in \Hpi$ with $V_gg \in L^1_w(G)$ exists;
see \cite[Lemma in Section~6.1]{feichtinger1988unified}.}
Denote by $\Reservoir_w := (\TestVectors)^{\neg}$ the \emph{antidual} of the space
\[
  \TestVectors
  := \{ f \in \Hpi \,\,\colon\,\, V_g f \in L_w^1(G) \},
\]
with norm $\| f \|_{\TestVectors} := \| V_g f \|_{L_w^1}$.
One has $\TestVectors \hookrightarrow \Hpi$, and hence $\Hpi \hookrightarrow \Reservoir_w$.
Using the notation $\langle f, \varphi \rangle := f( \varphi)$
for the pairing between $f \in \Reservoir_w$ and $\varphi \in \TestVectors$,
the (extended) matrix coefficients are defined by
\[
  V_g f (x) := \langle f, \pi(x) g \rangle,
  \quad  x \in G,
\]
which extends the definition of $V_g f$ for $f, g \in \Hpi$.
The coorbit space $\Co(Y)$ is defined as
\[
  \Co (Y)
  := \bigl\{ f \in \Reservoir_w \,\,\colon\,\, V_g f \in Y \bigr\}
\]
and equipped with the norm $\| f \|_{\Co (Y)} := \| V_g f \|_Y $.
The spaces $\Co(Y)$, $\TestVectors$ and $ \Reservoir_w$ are independent of the
precise choice of $g$.
Moreover, the space $\Co(Y)$ is complete and independent of the precise
choice of the control weight $w$.
We refer to \cite[Theorems~4.1 and 4.2]{feichtinger1989banach1} for the details.

\begin{remark}
  In \cite{feichtinger1989banach1}, instead of $Y \ast L_w^1 \subset Y$,
  it is actually assumed that $Y \ast (L_w^1)^{\vee} \subset Y$;
  however, due to \eqref{eq:ControlWeightSymmetry}
  we have $\| f^{\vee} \|_{L_w^1} = \| f \|_{L_w^1}$ for each measurable ${f : G \to \CC}$,
  so that both conditions are equivalent.
  Furthermore, the identity ${\| f^{\vee} \|_{L_w^1} = \| f \|_{L_w^1}}$ yields that
  the space $\WRw(G)$ as defined in the present paper
  coincides with the right-sided Wiener amalgam space as used in
  \cite{feichtinger1989banach1,feichtinger1989banach2,groechenig1991describing}.%
  \footnote{In these papers, the norm on $\WRw(G)$ is given by
  $\| f \|_{\WRw(G)} := \| (\maxR f)^{\vee} \|_{L_w^1}$ instead of our convention
  $\| f \|_{\WRw(G)} = \| \maxR f \|_{L_w^1}$; see \cite[Equation~(2.11)]{groechenig1991describing}.}
  Finally, the symmetry condition \eqref{eq:ControlWeightSymmetry} shows because of
  $|V_g g(x)| = |V_g g (x^{-1})|$ and since $\maxL (f^{\vee}) = (\maxR f)^{\vee}$
  that $V_g g \in \WLw(G)$ if and only if $V_g g \in \WRw (G)$.
\end{remark}

Theorem~\ref{thm:frame_main1} can be used to obtain the
following improved discretization result for the coorbit spaces:

\begin{theorem}\label{thm:frame_coorbit}
  Let $w : G \to [1,\infty)$ be continuous and submultiplicative
  and such that \eqref{eq:ControlWeightSymmetry} holds.
  Suppose that $g \in \Hpi \setminus \{ 0 \}$ satisfies $V_g g \in \WLw (G)$.
  Then there exists a compact unit-neighborhood $U \subset G$
  such that for every relatively separated $U$-dense family $\Lambda$ in $G$
  there exist vectors $\{f_\lambda:\lambda \in \Lambda \} \subset \TestVectors \subset \Hpi$
  with the following properties:
  \begin{enumerate}
  \item[(i)] There exists an envelope $\Phi \in \Wstw(G)$ such that
             \begin{align}\label{eq:CoorbitMoleculeCondition}
               \abs{V_g f_\lambda(x)}
               \leq \min \big\{
                           \Phi (\lambda^{-1} x),
                           \Phi(x^{-1} \lambda)
                         \big\}
               \qquad \text{for all} \quad x \in G \text{ and } \lambda \in \Lambda .
             \end{align}

  \item[(ii)] If $Y \subset L_{\loc}^1 (G)$ is any solid, translation-invariant Banach space
              with control weight $w$, then any $f \in \CoY$ satisfies
              \begin{align}\label{eq_bbb}
                f = \sum_{\lambda \in \Lambda}
                      \langle f, \pi(\lambda) g \rangle \, f_{\lambda}
                  = \sum_{\lambda \in \Lambda}
                      \langle f, f_{\lambda} \rangle \, \pi(\lambda) g,
              \end{align}
              where the series converge unconditionally in the weak-$\ast$-topology on $\Reservoir_w$.
              Moreover, the following norms are equivalent:
              \[
                \big\| f \big\|_{\CoY}
                \asymp \big\|
                         \bigl(
                           \langle f, f_{\lambda} \rangle
                         \bigr)_{\lambda \in \Lambda}
                       \big\|_{Y_d (\Lambda)}
                \asymp \big\|
                         \bigl(
                           \langle f, \pi(\lambda) g \rangle
                         \bigr)_{\lambda \in \Lambda}
                       \big\|_{Y_d (\Lambda)},
              \]
              where $Y_d (\Lambda) \subset \mathbb{C}^{\Lambda}$ denotes the sequence associated to $Y$,
              given by
              \[
                Y_d (\Lambda) = \{ c \in \CC^{\Lambda} \,\,\colon\,\, \| c \|_{Y_d} < \infty \},
                \quad \text{where} \quad
                \| c \|_{Y_d} = \bigg\|
                                  \sum_{\lambda \in \Lambda} |c_\lambda| \, \indicator_{\lambda Q}
                                \bigg\|_{Y} .
              \]
  \end{enumerate}
\end{theorem}

\begin{proof}
We consider the space $\RKHS=V_g (\Hpi)$ and verify the hypothesis of Theorem~\ref{thm:frame_main1}.
As seen around Equation~\eqref{eq:CoorbitReproducingKernel}, $\RKHS$ is a RKHS with kernel
$k(x,y) = \| C_\pi g \|_{\Hpi}^{-2} \cdot V_g g (y^{-1} x)$.
The diagonal of $k$ is constant, $k(x,x) = \|C_{\pi} g \|^{-2}_{\Hpi} \, \|g\|^2_{\Hpi}$,
so that (\Kone) holds.
For (\Ktwo), note that
$\overline{V_g [\pi(x) g] (y)} = \langle \pi(y) g , \pi(x) g \rangle_{\Hpi} = V_g g (y^{-1} x)$,
so that the orthogonality relation \eqref{eq:ortho_relations_a} implies
\[
  V_g g (x)
  = \langle g, \pi(x) g \rangle_{\Hpi}
  = \| C_\pi g \|_{\Hpi}^{-2} \cdot \langle V_g g, V_g [\pi(x) g] \rangle_{L^2}
  = \| C_\pi g \|_{\Hpi}^{-2} \cdot (V_g g \ast V_g g) (x)
  .
\]
Since ${V_g g \in \WLw (G) \cap \WRw(G)}$, the convolution relation
\eqref{eq:amalgam_convolution} shows that $V_g g \in \Wstw(G)$, so that
$\Theta := \| C_\pi \, g \|_{\Hpi}^{-2} \, |V_g g| \in \Wstw (G)$
is an envelope for the reproducing kernel $k$ of $\RKHS$.
Finally, (\Kthree) follows from Lemma~\ref{lem:SUCImpliesWUC}, since
\[
  \| k_x - k_y \|_{L^1}
  = \| C_{\pi} g \|_{\Hpi}^{-2} \, \| L_{y^{-1} x} V_gg - V_g g \|_{L^1}
  \to 0,
  \quad \text{as} \quad y^{-1} x \to e,
\]
by the strong continuity of $L_y : L^1 (G) \to L^1 (G)$.

We can now apply Theorem~\ref{thm:frame_main1} to the space $\mathcal{K} = V_g (\Hpi)$
to obtain a compact unit-neighborhood $U \subset G$ such that for every relatively separated
$U$-dense family $\Lambda$ in $G$, there exist vectors $\{f_\lambda:\lambda \in \Lambda \} \subset \Hpi$
such that (i) is satisfied, and such that \eqref{eq_bbb} holds for all $f \in \Hpi$.
Note that property (i) implies $V_g f_\lambda \in L_w^1(G)$ and hence $f_\lambda \in \TestVectors$.

To prove \eqref{eq_bbb} for $f \in \Co(Y)$, we use that
Lemma~\ref{lem:CoefficientReconstructionBounded} implies that the coefficient operators
\({
  \analysis : f \mapsto \big(
                          \langle f, \pi(\lambda) g \rangle
                        \big)_{\lambda \in \Lambda}
}\)
and
\({
  \widetilde{\analysis} : f \mapsto \big(
                                      \langle f, f_{\lambda} \rangle
                                    \big)_{\lambda \in \Lambda}
}\)
map $\CoY$ boundedly into $Y_d (\Lambda)$.
Similarly, the reconstruction maps
\({
  \synthesis :
  (c_{\lambda})_{\lambda \in \Lambda} \mapsto \sum_{\lambda \in \Lambda}
                                                c_{\lambda} \, \pi(\lambda) g
}\)
and
\({
  \widetilde{\synthesis} :
  (c_{\lambda})_{\lambda \in \Lambda} \mapsto \sum_{\lambda \in \Lambda}
                                                c_{\lambda} \, f_{\lambda}
}\)
map $Y_d (\Lambda)$ boundedly into $\Co (Y)$, with the required convergence properties.
The identity
\[
 f = \widetilde{\synthesis} (\analysis f) = \synthesis (\widetilde{\analysis} f)
\]
therefore extends by density from $f \in \Hpi$ to $f \in \CoY$;
see Lemma~\ref{lem:WeakStarContinuity} for the technical details.
This also implies that
\(
  \| f \|_{\Co(Y)}
  \lesssim \| \analysis f \|_{Y_d}
  \lesssim \| f \|_{\Co (Y)}
\)
and similarly ${\| f \|_{\Co(Y)} \asymp \| \widetilde{\analysis} f \|_{Y_d}}$,
which establishes the claimed norm equivalence.
\end{proof}

While \cite{feichtinger1989banach1, groechenig1991describing}
gives part (ii) of Theorem~\ref{thm:frame_coorbit}, the novelty of our result
is that the dual system $(f_\lambda)_{\lambda \in \Lambda}$
can be chosen to satisfy \eqref{eq:CoorbitMoleculeCondition}.

\begin{remark}
  Let $w : G \to [1,\infty)$ be continuous and submultiplicative
  and such that \eqref{eq:ControlWeightSymmetry} holds.
  Suppose there exists $g \in \Hpi \setminus \{ 0 \}$ such that $V_g g \in \WLw (G)$.
  Then, by Part (ii) of Theorem~\ref{thm:frame_main1}, there exists
  a \emph{tight} frame $(f_{\lambda})_{\lambda \in \Lambda}$ for
  $\Hpi$ satisfying the $w$-molecule condition \eqref{eq:CoorbitMoleculeCondition}.
  As in Theorem~\ref{thm:frame_coorbit}, the tight frame expansion
  $f = \sum_{\lambda \in \Lambda} \langle f, f_{\lambda} \rangle f_{\lambda}$ extends to $\Co(Y)$
  for any solid, translation-invariant Banach space $Y \subset L^1_{\loc} (G)$ with control weight $w$.
\end{remark}

In a similar fashion, Theorem~\ref{thm:riesz_main} can be used to obtain coherent Riesz sequences
with biorthogonal system forming a family of molecules:

\begin{theorem}\label{thm:riesz_coorbit}
  Let $w : G \to [1,\infty)$ be continuous and submultiplicative and such that
  \eqref{eq:ControlWeightSymmetry} holds.
  Suppose that $g \in \Hpi \setminus \{ 0 \}$ satisfies $V_g g \in \WLw (G)$.
  Then there exists a compact unit neighborhood $K \subset G$ such that,
  for every  $K$-separated family $\Lambda$ in $G$,
  there exist vectors $\{f_\lambda:\lambda \in \Lambda \} \subset \TestVectors \subset \Hpi$
  with the following properties:
  \begin{enumerate}
  \item[(i)] There exists an envelope $\Phi \in \Wstw(G)$ such that
             \begin{align}\label{eq_aaa}
               \abs{V_g f_\lambda(x)}
               \leq \min \big\{
                           \Phi (\lambda^{-1} x),
                           \Phi(x^{-1} \lambda)
                         \big\}
               \qquad \text{for all} \quad x \in G \text{ and } \lambda \in \Lambda .
             \end{align}

  \item[(ii)] If $Y \subset L_{\loc}^1 (G)$ is any solid, translation-invariant Banach space
              with control weight $w$, then for any $(c_{\lambda})_{\lambda \in \Lambda} \in Y_d (\Lambda)$,
              the vector
              \[
                f = \sum_{\lambda \in \Lambda} c_{\lambda} f_{\lambda}
              \]
              belongs to $\CoY$, satisfies $\| f \|_{\CoY} \lesssim \| c \|_{Y_d (\Lambda)}$,
              and solves the moment problem
              \[
                \langle f, \pi (\lambda) g \rangle = c_{\lambda},
                \quad \lambda \in \Lambda.
              \]
  \end{enumerate}
\end{theorem}

\begin{proof}
As in the proof of Theorem~\ref{thm:frame_coorbit}, we consider the space $\mathcal{K} = V_g (\Hpi)$;
in that proof we already showed that $\RKHS$ is a RKHS satisfying conditions (\Kone), (\Ktwo),
and (\Kthree).
An application of Theorem~\ref{thm:riesz_main} therefore provides
a compact unit-neighborhood $K \subset G$ such that for every $K$-separated family $\Lambda$
in $G$, there exist vectors $\{f_\lambda:\lambda \in \Lambda \} \subset \Hpi$ satisfying (i)
and such that
\begin{equation}\label{eq:left-right_riesz}
  c
  = \analysis (\widetilde{\synthesis} c)
  = \widetilde{\analysis} (\synthesis c) ,
  \qquad c \in \ell^2(\Lambda),
\end{equation}
with the operators $\analysis, \synthesis, \widetilde{\analysis}, \widetilde{\synthesis}$
from the proof of Theorem~\ref{thm:frame_coorbit}.
Again, property (i) implies that $f_\lambda \in \TestVectors$ for all $\lambda \in \Lambda$.
Furthermore, one can show that the identity \eqref{eq:left-right_riesz}
extends from $\ell^2(\Lambda)$ to all of $Y_d$;
see Lemma~\ref{lem:WeakStarContinuity} for the technical details.
The estimate ${\| f \|_{\Co(Y)} = \| \widetilde{\synthesis} c \|_{\Co(Y)} \lesssim \| c \|_{Y_d}}$
follows from the boundedness of $\widetilde{\synthesis} : Y_d (\Lambda) \to \Co(Y)$;
see Lemma~\ref{lem:CoefficientReconstructionBounded}.
\end{proof}

Part (ii) of Theorem~\ref{thm:riesz_coorbit} is contained in \cite{feichtinger1989banach2};
the novelty of the present result is that the moment problem has fundamental solutions $f_\lambda$
satisfying \eqref{eq_aaa}; see also Section~\ref{sec_int}.

\begin{remark}[Extensions and generalizations.]\label{rem_pro}
Our discretization results can also be applied
to representations $(\pi, \Hpi)$ that are only square-integrable \emph{modulo a central subgroup},
or to projective unitary representations as considered in
\cite{christensen1996atomic, christensen2019coorbits}.
In this setting, the phase function $\Gamma : G \times G \to \mathbb{T}$
appearing in Lemma~\ref{lem:SUCImpliesWUC} is particularly convenient to verify the kernel condition
(\Kthree).
In addition, since we only require smoothness conditions on the reproducing kernel,
our results can be applied to possibly \emph{reducible} representations $(\pi, \Hpi)$
and their associated coorbit spaces; see \cite{christensen2009examples, christensen2011coorbit}
for a general framework and examples.
\end{remark}

\subsection{The affine group and molecular decompositions}

Consider the affine group $\mathbb{R} \rtimes \mathbb{R}^\ast$,
with group law $(b, a)(b',a') = (b + a b', a a')$.
The representation $(\pi, L^2 (\mathbb{R}))$ of $\mathbb{R} \rtimes \mathbb{R}^{\ast}$,
given by
\[
  \pi(b,a) f = |a|^{-1/2} \cdot f \bigg( \frac{\cdot - b}{a} \bigg),
  \quad (b,a) \in \mathbb{R} \rtimes \mathbb{R}^\ast,
\]
is irreducible and square-integrable.
Moreover, it is integrable, meaning that there exists $\psi \in L^2 (\mathbb{R}) \setminus \{ 0 \}$
such that $V_{\psi} \psi \in L^1 (\mathbb{R} \rtimes \mathbb{R}^\ast)$,
where $V_{\psi} \psi$ is defined as in Section~\ref{sec:CoorbitSpaces}.

By choosing a particular index set, we can use Theorem~\ref{thm:AbstractCanonicalDualUniformityCondition}
to obtain the following result on wavelet frames and their \emph{canonical} dual frames.

\begin{theorem}\label{thm:affine_molecules}
  Let $w : \mathbb{R} \rtimes \mathbb{R}^\ast \to (0, \infty)$ be an admissible weight.
  Suppose that $\psi \in L^2 (\mathbb{R}) \setminus \{0\}$ satisfies
  $V_{\psi} \psi \in \Wstw (\mathbb{R} \rtimes \mathbb{R}^\ast)$.
  Then there exist $a > 1$ and $b > 0$ such that the vectors
  \[
    \psi_{j,k,t}
    := \pi(a^j t b k, a^j t) \psi
    = a^{-j/2} \cdot \psi \bigl(a^{-j} t \, (\cdot) - b k\bigr),
    \quad j,k \in \mathbb{Z}, t \in \{ \pm 1 \}
  \]
  form a frame
  for $L^2 (\mathbb{R})$ whose \emph{canonical} dual frame
  \(
    (\rho_{j,k,t} )_{j,k \in \mathbb{Z}, t \in \{ \pm 1 \}}
    \!=\! (\frameop^{-1} \psi_{j,k,t})_{j,k \in \mathbb{Z}, t \in \{ \pm 1 \}}
  \)
  in $L^2 (\mathbb{R})$ is a system of \emph{time-scale $w$-molecules}, i.e., there exists
  $\Phi \in \Wstw (\mathbb{R} \rtimes \mathbb{R}^\ast)$ such that
  \[
    \big| V_{\psi} \, \rho_{j,k,t} (b', a') \big|
    \leq \min
         \big\{
           \Phi \bigl(a^{-j} t \, b' - b k, \,\, a^{-j} t \, a'\bigr), \quad
           \Phi \bigl((a^j t b k - b') / a',\,\, a^j t / a' \bigr)
         \big\}
  \]
  for all $(b', a') \in \mathbb{R} \rtimes \mathbb{R}^\ast$, $j,k \in \mathbb{Z}$,
  and $t \in \{ \pm 1 \}$.
\end{theorem}

\begin{proof}
  As in the proof of Theorem~\ref{thm:frame_coorbit}, the image space $V_{\psi} (L^2 (\mathbb{R}))$
  is a RKHS satisfying conditions (\Kone), (\Ktwo), and (\Kthree).
  To apply Theorem~\ref{thm:AbstractCanonicalDualUniformityCondition},
  fix $a > 1, b > 0$, define the compact unit neighborhood
  $U  = [-\frac{b}{2}, \frac{b}{2}] \times [a^{-1/2}, a^{1/2}]$
  and let ${\Lambda =  \{ (a^j t b k, a^j t) : j,k \in \mathbb{Z} , t \in \{ \pm 1 \} \}}$.
  Then $\Lambda$ is easily seen to be relatively separated, $U$-dense and to have uniformity
  $\uniformity(\Lambda; U) = 1$.
  Note that if $U_0$ is an arbitrary unit neighborhood, then $U \subset U_0$ for suitable
  $a > 1$ and $b > 0$, and therefore $\uniformity(\Lambda; U_0) \leq \uniformity(\Lambda; U) = 1$.
  Overall, we thus see that an application of Theorem~\ref{thm:AbstractCanonicalDualUniformityCondition}
  yields $a > 1, b > 0$ such that $V_\psi [\pi (\Lambda) \psi]$ is a frame for $V_\psi (L^2(\R))$
  whose canonical dual frame consists of molecules.
  Since $V_\psi : L^2(\R) \to V_\psi [L^2(\R)]$ is isometric (up to a constant factor),
  this implies the claim.
\end{proof}

In contrast to Theorem~\ref{thm:frame_coorbit}, the dual frame in Theorem~\ref{thm:affine_molecules}
can be chosen to be the \emph{canonical} dual frame.
As in Theorem~\ref{thm:frame_coorbit}, the molecule property of the dual frame
implies that the \emph{canonical frame expansions}
\[
  f
  = \sum_{j,k \in \mathbb{Z}, t \in \{ \pm 1 \}}
      \langle f, \psi_{j,k,t} \rangle \rho_{j,k,t}
  = \sum_{j,k \in \mathbb{Z}, t \in \{ \pm 1 \}}
      \langle f, \rho_{j,k,t} \rangle \psi_{j,k,t}
\]
can be extended to coorbit spaces associated to the affine group,
such as the Besov-Triebel-Lizorkin spaces.

Theorem~\ref{thm:affine_molecules} comes close to \cite[Theorem~1.5]{memoir},
where the notion of time-scale molecule is described more concretely in terms of differentiability,
rate of decay, and number of vanishing moments.
See \cite[Section~4.2]{groechenig2009molecules} and \cite{Holschneider95}
for a comparison between different variants of the notion of time-scale molecule.

\subsection{Weighted spaces of analytic functions}

Theorems~\ref{thm:frame_main1} and \ref{thm:riesz_main} can also be used to produce
sampling or interpolating sets for certain Bergman spaces of analytic functions.
Example~\ref{ex_fock} discusses explicitly how Bargmann-Fock spaces of entire functions
with respect to general weights fit within the framework of this article;
see also \cite[Section~7.5]{jfa19} for more details on sampling and interpolation
in Fock spaces from the perspective of RKHS and molecules.

\appendix

\section{Postponed proofs}
\label{sec:PostponedProofs}

\subsection{Discrete sets}
\label{sub:DiscreteSetsAppendix}

\begin{lemma}\label{lem:nondiscrete_meanvalue}
  Let $G$ be a non-discrete $\sigma$-compact locally compact group with Haar measure $\mu_G$.
  For any Borel set $M \subset G$ and any $\delta \in [0, \mu_G (M)]$,
  there is a Borel set $M_{\delta} \subset M$ satisfying $\mu_G (M_{\delta}) = \delta$.
\end{lemma}

\begin{proof}
  Since $G$ is non-discrete, it follows by \cite[Proposition~1.4.4]{deitmar2014principles}
  that $\mu_G (\{e\}) = 0$.
  This implies that $\mu_G$ is atom free, since if $A \subset G$ is an atom of $\mu_G$,
  then (since $\mu_G$ is $\sigma$-finite), $0 < \eps := \mu(A) < \infty$,
  and by outer regularity of $\mu_G$ and since $\mu_G (\{ e \}) = 0$,
  there is an open unit neighborhood $U \subset G$ satisfying $\mu_G (U) \leq \eps/2$.
  Since $G$ is $\sigma$-compact, we have $G = \bigcup_{n=1}^\infty x_n U$
  for suitable $(x_n)_{n \in \N} \subset G$, and hence
  $0 < \mu_G(A \cap x_n U) \leq \mu_G (x_n U) \leq \eps/2 < \mu(A)$ for some $n \in \N$,
  which contradicts the fact that $A$ is an atom of $\mu_G$.

  Now, Sierpinski's theorem implies the claim;
  see \cite[Theorem~10.52]{AliprantisBorderHitchhiker}.
\end{proof}

\subsection{Amalgam spaces}
\label{sub:AmalgamAppendix}

We first provide a proof that $\WLw$ embeds into $L^\infty$.

\begin{lemma}\label{lem:LInftyEmbedding}
  There is a constant $C > 0$ satisfying
  $\| f \|_{L^\infty} \leq C \, \| f \|_{\WL} \leq C \, \| f \|_{\WLw}$
  for each admissible weight $w$ and measurable $f : G \to \CC$.
\end{lemma}

\begin{proof}
  Since $w \geq 1$, it suffices to prove $\| f \|_{L^\infty} \leq C \, \| f \|_{\WL}$.
  Choose a symmetric open unit neighborhood $P \subset G$ satisfying $P P \subset Q$.
  Since $G$ is $\sigma$-compact, there is a countable family $(x_n)_{n \in \N} \subset G$
  satisfying $G = \bigcup_{n=1}^\infty x_n P$.
  For $p \in P$, we have $x_n P \subset x_n p Q$, so that
  $\| f \|_{L^\infty (x_n P)} \leq \| f \|_{L^\infty(x_n p Q)} = \maxL f (x_n p)$.
  Averaging this over $p \in P$, we get
  \(
    \| f \|_{L^\infty (x_n P)}
    \leq \frac{1}{\mu_G (P)} \, \int_P \maxL f (x_n p) \, d \mu_G (p)
    \leq \frac{1}{\mu_G (P)} \, \| f \|_{\WL} .
  \)
  Since this holds for all $n \in \N$, and $G = \bigcup_{n=1}^\infty x_n P$,
  this implies the claim for $C = [\mu_G (P)]^{-1}$.
\end{proof}

We close this subsection with a proof of Lemma~\ref{lem:SynthesisOperatorQuantitativeBound}.

\begin{proof}[Proof of Lemma~\ref{lem:SynthesisOperatorQuantitativeBound}]

To prove \eqref{eq:StandardEstimateOne}, fix $x,y \in G$ and $\lambda \in \Lambda$.
For arbitrary $z \in \lambda Q$, we then have $\lambda^{-1} z \in Q$
and $z^{-1} \lambda \in Q^{-1} = Q$.
Therefore, $\lambda^{-1} x = \lambda^{-1} z z^{-1} x \in Q z^{-1} x$
and $y^{-1} \lambda = y^{-1} z z^{-1} \lambda \in y^{-1} z Q$.
Since $Q$ is open and the envelopes $\Phi,\Psi$ are continuous,
this implies $\Phi(\lambda^{-1} x) \leq \maxR \Phi (z^{-1} x)$,
as well as ${\Psi (y^{-1} \lambda) \leq \maxL \Psi (y^{-1} z)}$.
Taking the product of these two inequalities, averaging over $z \in \lambda Q$,
and summing over $\lambda \in \Lambda$, we thus see
\begin{align*}
  \sum_{\lambda \in \Lambda}
    \Phi(\lambda^{-1} x) \, \Psi(y^{-1} \lambda)
  & \leq \frac{1}{\mu_G (Q)}
         \sum_{\lambda \in \Lambda}
           \int_{\lambda Q}
             \maxL \Psi (y^{-1} z) \, \maxR \Phi (z^{-1} x)
           \, d \mu_G (z) \\
  & \leq \frac{\rel (\Lambda)}{\mu_G (Q)}
         \int_G
           \maxL \Psi(y^{-1} z) \, \maxR \Phi(z^{-1} x)
         \, d \mu_G (z) \\
  & =    \frac{\rel (\Lambda)}{\mu_G (Q)}
         (\maxL \Psi \ast \maxR \Phi) (y^{-1} x),
\end{align*}
as claimed.
Here, the second step is justified by the monotone convergence theorem
and Equation~\eqref{eq:RelativeSeparation}, and the final step used the change of variables
$t = y^{-1} z$. A similar argument proves \eqref{eq:StandardEstimateTwo}.

Lastly, if $\Theta$ is continuous, also $|\Theta^{\vee}|$ is continuous,
so that Equation~\eqref{eq:StandardEstimateTwo} applied to $\Psi = |\Theta^{\vee}|$ shows
\(
  \sum_{\lambda \in \Lambda}
    |\Theta(\lambda^{-1} x)|
  = \sum_{\lambda \in \Lambda}
      |\Theta^{\vee} (x^{-1} \lambda)|
  \leq C \, \| \Theta^{\vee} \|_{\WL}
\)
for $C := \frac{\rel(\Lambda)}{\mu_G(Q)}$.
Let ${c = (c_\lambda)_{\lambda \in \Lambda} \in \ell^2(\Lambda)}$.
By the Cauchy-Schwarz inequality and the preceding estimate,
it follows as required that
\begin{align*}
  \| D_{\Theta,\Lambda} \, c \|_{L^2}^2
  & \leq \int_{G}
         \bigg(
           \sum_{\lambda \in \Lambda}
             |c_\lambda| \, |\Theta(\lambda^{-1} x)|^{1/2} \, |\Theta(\lambda^{-1} x)|^{1/2}
         \bigg)^2
         \, d \mu_G (x) \\
  & \leq \int_G
           \bigg(
             \sum_{\lambda \in \Lambda}
               |c_\lambda|^2 \, |\Theta(\lambda^{-1} x)|
           \bigg)
           \bigg(
             \sum_{\lambda \in \Lambda}
               |\Theta(\lambda^{-1} x)|
           \bigg)
         \, d \mu_G (x) \\
  & \leq C \,
         \| \Theta^\vee \|_{\WL}
         \sum_{\lambda \in \Lambda}
                |c_\lambda|^2 \!
                \int_G \! |\Theta(\lambda^{-1} x)| \, d \mu_G (x)% \\
    =    C \,
         \| \Theta^\vee \|_{\WL} \,
         \| \Theta \|_{L^1} \,
         \| c \|_{\ell^2(\Lambda)}^2 .
\end{align*}
This completes the proof.
\end{proof}

\subsection{Localized integral kernels}
\label{sub:LocalizedIntegralKernelsAppendix}

In this subsection, we provide a proof for Lemma~\ref{lem:AdmissibleKernelProperties}.

\begin{proof}[Proof of Lemma~\ref{lem:AdmissibleKernelProperties}]
  Let $\Phi \in \WstCw (G)$ be an envelope for $H$.

  \medskip{}

  We first show that $T_H f (x) \in \CC$ is well-defined for $x \in G$ and $f \in L^p (G)$
  and that ${T_H : L^p (G) \to L^p(G)}$ is bounded.
  Since $\Phi$ is an envelope for $H$, we have $|H(x,\cdot)| \leq L_x \Phi$.
  Since $\Phi \in \WstCw (G) \subset L^1(G) \cap L^\infty(G) \subset L^{p'}(G)$,
  it follows that $L_x \Phi \in L^{p'}(G)$, and hence $H(x,\cdot) \in L^{p'}(G)$,
  since $H$ is measurable.
  Therefore, ${T_H (x) = \langle H(x,\cdot), f \rangle_{L^{p'}, L^p} \in \CC}$ is well-defined,
  with absolute convergence of the defining integral.
  Finally, the boundedness of ${T_H : L^p (G) \to L^p (G)}$
  is an easy consequence of Schur's test; see \cite[Theorem~6.18]{FollandRA}.

  \medskip{}

 (i) Let $f,g \in L^2(G)$ be arbitrary. Note that
  \begin{equation}
    \begin{split}
      & \int_G
          \int_G
            |H(x,y)| \, |f(y)| \, |g(x)|
          \, d \mu_G (y)
        \, d \mu_G (x) \\
      & \leq \int_{G \times G}
               [\Phi(x^{-1} y)]^{1/2} \, [\Phi(y^{-1}x)]^{1/2} \, |f(y)| \, |g(x)|
             \, d (\mu_G \otimes \mu_G)(x,y) \\
      & \leq \Big(
               \int_{G \times G} \!\!\!
                 \Phi (y^{-1} x) |f(y)|^2
               \, d (\mu_G \! \otimes \mu_G)(x,y)
             \Big)^{\frac{1}{2}}
             \Big(
               \int_{G \times G} \!\!\!
                 \Phi (x^{-1} y) |g(x)|^2
               \, d (\mu_G \! \otimes \mu_G)(x,y)
             \Big)^{\frac{1}{2}} \\
      & =    \Big(
               \int_G
                 |f(y)|^2 \cdot \| L_y \Phi \|_{L^1}
               \, d \mu_G (y)
             \Big)^{1/2}
             \Big(
               \int_G
                 |g(x)|^2 \cdot \| L_x \Phi \|_{L^1}
               \, d \mu_G (x)
             \Big)^{1/2}
        < \infty .
    \end{split}
    \label{eq:FubiniL2Justification}
  \end{equation}
  As shown above, the map $T_H$ is well-defined from $L^2(G)$ into $L^2 (G)$.
  Let $f,g \in L^2(G)$ with $g \perp \RKHS$ be arbitrary.
  The above calculation justifies an application of Fubini's theorem, and hence
  \begin{align*}
    \langle T_H f, g \rangle
    & = \int_G
          \overline{g(x)}
          \int_G
            H(x,y) \, f(y)
          \, d \mu_G (y)
        \, d \mu_G (x)
      = \int_G
          f(y) \cdot \langle H(\cdot,y), g \rangle
        \, d \mu_G (y)
      = 0,
  \end{align*}
  where the last equality follows since $H(\cdot,y) \in \RKHS$ and $g \perp \RKHS$.
  Thus, $T_H f \in \RKHS$.

  \medskip{}

  \noindent
  (ii) Since $H(\cdot, t) \in \RKHS$ and $\overline{H(x,\cdot)} \in \RKHS$,
  we see $\langle H(\cdot,t) , k_x \rangle \!=\! H(x,t)$
  and ${\langle \overline{H(x,\cdot)}, k_y \rangle \!=\! \overline{H(x,y)}}$.
  Thus, an application of Fubini's theorem (which is justified by \eqref{eq:FubiniL2Justification})
  gives
  \begin{align*}
    \langle T_H k_y, k_x \rangle
    & = \int_G
          k_y (t)
          \int_G
            H(z,t) \overline{k_x (z)}
          \, d \mu_G (z)
        \, d \mu_G (t)
     = \int_G
          k_y (t) \,
          \langle H(\cdot,t), k_x \rangle
        \, d \mu_G (t) \\
     &=
        \int_G
          k_y (t) \cdot H(x,t)
        \, d \mu_G (t)
     = \overline{
          \langle \overline{H(x,\cdot)}, k_y \rangle
        }
        =
        \overline{
          \overline{H(x,y)}
        }
      = H(x,y) .
  \end{align*}

  \noindent
  (iii) This is an immediate consequence of the definitions.

  \medskip{}

  \noindent
  (iv) Let $\Phi' \in \Wstw(G)$ be an envelope for $L$.
  Then
  \begin{align*}
    |H \odot L (x,y)|
    &      \leq \int_G
             \Phi (z^{-1} x) \, \Phi'(y^{-1} z)
           \, d \mu_G (z)
           =    \int_G
                  \Phi'(t) \, \Phi(t^{-1} y^{-1} x)
                \, d \mu_G (t)
      = \Phi' \ast \Phi (y^{-1} x) .
  \end{align*}
  Similarly, it follows that
  \(
    |H \odot L (y,x)|
          \leq \int_G
             \Phi (y^{-1} z) \, \Phi'(z^{-1} x)
           \, d \mu_G (z)
      = \Phi \ast \Phi' (y^{-1} x).
  \)
  Thus $\Psi := \Phi' \ast \Phi + \Phi \ast \Phi'$ is an envelope
  for $H \odot L$, and $\Psi \in \WstCw(G)$ by \eqref{eq:amalgam_convolution}.
  The above calculation also shows that $H \odot L (x,y) \in \CC$ is well-defined
  for all $x,y \in G$.

  Note that $H \odot L (\cdot,y) = T_H [L(\cdot,y)] \in \RKHS$ by Part~(i),
  since $L(\cdot,y) \in \RKHS \subset L^2(G)$, and $T_H : L^2(G) \to \RKHS$.
  Moreover,
  \[
    H \odot L (x,y)
    = \int_G
        \overline{\widetilde{L} (y,z)} \,\,
        \overline{ \overline{H} (x,z)}
      \, d \mu_G (z)
    = \overline{\big( T_{\widetilde{L}} [\overline{H(x,\cdot)}] \big)(y)} ,
  \]
  so that $\overline{H \odot L (x, \cdot)} = T_{\widetilde{L}} [\overline{H(x,\cdot)}] \in \RKHS$,
  since $\overline{H(x,\cdot)} \in \RKHS \subset L^2(G)$ and $T_{\widetilde{L}} : L^2(G) \to \RKHS$
  by Parts~(i) and (iii).
  Overall, this shows that $H \odot L$ is $w$-localized in $\RKHS$.

  \medskip{}

  To show that $T_H \circ T_L = T_{H \odot L}$,
  let $p \in \{1,\infty\}$.
  Then Fubini's theorem shows that
  \begin{align*}
    [T_H (T_L f)](x)
    & = \int_G
          H(x,y)
          \int_G
            L(y,z) \, f(z)
          \, d\mu_G (z)
        \, d \mu_G(y) \\
    & = \int_G
          f(z)
          \int_G
            H(x,y) \, L(y,z)
          \, d \mu_G (y)
        \, d \mu_G (z) \\
    & = \int_G
          f(z) \cdot (H \odot L)(x,z)
        \, d \mu_G (z)
      = T_{H \odot L} f (x)
  \end{align*}
  for $f \in L^p(G)$.
  Applying Fubini's theorem is justified in case of $p = \infty$ since
  \begin{align*}
    & \int_G
        |f(z)|
        \int_G
          |H(x,y)| \, |L(y,z)|
        \, d \mu_G (y)
      \, d \mu_G (z) \\
    & \leq \| f \|_{L^\infty}
           \int_G
             \Phi(x^{-1} y)
             \int_G
               \Phi'(y^{-1} z)
             \, d \mu_G (z)
           \, d\mu_G (y)
      =    \| f \|_{L^\infty} \, \| \Phi \|_{L^1} \, \| \Phi' \|_{L^1}
      < \infty ,
  \end{align*}
  and in case of $p = 1$ since
  \begin{align*}
    & \int_G
        |f(z)|
        \int_G
          |H(x,y)| \, |L(y,z)|
        \, d \mu_G (y)
      \, d \mu_G (z)
       \leq \| \Phi' \|_{L^\infty} \, \| \Phi \|_{L^1} \, \| f \|_{L^1}.
  \end{align*}
  The case of general $p \in (1,\infty)$ follows since $L^p(G) \subset L^1(G) + L^\infty(G)$.
  Finally, a direct calculation gives
  \[
    \widetilde{H \odot L} (x,y)
    = \overline{
        \int_G
          H(y,z) \, L (z,x)
        \, d \mu_G (z)
      }
    = \int_G
        \widetilde{L} (x,z) \, \widetilde{H}(z,y)
      \, d \mu_G (z)
    = \widetilde{L} \odot \widetilde{H} (x,y)
  \]
  for all $x,y \in G$.
  This completes the proof.
\end{proof}

\subsection{Convolution-dominated matrices}
\label{sub:convolutionmatrix_appendix}

\begin{proof}[Proof of Proposition \ref{prop:GoodMatrixSpaceReasonable}]
(i) The subadditivity and absolute homogeneity of $\| \cdot \|_{\goodMatrices}$
are immediate.
For the positive definiteness of $\| \cdot \|_{\goodMatrices}$, note that if
\(
  \goodMatrices(\Gamma,\Lambda)
  \ni M
  = (M_{\lambda,\gamma})_{\lambda \in \Lambda, \gamma \in \Gamma}
  \dominated \Theta,
\)
then
\(
  |M_{\lambda,\gamma}|
  \leq \Theta (\lambda^{-1} \gamma)
  \leq \| \Theta \|_{L^\infty}
  \lesssim \| \Theta \|_{\WstCw}
\)
by Lemma~\ref{lem:LInftyEmbedding}.
Thus,
\(
  |M_{\lambda,\gamma}| \lesssim \| M \|_{\goodMatrices}
\)
for all $\lambda \in \Lambda$ and $\gamma \in \Gamma$,
showing that if $\| M \| = 0$, then $M = 0$.

For the completeness, it suffices to show that if $(M^{(n)})_{n \in \N}$
satisfies $\sum_{n=1}^\infty \| M^{(n)} \|_{\goodMatrices} < \infty$, then the series
$\sum_{n=1}^\infty M^{(n)}$ is norm convergent in $\goodMatrices(\Gamma,\Lambda)$.
By the argument in the last paragraph, the series defining
$M_{\lambda,\gamma} := \sum_{n=1}^\infty M^{(n)}_{\lambda,\gamma}$
converges (absolutely) for all $\gamma \in \Gamma$ and $\lambda \in \Lambda$.
We claim that
\({
  M
  := (M_{\lambda,\gamma})_{\lambda \in \Lambda, \gamma \in \Gamma}
  \in \goodMatrices(\Gamma,\Lambda)
}\)
and that $M = \sum_{n=1}^\infty M^{(n)}$ with norm convergence in
$\goodMatrices(\Gamma, \Lambda)$.

For each $n \in \N$, choose a non-negative envelope $\Theta_n \in \strongWiener(G)$
satisfying $M^{(n)} \dominated \Theta_n$ and
${\| \Theta_n \|_{\strongWiener} \leq 2 \, \| M^{(n)} \|_{\goodMatrices}}$.
Define $\Phi_N := \sum_{n=N+1}^\infty \Theta_n$ for $N \in \N_0$.
Then $\Phi_N \in \strongWiener(G)$ since $\strongWiener(G)$ is complete and
\(
  \sum_{n=1}^\infty \| \Theta_n \|_{\strongWiener}
  \leq 2 \sum_{n=1}^\infty \| M^{(n)} \|_{\goodMatrices}
  < \infty.
\)
A direct calculation shows
\begin{align*}
  \Big|
    \Big( M - \sum_{n=1}^N M^{(n)} \Big)_{\lambda, \gamma}
  \Big|
  \leq \sum_{n=N+1}^\infty
         \min \big\{ \Theta_n (\gamma^{-1} \lambda), \Theta_n (\lambda^{-1} \gamma) \big\}
  \leq \min \big\{ \Phi_N (\gamma^{-1} \lambda), \Phi_N (\lambda^{-1} \gamma) \big\} ,
\end{align*}
and thus $M - \sum_{n=1}^N M^{(n)} \in \goodMatrices(\Gamma,\Lambda)$, with
\[
  \Big\|
   M - \sum_{n=1}^N M^{(n)}
  \Big\|_{\goodMatrices}
  \leq \| \Phi_N \|_{\strongWiener}
  \leq \sum_{n=N+1}^\infty
         \| \Theta_n \|_{\strongWiener}
  \to 0
\]
as $N \to \infty$.
This completes the proof of Part~(i).

\medskip{}

(ii) Let $\Phi \in \strongWiener(G)$ be an envelope for
\(
  M
  = (M_{\lambda,\gamma})_{\lambda \in\Lambda,\gamma \in \Gamma}
  \in \goodMatrices(\Gamma,\Lambda)
\).
Thanks to Equation~\eqref{eq:StandardEstimateTwo}, we see that
\(
  \sum_{\lambda \in \Lambda}
    |M_{\lambda, \gamma}|
  \leq \sum_{\lambda \in \Lambda}
         \Phi(\gamma^{-1} \lambda)
  \leq \frac{\rel(\Lambda)}{\mu_G(Q)} \, \| \Phi \|_{\Wstw}
\)
and similarly
\(
  \sum_{\gamma \in \Gamma}
    |M_{\lambda, \gamma}|
  \leq \sum_{\gamma \in \Gamma}
         \Phi(\lambda^{-1} \gamma)
  \leq \frac{\rel(\Gamma)}{\mu_G(Q)} \, \| \Phi \|_{\Wstw} .
\)
Since this holds for all envelopes $\Phi \in \strongWiener(G)$ for $M$,
this proves \eqref{eq:GoodMatricesAreSchur}.

\medskip{}

(iii) This follows directly by combining Schur's test (see \cite[Theorem~6.18]{FollandRA})
with Part~(ii).

\medskip{}

(iv) Choose non-negative $\Phi, \Phi' \in \controlFunctionSpace (G)$ which satisfy
$M = (M_{\omega,\gamma})_{\omega \in \Omega, \gamma \in \Gamma} \dominated \Phi$
and ${N = (N_{\gamma, \lambda})_{\gamma \in \Gamma, \lambda \in \Lambda} \dominated \Phi'}$.
Then, Equation~\eqref{eq:StandardEstimateOne} shows for arbitrary
$\omega \in \Omega$ and $\lambda \in \Lambda$ that
\[
  |(M \, N)_{\omega, \lambda}|
  \leq \sum_{\gamma \in \Gamma}
         |M_{\omega, \gamma}| \, |N_{\gamma, \lambda}|
  \leq \sum_{\gamma \in \Gamma}
         \Phi(\gamma^{-1} \omega) \, \Phi' (\lambda^{-1} \gamma)
  \leq \frac{\rel(\Gamma)}{\mu_G (Q)}
       (\maxL \Phi' \ast \maxR \Phi) (\lambda^{-1} \omega)
\]
and similarly
\[
  |(M \, N)_{\omega, \lambda}|
  \leq \sum_{\gamma \in \Gamma}
         |M_{\omega, \gamma}| \, |N_{\gamma, \lambda}|
  \leq \sum_{\gamma \in \Gamma}
         \Phi'(\gamma^{-1} \lambda) \Phi(\omega^{-1} \gamma)
  \leq \frac{\rel(\Gamma)}{\mu_G (Q)}
       (\maxL \Phi \ast \maxR \Phi') (\omega^{-1} \lambda) .
\]
Let
\(
  C := \frac{\rel(\Gamma)}{\mu_G(Q)}
\)
and
\(
  \Psi := C \cdot \big[ (\maxL \Phi') \ast (\maxR \Phi) \big]
          + C \cdot \big[ (\maxL \Phi) \ast (\maxR \Phi') \big].
\)
The above calculations show that
\(
  |(MN)_{\omega,\lambda}|
  \leq \min \{ \Psi(\lambda^{-1} \omega), \Psi(\omega^{-1} \lambda) \}.
\)
Equation~\eqref{eq:amalgam_convolution} shows that $\Psi \in \strongWiener(G)$, with
\(
  \| \Psi \|_{\strongWiener}
  \leq 2C \cdot \| \Phi' \|_{\strongWiener} \cdot \| \Phi \|_{\strongWiener}.
\)
Therefore, it follows that $M \, N \dominated \Psi$
and ${M \, N \in \goodMatrices(\Lambda,\Omega)}$, with
\(
  \| MN\|_{\goodMatrices}
  \leq \| \Psi \|_{\strongWiener}.
\)
Since $\Phi, \Phi' \in \strongWiener(G)$ with $M \dominated \Phi$
and $N \dominated \Phi'$ were chosen arbitrarily, the conclusion follows.
\end{proof}

\subsection{Coorbit theory miscellany}%
\label{sub:CoorbitDensityArgument}

We give an alternative proof of  \cite[Lemma~3.4]{groechenig2009molecules}
regarding the boundedness of the coefficient and reconstruction operators
associated to systems of molecules.
We do this to keep the paper self-cointained and
since the assumptions in \cite{groechenig2009molecules} concerning the space $Y$
are slightly different from this paper.

\begin{lemma}\label{lem:CoefficientReconstructionBounded}
  Let ${w : G \to [1,\infty)}$ be continuous and submultiplicative
  satisfying \eqref{eq:ControlWeightSymmetry},
  and let $Y \subset L_{\loc}^1 (G)$ be a solid translation-invariant Banach space
  satisfying \eqref{eq:CoorbitConvolutionRelation}.
  Let $g \in \Hpi \setminus \{ 0 \}$ satisfy ${V_g g \in \WLw (G)}$,
  let $\Lambda$ be relatively separated in $G$, and let $(f_\lambda)_{\lambda \in \Lambda} \subset \Hpi$
  satisfy \eqref{eq:CoorbitMoleculeCondition} for some $\Phi \in \Wstw (G)$.
  Finally, assume that $Y_d(\Lambda) \hookrightarrow \ell_{1/w}^\infty (\Lambda)$.

  Then the operators
  \[
    \analysis : \quad
    \Co(Y) \to Y_d(\Lambda), \quad
    f \mapsto \bigl( \langle f, f_\lambda \rangle \bigr)_{\lambda \in \Lambda}
  \]
  and
  \[
    \synthesis : \quad
    Y_d (\Lambda) \to \Co(Y), \quad
    c = (c_\lambda)_{\lambda \in \Lambda} \mapsto \sum_{\lambda \in \Lambda} c_\lambda \, f_\lambda
  \]
  are well-defined and bounded.
  The series defining $\synthesis c$ converges unconditionally
  with respect to the weak-$\ast$-topology on $\Reservoir_w$.
\end{lemma}
\begin{proof}
We split the proof into four steps.

  \textbf{Step 1:} Let $\Phi : G \to [0,\infty)$ be continuous.
  Then,since $Q$ is open and $\Phi$ is continuous,
  we have $\Phi(x) = \Phi(q q^{-1} x) \leq (M_Q^R \Phi)(q^{-1} x)$ for all $q \in Q$.
  This shows that
  \begin{equation} \label{eq:PhiDominatedByConvolution}
    \mu_G (Q) \cdot \Phi(x)
    \leq \int_G \indicator_Q (q) \cdot (M_Q^R \Phi)(q^{-1} x) \, d \mu_G (q)
    =    (\indicator_Q \ast M_Q^R \Phi) (x) .
  \end{equation}

  \medskip{}

  \textbf{Step 2:}
  Define $F_\lambda := V_g f_\lambda$ and note that \eqref{eq:CoorbitMoleculeCondition}
  and \eqref{eq:PhiDominatedByConvolution} imply
  \[
    \mu_G (Q) \cdot |F_\lambda|
    \leq \mu_G (Q) \cdot L_\lambda \Phi
    \leq L_\lambda (\indicator_Q \ast M_Q^R \Phi)
    =    \indicator_{\lambda Q} \ast M_Q^R \Phi,
  \]
  since $L_x (\indicator_Q \ast F) = \indicator_{x Q} \ast F$.
  Therefore, we see for $c = (c_\lambda)_{\lambda \in \Lambda} \in Y_d (\Lambda)$ that
  \[
    \mu_G (Q)
    \sum_{\lambda \in \Lambda}
      |c_\lambda| \, |F_\lambda|
    \leq \Big(
           \sum_{\lambda \in \Lambda}
             |c_\lambda| \, \indicator_{\lambda Q}
         \Big)
         \ast M_Q^R \Phi .
  \]
  By definition of $Y_d(\Lambda)$, we have $H_c :=  \sum_{\lambda \in \Lambda}
             |c_\lambda| \, \indicator_{\lambda Q} \in Y$.
  By the convolution relation \eqref{eq:CoorbitConvolutionRelation} and because of
  ${M_Q^R \Phi \in L_w^1(G)}$, this implies $H_c \ast M_Q^R \Phi \in Y$ as well,
  with norm estimate $\| H_c \ast M_Q^R \Phi \|_Y \lesssim \| H_c \|_Y = \| c \|_{Y_d}$.
  By the solidity of $Y$, this implies that the  map
  \[
    \synthesis_0 : \quad
    Y_d (\Lambda) \to Y, \quad
    (c_\lambda)_{\lambda \in \Lambda} \mapsto \sum_{\lambda \in \Lambda}
                                                c_\lambda \, V_g f_\lambda .
  \]
is well-defined and bounded. Moreover, we claim that
  \begin{equation}
    \Big\|
      \sum_{\lambda \in \Lambda}
        |c_\lambda| \, |V_g f_\lambda| \,
    \Big\|_{L_{1/w}^\infty}
    \lesssim \| c \|_{\ell_{1/w}^\infty (\Lambda)}
    \lesssim \| c \|_{Y_d} .
    \label{eq:SynthesisOperatorPointwiseBound}
  \end{equation}
  To see this, let $Z := L_{1/w}^\infty (G)$, noting that $Z$ also satisfies
  the assumptions of the lemma.
  Furthermore, \mbox{\cite[Lemma~3.5]{feichtinger1989banach1}} shows
  that $Z_d(\Lambda) = \ell_{1/w}^\infty (\Lambda)$, with equivalent norms.
  Thus, we can apply the above reasoning with $Z$ instead of $Y$ to see
  \(
    \big\|
      \sum_{\lambda \in \Lambda}
        |c_\lambda| \, |F_\lambda|
    \big\|_{L_{1/w}^\infty}
    \lesssim \| c \|_{\ell_{1/w}^\infty}
    .
  \)

  \medskip{}

  \textbf{Step 3:} In this step, we prove the boundedness and well-definedness of $\synthesis$.
  To this end, let ${c = (c_\lambda)_{\lambda \in \Lambda} \in Y_d (\Lambda)}$ be arbitrary.
  For $f \in \TestVectors \subset \Hpi$, the orthogonality relation \eqref{eq:ortho_relations_a}
  shows because of $f_\lambda \in \Hpi$ that
  \[
    |\langle f_\lambda, f \rangle|
    = \| C_\pi g \|_{\Hpi}^{-2} \cdot |\langle V_g f_\lambda, V_g f \rangle_{L^2}|
    \lesssim \int_{G} |F_\lambda(x)| \cdot |F(x)| \, d \mu_G (x) ,
  \]
  with $F := V_g f \in L_w^1(G)$ and $F_\lambda = V_g f_\lambda$ as in Step~2.
  In combination with Equation~\eqref{eq:SynthesisOperatorPointwiseBound}, this implies
  \[
    \sum_{\lambda \in \Lambda}
      |c_\lambda| \cdot | \langle \, f_\lambda, f \rangle |
    \lesssim \| F \|_{L_w^1}
             \cdot \Big\|
                     \sum_{\lambda \in \Lambda}
                       |c_\lambda| \, |F_\lambda|
                   \Big\|_{L_{1/w}^\infty}
    \lesssim \| V_g f \|_{L_w^1} \cdot \| c \|_{Y_d}
    =        \| f \|_{\TestVectors} \cdot \| c \|_{Y_d} .
  \]
  Thus the series
  $\sum_{\lambda \in \Lambda} c_\lambda \langle f_\lambda, f \rangle$
  converges absolutely for every ${f \in \TestVectors}$, so that the antilinear funtional
  $\synthesis c = \sum_{\lambda \in \Lambda} c_\lambda \, f_\lambda \in \Reservoir_w$ is well-defined,
  with unconditional convergence in the weak-$\ast$-topology.
  Finally, we have that $V_g [\synthesis c] = \sum_{\lambda \in \Lambda} c_\lambda V_g f_\lambda = \synthesis_0 c$,
  and hence
  \({
    \| \synthesis c \|_{\Co (Y)}
    = \| V_g [\synthesis c] \|_{Y}
    = \| \synthesis_0 c \|_{Y}
    \lesssim \| c \|_{Y_d},
  }\)
  as seen in Step~2.

  \medskip{}

  \textbf{Step 4:} We prove the boundedness of $\analysis$.
  By possibly rescaling $g$, we can assume without loss of generality that $\| C_\pi g \|_{\Hpi} = 1$.

  Note that \eqref{eq:CoorbitMoleculeCondition} shows ${V_g f_\lambda \in L_w^1(G)}$,
  meaning $f_\lambda \in \TestVectors$, so that
  ${c_\lambda (f) := \langle f, f_\lambda  \rangle \in \CC}$
  is well-defined for every $\lambda \in \Lambda$ and $f \in \Co(Y) \subset \Reservoir_w$.
  Furthermore, \cite[Equation~(4.9)]{feichtinger1989banach1}
  shows that if we write $V_g^\ast : L_{1/w}^\infty (G) \to \Reservoir_w$
  for the adjoint of ${V_g : \TestVectors \to L_w^1(G)}$,
  then $V_g^\ast V_g f = f$ for all $f \in \Reservoir_w$.
  Therefore, setting $F := |V_g f| \in Y$ and using \eqref{eq:CoorbitMoleculeCondition},
   we see that
  \[
    |c_\lambda (f)|
    = |\langle V_g^\ast V_g f, f_\lambda \rangle|
    = |\langle V_g f, V_g f_\lambda \rangle|
    \leq \int_G F(x) \Phi(x^{-1} \lambda) \, d \mu_G(x)
    =    (F \ast \Phi)(\lambda) .
  \]

  Note that if $x \in \lambda Q$, then $\lambda \in x Q$, say $\lambda = x q$.
  Since $Q$ is open and $\Phi$ continuous, this implies for arbitrary $y \in G$ that
  $\Phi(y^{-1} x q) \leq M_Q \Phi (y^{-1} x)$, and therefore
  \[
    F \ast \Phi (\lambda)
    =    \int_G F(y) \, \Phi(y^{-1} x q) \, d \mu_G (y)
    \leq \int_G F(y) M_Q \Phi (y^{-1} x) \, d \mu_G (y)
    =    (F \ast M_Q \Phi) (x) ,
  \]
  provided that $x \in \lambda Q$.
  Thus,
  \(
    |c_\lambda (f)| \, \indicator_{\lambda Q}(x)
    \leq (F \ast \Phi)(\lambda) \, \indicator_{\lambda Q} (x)
    \leq \indicator_{\lambda Q}(x) \cdot (F \ast M_Q \Phi) (x) ,
  \)
  and hence
  \begin{align*}
    \big\| (c_\lambda (f))_{\lambda \in \Lambda} \big\|_{Y_d}
    & = \Big\|
          \sum_{\lambda \in \Lambda}
            |c_\lambda (f)| \, \indicator_{\lambda Q}
        \Big\|_{Y}
      \leq \Big\|
             (F \ast M_Q \Phi) \cdot \sum_{\lambda \in \Lambda} \indicator_{\lambda Q}
           \Big\|_{Y} \\
    & \leq \rel (\Lambda) \cdot \| F \ast M_Q \Phi \|_Y
      \lesssim \| F \|_Y
      =        \| f \|_{\Co (Y)} .
  \end{align*}
  where the last inequality used the convolution relation
  \eqref{eq:CoorbitConvolutionRelation} and that $M_Q \Phi \in L_w^1(G)$.
\end{proof}

The next lemma contains the technical density arguments used in the proofs of Theorems~\ref{thm:frame_coorbit}
and \ref{thm:riesz_coorbit}.

\begin{lemma}\label{lem:WeakStarContinuity}
  Let $Y \subset L_{\loc}^1(G)$ be a solid, translation-invariant Banach space
  with control weight $w$.
  Let $g \in \Hpi \setminus \{ 0 \}$ satisfy $V_g g \in \WLw (G)$,
  let $\Lambda$ be relatively separated in $G$,
  and let $(f_\lambda)_{\lambda \in \Lambda}, (h_\lambda)_{\lambda \in \Lambda} \subset \Hpi$
  both satisfy \eqref{eq:CoorbitMoleculeCondition}.
  Denote by $\analysis,\synthesis$ the operators from Lemma~\ref{lem:CoefficientReconstructionBounded}
  associated to $(f_\lambda)_{\lambda \in \Lambda}$, and by $\widetilde{\analysis}, \widetilde{\synthesis}$
  those associated to $(h_\lambda)_{\lambda \in \Lambda}$.
  Then the following hold:
  \begin{enumerate}
    \item If $\widetilde{\synthesis} (\analysis f) = f$ for all $f \in \Hpi$,
          then the same holds for all $f \in \Reservoir_w$.
          \vspace*{0.1cm}

    \item If $\widetilde{\analysis} (\synthesis c) = c$ for all finitely supported $c \in \CC^{\Lambda}$,
          then the same holds for all $c \in \ell_{1/w}^\infty (\Lambda)$, and in particular
          for all $c \in Y_d(\Lambda)$.
  \end{enumerate}
\end{lemma}

\begin{proof}
  Below, we will use that $\Co (L_w^1) = \TestVectors$ and $\Co (L_{1/w}^\infty) = \Reservoir_w$;
  see \cite[Corollary~4.4]{feichtinger1989banach1}, and that
  $\ell_w^1(\Lambda) = (L_w^1)_d (\Lambda)$ and $\ell_{1/w}^\infty (\Lambda) = (L_{1/w}^\infty)_d (\Lambda)$;
  see \cite[Lemma~3.5]{feichtinger1989banach1}.
  Furthermore, we will use the convention that the dual pairing
  $\langle \cdot, \cdot \rangle$ between $\ell_{1/w}^\infty$ and $ \ell_w^1$ is antilinear in the second component.

  \medskip{}

  \textbf{Step 1:}
  Applying Lemma~\ref{lem:CoefficientReconstructionBounded} to $Y = L_{1/w}^\infty$,
  we see that ${\synthesis : \ell_{1/w}^\infty (\Lambda) \to \Co(L_{1/w}^\infty) = \Reservoir_w}$
  is continuous, with the defining series converging unconditionally in the weak-$\ast$-topology.
  Therefore, given $f \in \TestVectors$ and $c \in \ell_{1/w}^\infty(\Lambda)$, we see that
  \begin{equation}
    \langle \synthesis c, f \rangle
    = \sum_{\lambda \in \Lambda}
        c_\lambda \, \langle f_\lambda, f \rangle
    = \sum_{\lambda \in \Lambda}
        c_\lambda \, \overline{\langle f, f_\lambda \rangle}
    = \langle c, \analysis f \rangle .
    \label{eq:SynthesisAsAdjoint}
  \end{equation}
  Here, we used that $\analysis : \TestVectors = \Co(L_w^1) \to \ell_w^1 (\Lambda) = (L_w^1)_d (\Lambda)$
  is bounded, as shown by Lemma~\ref{sub:LocalizedIntegralKernelsAppendix} with $Y = L_w^1$.
  Equation~\eqref{eq:SynthesisAsAdjoint} shows that
  ${\synthesis : \ell_{1/w}^\infty (\Lambda) \to \Reservoir_w}$ is continuous
  if we equip the domain and co-domain with the weak-$\ast$-topology.
  Clearly, the same holds for $\widetilde{\synthesis}$.

  \medskip{}

  \textbf{Step 2:}
  For $c = (c_\lambda)_{\lambda \in \Lambda} \in \ell_w^1(\Lambda)$,
  we have $c = \sum_{\lambda \in \Lambda} c_\lambda \, \delta_\lambda$, with unconditional
  convergence.
  Since Lemma~\ref{lem:CoefficientReconstructionBounded} shows that
  $\synthesis : \ell_w^1(\Lambda) \to \TestVectors$ is bounded,
  we thus have $\synthesis c = \sum_{\lambda \in \Lambda} c_\lambda \, f_\lambda$
  with unconditional convergence in $\TestVectors$.
  Thus, we see for $c \in \ell_w^1(\Lambda)$ and $f \in \Reservoir_w$ that
  \[
    \langle f, \synthesis c \rangle
    = \sum_{\lambda \in \Lambda}
        \langle f, f_\lambda \rangle \, \overline{c_\lambda}
    = \langle \analysis f, c \rangle .
  \]
  As in the preceding step, this implies that $\analysis : \Reservoir_w \to \smash{\ell_{1/w}^\infty}$
  is continuous if we equip both domain and co-domain with the weak-$\ast$-topology.
  The same holds for $\widetilde{\analysis}$.

  \medskip{}

  \textbf{Step 3:} Recall that $\TestVectors \subset \Hpi$.
  Therefore, $\Hpi \subset (\TestVectors)^{\neg} = \Reservoir_w$ separates the points.
  Hence
  $\Hpi$ is weak-$\ast$-dense in $\Reservoir_w$, e.g., see \cite[Corollary~5.108]{AliprantisBorderHitchhiker}.
  Thus, if $\widetilde{\synthesis} (\analysis f) = f$ for all $f \in \Hpi$,
  then by Steps~1 and 2 this extends to the weak-$\ast$-closure of $\Hpi$ in $\Reservoir_w$
  and hence to all of $\Reservoir_w$.

  Similarly, given $c \in \ell_{1/w}^\infty (\Lambda)$,
  if we choose finite subsets $\Lambda_n \subset \Lambda$ ($n \in \N$)
  with $\Lambda_n \subset \Lambda_{n+1}$ and $\Lambda = \bigcup_n \Lambda_n$, it is easy to see
  that $c \cdot \indicator_{\Lambda_n} \to c$ with respect to the weak-$\ast$-topology
  on $\ell_{1/w}^\infty (\Lambda)$.
  Thus, if $\widetilde{\analysis} (\synthesis c) = c$ holds for all finitely supported sequences,
  we see as before that this extends to all $c \in \ell_{1/w}^\infty (\Lambda)$.
  Finally, 
  since \eqref{eq:CoorbitWeightCondition} implies that $w(x) \gtrsim \| L_{x^{-1}} \|_{Y \to Y}$,
  we have
  \(
    \| \indicator_Q \|_Y
    = \| L_{\lambda^{-1}} \indicator_{\lambda Q} \|_Y
    \lesssim w(\lambda) \cdot \| \indicator_{\lambda Q} \|_Y,
  \)
  and hence
  $\| c \|_{Y_d} \geq |c_\lambda| \cdot \| \indicator_{\lambda Q} \|_Y \gtrsim |c_\lambda| / w(\lambda)$
  for all $\lambda \in \Lambda$.
  This shows $Y_d(\Lambda) \subset \ell_{1/w}^\infty (G)$,
  so that $\widetilde{\analysis} (\synthesis c) = c$ for all $c \in Y_d(\Lambda)$.
\end{proof}

\end{document}